\documentclass[a4paper,12pt]{amsart}

\usepackage[latin1]{inputenc}
\usepackage{amsfonts}
\usepackage{amsmath}
\usepackage{amssymb}
\usepackage{amsthm}
\usepackage[english]{babel}
\usepackage{graphicx}
\usepackage{mathdots}
\usepackage{xy}
\usepackage{hyperref, url}

\input xy
\xyoption{all}

\newtheorem{thm}{Theorem}[section]
\newtheorem{lemma}[thm]{Lemma}
\newtheorem{prop}[thm]{Proposition}
\newtheorem{cor}[thm]{Corollary}
\newtheorem{defn}[thm]{Definition}

\newtheorem{claim}[thm]{Claim}

\theoremstyle{remark}
\newtheorem{ex}[thm]{Example}
\newtheorem{rmk}[thm]{Remark}

\numberwithin{figure}{section}
\numberwithin{equation}{section}

\newcommand{\Z}{\mathbb{Z}}
\newcommand{\Q}{\mathbb{Q}}
\newcommand{\R}{\mathbb{R}}
\newcommand{\C}{\mathbb{C}}
\newcommand{\F}{\mathbb{F}}
\newcommand{\T}{\mathbb{T}}
\newcommand{\HF}{\widehat{HF}}
\newcommand{\CF}{\widehat{CF}}
\newcommand{\HFK}{\widehat{HFK}}
\newcommand{\CFK}{\widehat{CFK}}
\newcommand{\CFD}{\widehat{CFD}}
\newcommand{\CFDA}{\widehat{CFDA}}
\newcommand{\de}{\partial}
\newcommand{\ba}{\boldsymbol{\alpha}}
\newcommand{\bb}{\boldsymbol{\beta}}
\newcommand{\bc}{\boldsymbol{\gamma}}
\newcommand{\bd}{\boldsymbol{\delta}}

\newcommand{\x}{\mathbf{x}}
\newcommand{\y}{\mathbf{y}}

\newcommand{\p}{\mathbf{p}}
\newcommand{\q}{\mathbf{q}}
\newcommand{\bs}{\mathbf{s}}
\newcommand{\bfT}{\mathbf{T}}

\newcommand{\bfTheta}{\boldsymbol{\Theta}}
\renewcommand{\epsilon}{\varepsilon}
\newcommand{\spin}{{\rm Spin}^c}
\newcommand{\s}{\mathfrak{s}}
\renewcommand{\L}{\mathcal{L}}
\renewcommand{\H}{\mathcal{H}}
\newcommand{\D}{\mathcal{D}}
\newcommand{\MS}{\mathcal{M}}
\newcommand{\Lhat}{\widehat{\mathcal{L}}}
\newcommand{\stilde}{\tilde}
\newcommand{\ltilde}{\widetilde}
\newcommand{\Sym}{{\rm Sym}}
\newcommand{\sgn}{{\rm sgn}}
\newcommand{\lk}{{\rm lk}}
\renewcommand{\wr}{{\rm wr}}
\newcommand{\bbL}{\mathbb{L}}

\newcommand{\deff}{\textbf}

\title{Ozsv\'ath-Szab\'o invariants of contact surgeries} 
\author{Marco Golla}
\date{}

\begin{document}
\maketitle

\begin{abstract}
We give new tightness criteria for positive surgeries along knots in the 3-sphere, generalising results of Lisca and Stipsicz, and Sahamie. The main tools will be Honda, Kazez and Mati\'c's, Ozsv\'ath and Szab\'o's Floer-theoretic contact invariants. We compute Ozsv\'ath-Szab\'o contact invariant of positive contact surgeries along Legendrian knots in the 3-sphere in terms of the classical invariants of the knot. We also combine a Legendrian cabling construction with contact surgeries to get results about rational contact surgeries.
\end{abstract}

\section{Introduction}

Every contact manifold falls in one of two families: overtwisted or tight. Eliashberg \cite{El} classified overtwisted contact structures on 3-manifolds according to the homotopy type of the underlying plane field, showing that overtwisted structures are in some sense simple. The classification of tight contact structures, on the other hand, provides us with a much harder task, and many questions remain open: which 3-manifolds do support a tight contact structures? If so, how many up to isotopy? And how can we describe them?

Many tools have been developed to detect tightness; among them, Ozsv\'ath and Szab\'o's Floer-theoretic invariant $c$, living in the `hat' flavour of Heegaard Floer (co)homology of the underlying manifold: the set of contact structures with $c \neq 0$ sits in a chain of inclusions between the set of Stein fillable and the set of tight contact structures \cite{OScontact}.

Lisca and Stipsicz \cite{LS1, LS2, LStrans} extensively used this tool and the surgery exact sequences in Heegaard Floer homology to produce examples of tight contact structures on several manifolds, chiefly obtained by surgery on $S^3$ along a knot. The twisted version of $c$ has been used, for example, by Ghiggini and Van Horn-Morris \cite{GVH-M} to classify tight contact structures on some Brieskorn spheres. Vanishing results for these invariants have been given by Sahamie \cite{Sahamie}.

Recall that, for any knot $K\subset S^3$, Ozsv\'ath and Szab\'o defined two concordance invariants $\tau(K),\nu(K)\in\Z$, such that $\tau(K)\le\nu(K)\le\tau(K)+1$. On the other hand, to any Legendrian knot $L\subset(S^3,\xi)$, we can associate two other integers, $tb(L)$ and $r(L)$, the Thurston-Bennequin and the rotation number respectively: these two and the topological type of $L$ are collectively called the \deff{classical invariants} of $L$.

Finally, recall that there are two possible contact structures that are obtained as contact $n$-surgery on a given Legendrian $L$, and we'll denote them with $\xi^\pm_n(L)$. We can now state our main theorem:

\begin{thm}\label{mainthm}
Let $L$ be an oriented Legendrian knot in the standard contact structure $\xi_\textrm{st}$ on $S^3$, and let $K$ be the topological type of $L$.

For positive $n$, $\xi^-_n(L)$ has nonvanishing contact invariant if and only if the following hold:
\begin{itemize}
\item[(SL)] $tb(L)- r(L)=2\tau(K)-1$;
\item[(SC)] $n+tb(L)\ge 2\tau(K)$;
\item[(TN)] $\tau(K) = \nu(K)$.
\end{itemize}
Moreover, if $L'$ is another Legendrian knot with the same classical invariants (whether or not the three conditions hold), $c(\xi_n^-(L')) = c(\xi_n^-(L))$.
\end{thm}

\begin{rmk}\label{naturalityrmk}
There is an action of $M(Y):=MCG(Y\setminus B, \de B)$, the mapping class group of $Y$ with a ball removed, relative to the boundary, on $\HF(-Y)$ \cite{JOT}. The contact invariants $c(\xi)\in \HF(-Y)$ , $c(\xi')\in \HF(-Y')$ of two contact manifolds $(Y,\xi)$, $(Y',\xi')$, with $Y$ diffeomorphic to $Y'$, can only be compared using a diffeomorphism $Y\setminus B \to Y'\setminus B'$. Any two such diffeomorphisms differ by an element of $M(Y)$.
 
The equality $c(\xi_n^-(L')) = c(\xi_n^-(L))$ has to be taken as saying that there is a diffeomorphism $S^3_{tb(L)+n}(L)\to S^3_{tb(L)+n}(L')$ that takes $c(\xi_n^-(L'))$ to $c(\xi_n^-(L))$; this is equivalent to saying that $c(\xi^-_n(L))$ and $c(\xi^-_n(L'))$ lie in the same orbit of the action of $M(S^3_{tb(L)+n}(K))$ on $\HF(-S^3_{tb(L)+n}(K))$.
\end{rmk}

\begin{rmk}
As a mnemonic trick, the abbreviations \emph{SL}, \emph{SC} and \emph{TN} stand for ``self-linking'', ``surgery coefficient'' and ``tau-nu'' respectively.

The first condition can be interpreted as a transverse condition, \emph{i.e.} a condition on the self-linking number of the transverse push-off of $L$.

The second condition is a condition on the pair (Legendrian knot, surgery coefficient) $(L,n)$; it can also be read as $n+tb(L)\ge sl(L)+1$ or $n\ge 1-r(L)$.

The third condition could be absorbed in the first one if we just replaced $\tau$ by $\nu$ in \emph{(SL)}, since $\nu(K)$ is either $\tau(K)$ or $\tau(K)+1$; on the other hand, conditions as they are split into contact, topological and Floer-theoretical conditions separately; moreover, we'll realise along the proof that they really are three separate conditions rather than two.
\end{rmk}

In other words, the contact invariant  of an integral surgery along $L\subset(S^3,\xi_{\rm st})$ doesn't contain more information about $L$ than the classical invariants, and in particular can't distinguish surgeries along non-Legendrian isotopic knots that share the same classical invariants.

\begin{rmk}
As we'll see in Section \ref{cont_surg}, the `positive' contact surgery $\xi^+_n(L)$ is isotopic to $\xi^-_n(-L)$: the only condition that gets affected by orientation reversal of $L$ is \emph{(SL)}, so we get an analogous statement about $c(\xi^+_n(L))$ if we replace it with the condition $tb(L)+r(L)=2\tau(K)-1$.
\end{rmk}

\begin{ex}
Let's consider the knot $8_{20}$: it has genus $g(8_{20})=1$, but its slice genus is $g_*(8_{20})=0$ (which in turn implies also $\tau(8_{20})=\nu(8_{20})=0$). On the other hand, its maximal Thurston-Bennequin number is $\overline{tb}(8_{20}) = -2$ and its maximal self-linking number is $\overline{sl}(8_{20})=-1$. In particular, our Theorem \ref{mainthm} applies here, whereas neither the main result in \cite{LS1} or \cite{LStrans} does. We can therefore exhibit new examples of tight contact structures on the manifolds $S^3_q(8_{20})$ for all $q\ge0$ rational (see Corollary \ref{qtight} below).

The knot $m(10_{125})$ has $\tau(m(10_{125})) = -g_*(m(10_{125})) = -1$ and $\overline{sl}(m(10_{125}))=-3$: the first equality implies that $\nu(m(10_{125}))=0$. In particular, \emph{(TN)} doesn't hold for $m(10_{125})$, but it has a Legendrian representative for which \emph{(SL)} does hold. We're grateful to Lenny Ng for this example.
\end{ex}

As a byproduct of the proof of Theorem \ref{mainthm}, without much effort, we get:

\begin{cor}\label{qtight}
If $\tau(K)=1$ (respectively $\tau(K)=0$) and there's a Legendrian representative $L$ of $K$ that satisfies (SL), then for all $q>2\tau(K)-1$ (resp. $q\ge 0$) the manifold $S^3_q(K)$ supports a tight contact structure.
\end{cor}

In \cite{LStrans}, a new transverse invariant $\stilde{c}$ was also defined. Given $T$ transverse knot in $(Y,\xi)$, for sufficiently large $f$ we can define contact surgery along $T$ with framing $f$, and take the inverse limit of the contact invariants of these objects. Since we have complete control on these contact invariants for $T\subset S^3$, we can draw the following corollary:

\begin{cor}\label{ctilde}
Given $T$ in $(S^3,\xi)$ of topological type $K$, the transverse invariant $\stilde{c}(T)$ is nonzero if and only if $\xi=\xi_{\rm st}$, $sl(T)=2\tau(K)-1$ and $\tau(K)=\nu(K)$. Moreover, if $T'$ is another transverse knot of the same topological type of $T$ with $sl(T')=sl(T)$, then, up to the action of $MCG(S^3\setminus K)$, $\stilde{c}(T')=\stilde{c}(T)$.
\end{cor}

The proof of Theorem \ref{mainthm} has an algebraic flavour, with a topological input coming from a Legendrian cabling construction.

\vskip 0.2 cm

{\bf Organisation}. The paper is organised as follows. In Section \ref{SFH+gluing} we introduce some standard background in Heegaard Floer homology, sutured Floer homology, contact invariants and gluing maps. Section \ref{HFKstab} is devoted to the study of some sutured Floer homology groups and some gluing maps between them. In Section \ref{cont_surg} we prove some useful lemmas about contact surgeries and stabilisations; in Section \ref{cont_cabling} we discuss a Legendrian cabling construction and its interactions with contact surgeries. Finally, Section \ref{mainproof} contains the proof of Theorem \ref{mainthm}, its corollaries; we deferred the proof of some technical lemmas to Section \ref{technical}.

\vskip 0.2cm

{\bf Acknowledgments}. I owe much gratitude to my supervisor, Jake Rasmussen: he suggested that I could think about this problem and patiently supported me. I'd like to thank Matthew Hedden for pointing out a mistake in an earlier version, and for some insight on its solution; Lenny Ng for providing me with some numerical data and some references; and Paolo Ghiggini, Jonathan Hales, Robert Lipshitz, Paolo Lisca, Olga Plamenevskaya, Andr\'as Stipsicz for interesting conversations.

Most of this work has been done while I was visiting the Simons Center for Geometry and Physics in Stony Brook: I acknowledge their hospitality and support.

\section{Sutured Floer homology and gluing maps}\label{SFH+gluing}

\subsection{Sutured manifolds}\label{sutured}

The definition of sutured manifold is due to Gabai \cite{Ga}.

\begin{defn}
A \textbf{sutured manifold}, is a pair $(M,\Gamma)$ where $M$ is an oriented 3-manifold with nonempty boundary $\de M$, and $\Gamma$ is a family of oriented curves in $\de M$ that satifies:
\begin{itemize}\itemsep-1.5pt
 \item $\cup \Gamma$ intersects each component of $\de M$;

 \item $\cup \Gamma$ disconnects $\de M$ into $R_+$ and $R_-$, with $\pm\Gamma = \de R_\pm$ (as oriented manifolds);

 \item $\chi (R_+) = \chi (R_-)$.
\end{itemize}
\end{defn}

\begin{rmk}
The definition we gave is the definition of a balanced sutured manifold, due to Juh\'asz \cite{Ju}, and the condition $\chi(R_+) = \chi(R_-)$ is called the \emph{balancing} condition. Since this is the only kind of sutured manifolds we're dealing with, we prefer to just drop the adjective `balanced' everywhere.
\end{rmk}

\begin{ex}\label{cont_sut_ex1}
Any $M$ oriented 3-manifold with $S^2$-boundary, can be turned into a sutured manifold $(M,\{\gamma\})$ by choosing any simple closed curve $\gamma$ in $\de M$. We'll often write $M=Y(1)$, where $Y = M\cup_\de D^3$ is the ``simplest'' closed 3-manifold containing $M$.

For every integer $f$, we have a sutured manifold $S^3_{K,f}$ given by pairs $(S^3\setminus N(K), \{\gamma_f,-\gamma_f\})$, where $\gamma_f$ is an oriented curve on the boundary torus $\de N(K)$ of an open small neighbourhood $N(K)$ of $K$. The slope of $\gamma_f$ is $\lambda_S + f\cdot \mu$, and $-\gamma_f$ is a parallel push-off of $\gamma_f$, with the opposite orientation. 
Here $\lambda_S$ denotes the Seifert longitude of $K$. We'll use the shorthand $\Gamma_f$ for $\{\gamma_f,-\gamma_f\}$.
\end{ex}

\begin{ex}\label{cont_sut_ex2}
To any Legendrian knot $L\subset (Y,\xi)$ in an arbitrary 3-manifold $Y$ one can associate in a natural way a sutured manifold, that we'll denote with $Y_L$, constructed as follows: there's a standard (open) Legendrian neighbourhood $\nu(L)$ for $L$, with convex boundary. The dividing set $\Gamma_L$ on the boundary consists of two parallel oppositely oriented curves parallel to the contact framing of $L$. The manifold $Y_L$ is then defined as the pair $(Y\setminus \nu(L), \Gamma_L)$. In the case we're mainly interested in, where $Y=S^3$ and $L$ is of topological type $K$, we have $S^3_L = S^3_{K, tb(L)}$. More generally, the same identification $\{{\rm framings}\}\leftrightarrow \Z$ can be made canonical whenever $K$ is nullhomologous in a rationaly homology sphere $Y$ (that is $H_2(Y)=0$), and we then have $Y_L = Y_{K,tb(L)}$.

We'll often use $Y_L$ also to denote the contact manifold with convex boundary $(Y\setminus \nu(L),\xi|_{Y\setminus \nu(L)})$, without creating any confusion.

\end{ex}

There's a decomposition/classification theorem for sutured manifolds, completely analogous to the Heegaard decomposition/Reidemeister-Singer theorem for closed three-manifolds. Consider a compact surface $\Sigma$ with boundary, and a collection of simple cosed curves $\ba,\bb\subset \Sigma$, such that no two $\alpha$-curves intersect, and no two $\beta$-curves intersect; suppose moreover that $|\ba|=|\bb|$. We can build a balanced sutured manifold out of this data as follows: take $\Sigma\times[0,1]$, glue a 2-handle on $\Sigma\times\{0\}$ for each $\alpha$-curve, and a 2-handle on $\Sigma\times\{1\}$ for each $\beta$-curve, and let $M$ be the manifold obtained after smoothing corners; declare $\Gamma = \de\Sigma\times\{1/2\}$. The pair $(M,\Gamma)$ is a balanced sutured manifold, and $(\Sigma,\ba,\bb)$ is called a \deff{(sutured) Heegaard diagram} of $(M,\Gamma)$.

\begin{thm}[\cite{Ju}]
Every balanced sutured manifold admits a Heegaard diagram, and every two such diagrams become diffeomorphic after a finite number of isotopies of the curves, handleslides and stabilisations taking place in the interior of the Heegaard surface.
\end{thm}

There's one further description of a sutured manifold, relying on arc diagrams: an \deff{arc diagram} $\H^a$ is a quintuple $(\Sigma,\ba,\bb^a,\bb^c,D)$, where $\Sigma$ is a closed surface, $\ba$ and $\bb^c$ are sets of simple closed curves in $\Sigma$, with $\ba$ linearly independent in $H_1(\Sigma)$, $D$ is a closed disc disjoint from $(\cup \ba)\cup(\cup \bb)$ and $\bb^c$ is a set of pairwise disjoint closed arcs in $\Sigma\setminus {\rm Int}(D)$ with endpoints on $\de D$ (and elsewhere disjoint from $D$), each disjoint from every $\beta$-curve. We ask for $|\ba| = g = g(\Sigma)$, and $|\bb^c|+|\bb^a|= g$.

We build a sutured manifold out of $\H^a$ in the following way: the set of $\alpha$-curves determines the attaching circles of $g$ 2-handles on $\Sigma\times\{0\}\subset\Sigma\times[0,1]$; we attach a 3-handle (a ball) to fill up the remaining component of the lower boundary; the set $\bb^c$ of $\beta$-\emph{curves} determines the attaching circles of 2-handles on $\Sigma\times\{1\}$. We define $M$ to be the manifold obtained by smoothing corners after these handle attachments; notice that $D$ is an embedded disc in $\de M$, and $\bb^c$ is a set of embedded arcs in $\de M$. Let $R_+$ be a small neighbourhood of $D\cup\bb^c$ that retracts onto it, and $\Gamma = \de R_+$.

\begin{lemma}[\cite{Zarev}]
Every sutured manifold with connected $R_+$ admits an arc diagram.
\end{lemma}

\subsection{The Floer homology packages}

This is meant to be just a recollection of facts about the Floer homology theories we'll be working with. The standard references for the material in this subsection are \cite{OSHF, OSPA,Lip} for the Heegaard Floer part, and \cite{Ju} for the sutured Floer part.

In order to avoid sign issues, we'll work with $\F=\F_2$ coefficients.

Consider a pointed Heegaard diagram $\H=(\Sigma_g,\ba,\bb,z)$ representing a three-manifold $Y$, and form two Heegaard Floer complexes $\widehat{CF}(Y)$ and $CF^-(Y)$: the underlying modules are freely generated over $\F$ and $\F[U]$ by $g$-tuples of intersection points in $\bigcup_{i,j}(\alpha_i\cap\beta_j)$, so that there's exactly one point on each curve in $\ba\cup\bb$.

The differentials $\widehat{\de}, \de^-$ are harder to define, and count certain pseudo-holomorphic discs in a symmetric product $\Sym^g(\Sigma_g)$, or maps from Riemann surfaces with boundary in $\Sigma_g\times\R\times [0,1]$, with the appropriate boundary conditions. The homology groups $\HF(Y)=H_*(\CF(Y),\widehat{\de})$ and $HF^-(Y)=H_*(CF^-(Y),\de^-)$ so defined are called \deff{Heegaard Floer homologies} of $Y$, and are independent of the (many) choices made along the way \cite{OSHF}.

Sutured Floer homology is a variant of this construction for sutured manifolds $(M,\Gamma)$. The starting point is a sutured Heegaard diagram $\H=(\Sigma,\ba,\bb)$ for $(M,\Gamma)$. We form a complex $SFC(M,\Gamma)$ in the same way, generated over $\F$ by $d$-tuples of intersection points as above, where $d=|\ba|=|\bb|$. The differential $\de$ is defined by counting pseudo-holomorphic discs in $\Sym^d(\Sigma)$ or maps from Riemann surfaces to $\Sigma\times\R\times[0,1]$, again with the appropriate boundary conditions.

The homology $SFH(M,\Gamma)=H_*(SFC(M,\Gamma),\de)$ is called the \deff{sutured Floer homology} of $(M,\Gamma)$, and is shown to be independent of all the choices made \cite{Ju}. It naturally corresponds to a `hat' theory.

There's one more description of sutured Floer homology, due to Zarev \cite{Zarev}, coming from arc diagram representations: given a balanced sutured manifold $(M,\Gamma)$ with $R_+$ connected, we can form a Floer complex starting from an arc diagram associated to it. The underlying module is free over the $g$-tuples of intersection points between $\alpha$-curves and $\beta$-curves and arcs as above; the differential counts holomorphic discs in the symmetric product with boundary on these curves, such that the multiplicity at the regions touching the base-disc $D$ are all 0.

If $R_+$ is not connected, then $(M,\Gamma)$ is a product disc decomposition of a manifold $(M',\Gamma')$ with $R'_+$ connected. Juh\'asz showed that $SFH(M,\Gamma)=SFH(M',\Gamma')$, so we can compute $SFH(M,\Gamma)$ using an arc diagram for $(M', \Gamma')$.

\begin{prop}[\cite{Ju}]\label{HFvsSFH}
For a closed 3-manifold $Y$, $\HF (Y) = SFH(Y(1))$.

For a knot $K$ in a closed 3-manifold $Y$, $\widehat{HFK}(Y,K) = SFH(Y_{K,m})$, where $m$ is the meridian for $K$ in $Y$.
\end{prop}

One key feature of Heegaard Floer homology is a TQFT-like behaviour: given a four-dimensional cobordism $W: Y_1 \leadsto Y_2$, to each $\spin$-structure $\mathfrak{t}\in \spin(W)$ we associate a map $F_{W,\mathfrak{t}}: \HF(Y_1)\to\HF(Y_2)$; only a finite number of $\spin$-structures induce a nontrivial map \cite{OStriangles}, so it makes sense to define the total cobordism map $F_W = \sum F_{W,\mathfrak{t}}$. We'll be dealing with cobordisms induced by a single (four-dimensional) 2-handle attachment: in this case, the total cobordism map can be described explicitly as follows.

In such a cobordism, $Y_2$ is obtained from $Y_1$ as an integral surgery along a knot $K$, and in particular $Y_1$ and $Y_2$ can be represented as two Heegaard diagrams $(\Sigma,\ba,\bb)$ and $(\Sigma,\ba,\bc)$ such that the curves $\gamma_2,\dots,\gamma_g$ in $\bc$ are obtained from $\beta_2,\dots,\beta_g$ respectively by a small Hamiltonian perturbation. The two remaining curves $\beta_1$ and $\gamma_1$ represent a pair (meridian, longitude) on the boundary of a neighbourhood of $K$, and in particular they intersect exactly once. There's a canonical intersection point $\bfTheta$ in $(\Sigma,\bb,\bc)$, that corresponds to the top Maslov degree element in $\HF(\Sigma,\bb,\bc) = \HF(\#^{g-1}(S^1\times S^2))$.

The map $CF(Y_1)\to CF(Y_2)$ associated to this handle attachment counts pseudo-holomorphic triangles in $\Sym^g(\Sigma)$ or maps from Riemann surfaces to $\Sigma\times \Delta$ ($\Delta$ being a standard triangle), again, with appropriate boundary conditions, and involving the point $\bfTheta$.

One can also compute every single $F_{W,\mathfrak{t}}$: the domain associated to each holomorphic triangle has a well-defined $\spin$-structure, and we restrict our sum to the triangles whose structure is $\mathfrak{t}$.

Arguably, one of the most useful features of Heegaard Floer homology is the \deff{surgery exact triangle}:

\begin{thm}[\cite{OSPA}]
Given a knot $K$ in a three-manifold $Y$, and three slopes $f,g,h\in H_1(\de N(K))$ such that $f\cdot g = g\cdot h = h\cdot f = 1$, there are three maps induced by appropriate integral surgeries, such that the triangle,
\[
\xymatrix{
\HF(Y_f(K))\ar[rr] & & \HF(Y_g(K))\ar[dl]\\
& \HF(Y_h(K))\ar[ul]
}
\]
is exact. In particular, this holds when $K$ is nullhomologous, $f=\infty$, $g$ is integral and $h=g+1$.
\end{thm}

\subsection{Floer-theoretic contact invariants}

The first contact invariant to be defined in Heegaard Floer homology was Ozsv\'ath and Szab\'o's $c$ \cite{OScontact}. The definition that we give here was given by Honda, Kazez and Mati\'c \cite{HKM1}, and lead to the fruitful extension to invariants for manifolds with convex boundary, called $EH$, living in sutured Floer homology.

Since the latter is a strict generalisation of the former, we just give the definition of $EH$: if $\xi$ is a contact structure on $Y$, $c(\xi)$ is equivalent to $EH(\xi')$, where $\xi'$ is the restriction of $\xi$ to $Y\setminus B$, and $B$ is a small Darboux ball.

\begin{defn}
A \deff{partial open book} is a triple $(S,P,h)$ where $S$ is a compact open surface, $P$ is a proper subsurface of $S$ which is a union of 1-handles attached to $S\setminus P$ and $h:P\to S$ is an embedding that pointwise fixes a neighborhood of $\de P \cap \de S$.
\end{defn}

We can build a contact manifold with convex boundary out of these data in a fashion similar to the usual open books: instead of considering a mapping torus, though, we glue two asymmetric halves, quotienting the disjoint union $S\times[0,1/2]\coprod P\times [1/2,1]$ by the relations $(x,t)\sim (x,t')$ for $x\in \de S$, $(y,1/2)\sim (y,1/2)$, $(h(y),1/2)\sim (y,1)$ for $y\in P$. The contact structure is uniquely determined if we require -- as we do -- tightness and prescribed sutures on each half $S\times[0,1/2]/\mathord{\sim}$ and $P\times[1/2,1]/\mathord{\sim}$ (see \cite{Ho} for details). Moreover, to any contact manifold with convex boundary we can associate a partial open book, unique up to Giroux stabilisations.

We can build a balanced diagram out of a partial open book. The Heegaard surface $\Sigma$ is obtained by gluing $P$ to $-S$ along the common boundary.

\begin{defn}\label{basis}
A \deff{basis} for $(S,P)$ is a set $\mathbf{a} = \{a_1,\dots,a_k\}$ of arcs properly embedded in $(P,\de P \cap \de S)$ whose homology classes generate $H_1(P,\de P \cap \de S)$.
\end{defn}

Given a basis as above, we produce a set $\mathbf{b} = \{b_1,\dots,b_k\}$ of curves using a Hamiltonian vector field on $P$: we require that under this perturbation the endpoints of $a_i$ move in the direction of $\de P$, and that each $a_i$ intersects $b_i$ in a single point $x_i$, and is disjoint from all the other $b_j$'s.

Finally define the two sets of attaching curves: $\ba = \{\alpha_i\}$ and $\bb = \{\beta_i\}$, where $\alpha_i = a_i \cup -a_i$ and $\beta_i = h(b_i) \cup -b_i$: the sutured manifold associated to $(\Sigma,\ba,\bb)$ is $(M,\Gamma)$. We call $\x(S,P,h)$ the generator $\{x_1,\dots,x_k\}$ in $SFC(\Sigma,\bb,\ba)$ supported inside $P$.

\begin{prop}[\cite{HKM1}]
The chain $\x(S,P,h)\in SFC(\Sigma,\bb,\ba)$ is a cycle, and its class in $SFH(-M,-\Gamma)$ is an invariant of the contact manifold $(M,\xi)$ defined by the partial open book $(S,P,h)$.
\end{prop}

\begin{defn}
$EH(M,\xi)$ is the class $[\x(S,P,h)] \in SFH(-M,-\Gamma)$ for some partial open book $(S,P,h)$ supporting $(M,\xi)$.
\end{defn}

The type of invariants that we're going to deal with are either invariants of (complements of) Legendrian knots or invariants coming from contact structures on closed manifolds: this allows us to consider (except for Section \ref{technical}) only sutured manifolds with sphere/torus boundary and one/two sutures, as described in Examples \ref{cont_sut_ex1} and \ref{cont_sut_ex2}.

Consider a closed contact manifold $(Y,\xi)$, and let $B\subset Y$ be a small, closed Darboux ball with convex boundary. Then consider the manifold $(Y(1),\xi(1))$ where $Y(1)$ is obtained from $Y$ by removing the interior of $B$, and $\xi(1)$ is $\xi|_{Y(1)}$. We can now state the following proposition, that will be our definition of the contact invariant $c$ in Heegaard Floer homology.

\begin{prop}[\cite{HKM1}]
The Ozsv\'ath-Szab\'o class $c(Y,\xi)$ is mapped to the Honda-Kazez-Mati\'c class $EH(Y(1),\xi(1))$ under the isomorphism of \ref{HFvsSFH}.
\end{prop}

As a corollary, all properties of $c$ are inherited by $EH$, and in particular we recall the following:

\begin{cor}
If $(Y,\xi)$ is Stein fillable (respectively overtwisted) then the contact invariant $EH(Y(1),\xi(1))$ doesn't vanish (resp. vanishes).
\end{cor}

The second type of invariants comes from Legendrian knots: let's suppose that $L\subset Y$ is a Legendrian knot with respect to a contact structure $\xi$: then the contact manifold $Y_L$ defined in Example \ref{cont_sut_ex2} has a contact invariant $EH(Y_L) \in SFH(-Y_L)$. We'll denote this invariant by $EH(L)$, considering it as an invariant of the Legendrian isotopy class of $L$ rather than of its complement.

\subsection{Gluing maps}

In their paper \cite{HKM2}, Honda, Kazez and Mati\'c define maps associated to the gluing of a contact manifold to another one along some of the boundary components, and show that these maps preserve their $EH$ invariant. Consider two sutured manifolds $(M,\Gamma) \subset (M',\Gamma')$, where $M$ is embedded in ${\rm Int}(M')$; let $\xi$ be a contact structure on $N:=M'\setminus {\rm Int}(M)$ such that $\de N$ is $\xi$-convex and has dividing curves $\Gamma \cup \Gamma'$. For simplicity, and since this will be the only case we need, we'll restrict to the case when each connected component of $N$ intersects $\de M'$ (\emph{i.e.} gluing $N$ to $M$ doesn't kill any boundary component).

\begin{thm}\label{EHmap}
The contact structure $\xi$ on $N$ induces a \deff{gluing map} $\Phi_\xi$, that is a linear map $\Phi_\xi: SFH(-M,-\Gamma)\to SFH(-M',-\Gamma')$. If $\xi_M$ is a contact structure on $M$ such that $\de M$ is $\xi_M$-convex with dividing curves $\Gamma$, then $\Phi_{\xi}(EH(M,\xi_M)) = EH(M',\xi_M\cup \xi)$.
\end{thm}

This theorem has interesting consequences, even in simple cases:

\begin{cor}\label{nonvanishing}
If $(M,\Gamma)$ embeds in a Stein fillable contact manifold $(Y,\xi)$, and $\de M$ is $\xi$-convex, divided by $\Gamma$, then $EH(M,\xi|_M)$ is not trivial.
\end{cor}

\begin{proof}
We know that $c(Y,\xi)$ doesn't vanish, and so does $EH(Y(1),\xi(1))$. Since we allowed ourselves much freedom in the choice of the ball to remove to get $Y(1)$, we can suppose that $M\subset {\rm Int}(Y(1))$. Call $N=Y(1)\setminus{\rm Int}(M)$ the closure of the complement of $M$: the map $\Phi_{\xi|_N}$ carries $EH(M,\xi|_M)$ to $EH(Y(1),\xi(1))$, and since the latter is nonzero, so is the former.
\end{proof}

\begin{rmk}
In the proof we've been using something less than being Stein fillable, but just that $c(Y,\xi)\neq 0$: this is equivalent to being Stein fillable for the 3-sphere and for lens spaces (by results of Eliashberg \cite{El2} and Honda \cite{Ho} respectively), but in general the second condition is weaker (as shown, for example, by Lisca and Stipsicz in \cite{LS1}).
\end{rmk}

There's also a naturality statement, concerning the composition of two gluing maps: suppose that we have three sutured manifolds $(M,\Gamma)\subset(M',\Gamma')\subset(M'',\Gamma'')$ as at the beginning of the section, and suppose that $\xi$ and $\xi'$ are contact structures on $M'\setminus{\rm Int}(M)$ and $M''\setminus{\rm Int}(M')$ respectively, that induce sutures $\Gamma$, $\Gamma'$ and $\Gamma''$ on $\de M$, $\de M'$ and $\de M''$ respectively.

\begin{thm}
If $\xi$ and $\xi'$ are as above, then $\Phi_{\xi\cup\xi'} = \Phi_{\xi'}\circ\Phi_{\xi}$.
\end{thm}

Much of our interest will be devoted to stabilisations of Legendrian knots and associated maps, whose discussion will occupy Subsection \ref{substabmaps}: we give a brief summary of the contact side of their story here.

Let's start with a definition, due to Honda \cite{Ho}:
\begin{defn}\label{basic_slices}
Let $\eta$ be a tight contact structure on $T^2\times I$ with two dividing curves on each boundary component: call $\gamma_i$, $-\gamma_i$ the homology class of the two dividing curves on $T^2 \times\{i\}$, and let $s_i\in \Q\cup\{\infty\}$ be their slope. $(T^2 \times I,\eta)$ is a \deff{basic slice} if it is of the form above, and also satisfies the following three conditions:
\begin{itemize}\itemsep-1pt
\item $\{\gamma_0, \gamma_1\}$ is a basis for $H_1(T^2)$;
\item whenever $T_t$ is convex, there are exactly two dividing curves.
\item $\xi$ is \deff{minimally twisting}, \emph{i.e.} if $T_t = T\times \{t\}$ is convex, the slope of the dividing curves on $T_t$ belongs to $[s_0, s_1]$ (where we assume that if $s_0>s_1$ the interval $[s_0,s_1]$ is $[-\infty,s_1]\cup[s_0,\infty]$);
\end{itemize}
\end{defn}

Honda proved the following:

\begin{prop}[\cite{Ho}]
For every integer $t$ there exist exactly two basic slices $(T^2\times I, \xi_j)$ (for $j=1,2$) with boundary slopes $(t,1)$ and $(t-1,1)$. The sutured complement of a stabilisation $L'$ of $L$ is gotten by attaching one of the two basic slices to $Y_L$, where the trivialization of $T^2$ is given by $(0,1)=\mu$ and $(t,1)=c$, where $\mu$ and $c$ are a meridian and the contact framing for $L$ respectively.
\end{prop}

These two different layer correspond to the positive and negative stabilisation of $L$, once we've chosen an orientation for the knot; reversing the orientation swaps the labelling signs. Since we'll be considering oriented Legendrian knots, we can label the two slices with a sign.

\begin{defn}\label{stabmaps_def}
We call \deff{stabilisation maps} the gluing maps associated to the attachment of a stabilisation basic slice: these will be denoted with $\sigma_\pm$.
\end{defn}

\begin{rmk}
As it happens for the Stipsicz-V\'ertesi map \cite{SV}, also this basic slice attachment corresponds to a single bypass attachment.
\end{rmk}

We also collect here the definition of some gluing maps that we'll be considering later. The letter $\psi$ will be used to denote gluing maps associated to contact surgeries:

\begin{defn}\label{psiinfty_def}
Let $L\subset \nu(L)\subset (Y,\xi)$ be a Legendrian knot, and $\nu(L)$ be its standard neighbourhood. Let $B\subset\nu(L)$ be a ball with convex boundary. The map $\psi_\infty$ is associated to the layer $(\nu(L)\setminus B,\xi|_{\nu(L)\setminus B})$. This map is a homomorphism
\[\psi_\infty: SFH(-Y_{K,n})\to SFH(-Y(1)) = \HF(-Y)\]
for every nullhomologous knot $K\subset Y$.
\end{defn}

More generally, contact $p/q$-surgery is an operation that, given an oriented Legendrian knot $L\subset (Y,\xi)$, removes the standard neighbourhood $\nu(L)$ of $L$ and replaces it with a tight solid torus $(\bfT_{p/q}, \xi_{p/q})$. When $p/q=1$ there's only one such torus, and when $p/q=n>1$ is an integer, there are two such choices, called $(\bfT^\pm_{n}, \xi^\pm_{n})$. When $q>1$ there are many choices for the contact structure on $\bfT_{p/q}$; we still have two ``preferred'' choices, that we denote with $(\bfT^\pm_{p/q}, \xi^\pm_{p/q})$. Notice that, regardless of the value of $p/q$, the manifold $\bfT_{p/q}$ is simply a solid torus $S^1\times D^2$; on the other hand, the resulting sutures do change with $p/q$. We refer the reader to Chapter \ref{cont_surg} for further details.

\begin{defn}\label{psin_def}
Let $B\subset \bfT_{p/q}$ be a closed ball with convex boundary, and define $\bfT_{p/q}(1)$ to be  $\bfT_{p/q}\setminus {\rm Int}(B)$.
\begin{itemize}\itemsep -1.5pt
\item For a positive integer $n$, we define $\psi^\pm_n$ as the gluing map associated to the layer $(\bfT^\pm_{n}(1), \xi_{n}|_{\bfT^\pm_{n}(1)})$.

\item We define $\psi_{+1} = \psi^{\pm}_{+1}$ as the gluing map associated to $(\bfT_{1}(1), \xi_{1}|_{\bfT_{1}(1)})$.

\item For a positive rational $p/q$, we define $\psi^\pm_{p/q}$ as the gluing map associated to the layer $(\bfT^\pm_{p/q}(1), \xi_{p/q}|_{\bfT^\pm_{p/q}(1)})$.
\end{itemize}

Fix a knot $K\subset Y$, together with an open tubular neighbourhood $N(K)$ and a framing $f$, that we look at as a curve in $\de N(K)$; as before, denote with $Y_{K,f}$ the sutured manifold $(Y\setminus N(K), \Gamma_f = \{f, -f\})$. The map $\psi^\pm_{p/q}$ is a homomorphism
\[\psi^\pm_{p/q}: SFH(-Y_{K,f})\to SFH(-Y_{p/q}(K,f)(1)) = \HF(-Y_{p/q}(K,f)),\]
where the notation $Y_{p/q}(K,f)$ stands for the manifold that we obtain by topological $p/q$-surgery along $K$ with respect to the framing given by a meridian for $K$ in $Y$ and the longitude $f$. If $K\subset Y$ is nullhomologous and $Y$ is a rational homology sphere, $K$ has a canonical framing, the Seifert framing $f_S$: in this case, we're going to write $Y_{p/q}(K)$ for $Y_{p/q}(K,f_S)$.
\end{defn}

\section{A few facts on $SFH(S^3_{K,n})$ and $\sigma_{\pm}$}\label{HFKstab}

Given a topological knot $K$ in $S^3$, denote with $S^3_m(K)$ the manifold obtained by (topological) $m$-surgery along $K$, and let $\ltilde{K}$ be the dual knot in $S^3_m(K)$, that is the core of the solid torus we glue back in. Notice that an orientation on $K$ induces an orientation of $\ltilde{K}$, by imposing that the intersection of the meridian $\mu_K$ of $K$ on the boundary of the knot complement has intersection number $+1$ with the meridian $\mu_{\ltilde K}$ of $\ltilde{K}$ on the same surface.

Fix a contact structure $\xi$ on $S^3$ and a Legendrian representative $L$  of $K$: we'll write $t$ for $tb(L)$. Since $t$ measures the difference between the contact and the Seifert framings of $L$, $S^3_{t}(K)_{\ltilde{K},\infty}$ and $S^3_L$ are sutured diffeomorphic: in particular, $EH(L)$ lives in $SFH(-S^3_{t}(K)_{\ltilde{K},\infty}) = \HFK(-S^3_{t}(K),\ltilde{K})$, the identification depending on the choice of an orientation for $K$ (or $\ltilde{K}$).

\subsection{Gradings and concordance invariants}

The groups $\HFK(S^3,K)$ and $\HFK(-S^3_{m}(K),\ltilde{K})$ come with a grading, that we call the \deff{Alexander grading}. A Seifert surface $F\subset S^3$ for $K$ gives a relative homology class
\[
[F,\de F]\in H_2(S^3\setminus N(K), \de N(K)) = H_2(S^3_m(K)\setminus N(\ltilde{K}), \de N(\ltilde{K})).\]

Given a generator\footnote{Notation is sloppy here: the chain complex $(\CFK(S^3,K),\de)$ depends on the choice of a suitable doubly-pointed Heegaard diagram representing $(S^3,K)$: the choice of such diagram doesn't matter, here.} $\x\in\CFK(S^3,K)$, there's an induced relative $\spin$ structure $\s(\x)$ in $\underline{\rm Spin}^c(S^3,K)$ \cite[Section 2]{OSQsurg}, and the Alexander grading of $\x$ is defined as
\[
A(\x) = \frac12\langle c_1(\s(\x)), [F,\de F]\rangle.
\]

On the other hand, given a generator $\x\in\CFK(-S^3_{m}(K),\ltilde{K})$, there's an induced relative $\spin$ structure $\s(\x)\in\underline{\rm Spin}^c(S^3_m(K),\ltilde{K})$, and we can define $A(\x)$ as
\begin{equation}\label{defAlex}
A(\x) = \frac12\langle c_1(\s(\x)), [F,\de F]\rangle-\frac{m}2.
\end{equation}

We now turn to recalling the definition of $\tau(K)$ and $\nu(K)$, due to Ozsv\'ath and Szab\'o \cite{OStau, OSQsurg}, and of a third concordance invariant, $\epsilon(K)$, defined by Hom \cite{Hom}.

Recall that the Alexander grading induces a filtration on the knot Floer chain complex $(\CFK(S^3,K), \de)$, where the differential $\de$ ignores the presence of the second basepoint, that is $H_*(\CFK(S^3,K), \de) = \HF(S^3)$. In particular, every sublevel $\CFK(S^3,K)_{A\le s}$ is preserved by $\de$, and we can take its homology.

\begin{defn}
$\tau(K)$ is the smallest integer $s$ such that the inclusion of the $s$-th filtration sublevel induces a nontrivial map \[H_*(\CFK(S^3,K)_{A\le s}, \de)\longrightarrow \HF(S^3) = \F.\]
\end{defn}

This invariant turns out to provide a powerful lower bound for the slice genus of $K$, in the sense that $|\tau(K)|\le g_*(K)$ \cite{OStau}. One of the properties it enjoys, and that we'll need, is that $\tau(\overline{K}) = -\tau(K)$ for every $K$.

The definition of $\nu$ is somewhat more involved, and it comes from the mapping cone construction \cite{OSint, OSQsurg} that Ozsv\'ath and Szab\'o used to compute the rank of the Heegaard Floer homology of integer and rational surgeries along $K$. We just recall the parts of the construction that we need to get to the definition, without giving any motivation or complete explanation of the mapping cone, following Rasmussen \cite{Jake}.

We define a new complex $(A_s, \de_s)$ for each integer $s$ as follows: the underlying module $A_s$ is just $C=\CFK(S^3,K)$. The differential $\de_s$ takes into account both the differential $\de$ and the differential $\de'$, for which the r\^ole of the basepoints is reversed (\emph{i.e.} $\de$ counts differentials whose domains pass through the basepoint $w$ but not through $z$, while $\de'$ counts differential whose domains pass through $z$ but not through $w$); in the next formula, $\de_K$ is just the ``graded'' differential, that counts only discs whose domains avoid both basepoints:
\[
\de_s \x = \left\{
\begin{array}{ll}
\de \x & {\rm if}\; A(\x)<s\\
\de\x + \de'\x + \de_K\x &{\rm if}\; A(\x)=s\\
\de' \x & {\rm if}\; A(\x)>s
\end{array}
\right.
\]
The quotient complexes $A_s/C_{>s}$ and $A_s/C_{<s}$ come with natural chain maps into $(C,\de)$ and $(C,\de')$ respectively; the composition of the projection with these chain maps gives two maps $v_s, h_s: H_*(A_s,\de_s)\to \HF(S^3) = \F$.

In analogy with the definition of $\tau$, we have the following:

\begin{defn}
$\nu(K)$ is the smallest integer $s$ such that the map $v_s$ is nontrivial.
\end{defn}

Ozsv\'ath and Szab\'o proved that $\nu$ is a concordance invariant, and that the inequalities $\tau(K)\le\nu(K)\le\tau(K)+1$ hold for all knots $K$. We remarked earlier that $\tau$ changes sign when taking the mirror of the knot: $\nu$ doesn't have this property, and the discrepancy between $\nu(K)$ and $-\nu(\overline{K})$ is measured by Hom's invariant $\epsilon$:

\begin{defn}
$\epsilon(K)$ is defined to be $(\tau(K)-\nu(K)) - (\tau(\overline{K})-\nu(\overline{K}))$.
\end{defn}

$\epsilon$ can only take values in $\{-1,0,1\}$, and manifestly changes sign when we take the mirror of the knot. Hom also proves that:

\begin{prop}[\cite{Hom}]
$\epsilon(K)$ controls the relationship between $\tau(K)$ and $\nu(K)$ as follows:
\begin{itemize}\itemsep-1.5pt
\item If $\epsilon(K) = 0$, then $\tau(K) = \nu(K) = \nu(\overline{K}) = 0$.

\item If $\epsilon(K) = 1$ then $\nu(K) = \tau(K)$.

\item If $\epsilon(K) = -1$ then $\nu(K) = \tau(K)+1$.
\end{itemize}
\end{prop}

\subsection{Modules}

We now turn our attention back to $\HFK(-S^3_t(K),\ltilde{K}) \simeq SFH(-S^3_{K,t})$. Recall that this is a graded $\F$-vector space, and that we called $A$ the grading.

The group $\CFK(S^3,K)$ is a graded vector space that comes with two differentials, $\de_K$ and $\de$, such that the complex $(\CFK(S^3,K),\de)$ has homology $\HF(S^3) = \F$, while the complex $(\CFK(S^3,K),\de_K)$ has homology $\HFK(S^3,K)$. We call $A$ the Alexander grading on this group as well.

Let's call $d = \dim\HFK(S^3,K)$, and fix a basis $\mathcal{B} = \{\eta_i,\eta'_j \mid 0\le i < d\}$ of $\CFK(S^3,K)$ such that the set $\{\eta^{top}_i, (\eta'_j)^{top}\}$ of the highest nontrivial Alexander-homogeneous components of the $\eta_i$'s and $\eta'_j$'s is still a basis for $\CFK(S^3,K)$, and the following relations hold (see \cite[Section 11.5]{LOT}):
\[
\begin{array}{lll}
\de \eta_0 = 0 &\quad & \de_K \eta_0 = 0 \\ \de \eta_{2i-1} = \eta_{2i} &\quad & \de_K \eta_i =  0 \\
\de \eta_{2j-1}'  = \eta_{2j}' &\quad & \de_K \eta'_{2j-1} = \eta'_{2j}.
\end{array}
\]
Observe that the set of homology classes of the $\eta_i$'s is a basis for $\HFK(S^3,K) = H_*(\CFK(S^3,K),\de_K)$. Finally, call $\delta(i) = A(\eta_{2i})-A(\eta_{2i-1})$. Let's remark that by definition $A(\eta_0) = \tau:=\tau(K)$.

\begin{thm}[\cite{LOT}]\label{Heddenthm}
The \emph{homology} group $\HFK(-S^3_{m}(K),\ltilde{K})$ is an $\F$-vector space with basis $\{d_{i,j}, d_{i,j}^*, u_\ell\mid 1\le i \le k, 1\le j \le \delta(i), 1\le \ell \le |2\tau-m|\}$, where the generators satisfy $A(d_{i,j}) = A(\eta_{2i})-(j-1)-(m-1)/2 = -A(d_{i,j}^*)$ and $A(u_\ell) = \tau-(\ell-1)-(m-1)/2$.
\end{thm}

Generators with a $*$ are to be thought of as symmetric to the generators without it, and each family $\{d_{i,j}\}_j$ can be interpreted as representing the arrow $\eta_{2i-1}\stackrel{\de}{\mapsto} \eta_{2i}$ (notice that $i$ varies among \emph{positive} integers), counted with a multiplicity equalling its length (\emph{i.e.} the distance it covers in Alexander grading).

\begin{rmk}
Not any basis of $\HFK(-S^3_m(K),\ltilde{K})$ with the same degree properties works for our purposes: we're actually choosing a basis that's compatible with stabilisation maps, as we're going to see in Theorem \ref{stabmaps} (see also Remark \ref{naturality}).
\end{rmk}

\begin{defn}\label{defn3.3}
Call $S_+$ the subspace of $\HFK(-S^3_{m}(K),\ltilde{K})$ generated by $\{d_{i,j}\}$, and $S_-$ the one generated by $\{d_{i,j}^*\}$: the subspace $S=S_+ \oplus S_-$ is the \deff{stable complex}, and elements of $S$ are called \deff{stable elements}. The subspace spanned by $\{u_\ell\}$ is called the \deff{unstable complex} and will be denoted with $U_m$ (although the subscript will be often dropped), so that $\HFK(-S^3_{m}(K),\ltilde{K})$ decomposes as $S_+\oplus U_m \oplus S_-$.
\end{defn}

It's worth remarking that the decomposition given in the definition above is \emph{not} canonical: the three stable subspaces $S_\pm$ and $S$ are canonically defined, but the unstable complex isn't. This issue will be addressed at the end of the next section.

There's a good and handy pictorial description when $|m|$ is sufficiently large; we'll be mostly dealing with negative values of $m$, so let's call $m'=-m\gg 0$. Consider a direct sum $\ltilde{C} = \bigoplus_{i=1}^{m'} C_i$ of $m'$ copies of $C=\CFK(S^3,K)$, and (temporarily) denote by $\x_i$ the copy of the element $\x\in C$ in $C_i$. Endow $\ltilde{C}$ with a shifted Alexander grading:
\[
\ltilde{A}(\x_i) = \left\{\begin{array}{ll}
A(\x)-(i-1)-(m-1)/2 & \text{for } i\le m'/2\\
-A(\x)-(i-1)-(m-1)/2 & \text{for } i>m'/2
\end{array}\right.
\]
for each homogeneous $\x$ in $\CFK(S^3,K)$. We picture this situation by considering each copy of $C$ as a vertical tile of $2g(K)+1$ boxes -- each corresponding to a value for the Alexander grading, possibly containing no generators at all, or more than one generator -- and stacking the $m'$ copies of $C$ in staircase fashion, with $C_1$ as the top block and $C_{m'}$ as the bottom block. Notice that, by our grading convention, the copies in the bottom part of the picture are turned upside down: for example, if $\x^{\rm max}\in C$ has maximal Alexander degree $A(\x) = g(K)$, then $\x^{\rm max}_1$ lies in the top box of $C_1$, while $\x^{\rm max}_{m'}$ lies in the bottom box of $C_{m'}$. Likewise, an element $\x^\tau\in C$ has Alexander degree $A(\x)=\tau$, then $\x^\tau_1$ lies in the $(g(K)-\tau+1)$-th box from the top in $C_1$, and $\x_{m'}^\tau$ lies in the $(g(K)-\tau+1)$-th box from the bottom in $C_{m'}$.

Our construction is slightly different from Hedden's construction, and in general it gives a different chain complex for $\HFK(S^3_m(K),\ltilde{K})$, but their homologies agree.

The situation is depicted in Figure \ref{Heddenfigure}: in this concrete example we have $g(K)=2$ and $\tau(K)=-1$; accordingly, there are $2g(K)+1=5$ boxes in each vertical column and $\x^\tau_1$ lies in the fourth box from the top in $C_1$.

\begin{figure}
\begin{center}
\includegraphics[scale=0.6]{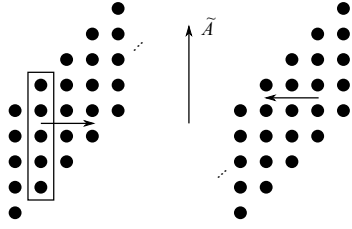}
\end{center}
\caption{We represent here the top (on the right) and bottom (on the left) parts of $\HFK(S^3_m(K),\ltilde{K})$ for  $m\ll 0$. Each vertical tile is a copy of $\CFK(S^3,K)$, and the arrows show the direction of the differentials.}\label{Heddenfigure}
\end{figure}

Now define a differential $\ltilde{\de}$ on $\ltilde{C}$ in the following way:
\[\ltilde{\de}:\left\{\begin{array}{ll}
 (\eta_0)_i \mapsto 0 & \text{for small and large } i\\
 (\eta_{2j-1})_i \mapsto (\eta_{2j})_{i+\delta(j)}\mapsto 0 & \text{for small } i\\
 (\eta_{2j-1})_i \mapsto (\eta_{2j})_{i-\delta(j)}\mapsto 0 & \text{for large } i\\
 (\eta'_{2j-1})_i \mapsto (\eta'_{2j})_i\mapsto 0 & \text{for every } i\\
  \end{array}\right..
\]
We haven't yet defined what the differential does to non-primed generators for intermediate values of $i$: in the picture we have, the differentials are horizontal and point ``inwards'' (see Figure \ref{Heddenfigure}); in particular, every horizontal (\emph{i.e.} Alexander-homogeneous) block of boxes is a subcomplex. We now extend the differential to be any map $\ltilde{\de}$ such that any horizontal block is a subcomplex, and such that the homology of every such subcomplex is $\HF(S^3) = \F$ for intermediate values of the Alexander degree\footnote{We can choose such an extension thanks to the symmetry of $\HFK(S^3,K)$: this symmetry ensures that there are always $d = \dim\HFK(S^3,K)$ generators in each Alexander degree $\ltilde{A}$; since $d$ is odd, we can choose a differential such that the homology is 1-dimensional.}.

We're now going to analyse what happens on the top and bottom part of the complex (\emph{i.e.} when $i$ is small or large, in what follows), when we take the homology.

Pairs $(\eta'_{2j-1})_i, (\eta'_{2j})_i$ cancel out in homology. The element $(\eta_{2j})_i$ is a cycle for each $i,j$, and it's a boundary only when $j>0$ and either $i>\delta(j)$ or $i< m'-\delta(j)$: so there are $2\delta(j)$ surviving copies of $\eta_{2j}$, in degrees $A(\eta_{2j})-k-(m-1)/2$ and $-A(\eta_{2j})+k+(m-1)/2$ for $k=0,\dots,\delta(j)-1$. We can declare $d_{i,j} = [(\eta_{2j})_i]$ and $d_{i,j}^* = [(\eta_{2j})_{m'-i}]$.

The element $(\eta_0)_i$ is a cycle for every $i$, and it's never canceled out, so it survives when taking homology. Given our grading convention, for \emph{small} values of $i$, $\ltilde{A}((\eta_0)_i) = A(\eta_0)-(i-1)-(m-1)/2 = \tau(K)-(i-1)-(m-1)/2$, and in particular we have a nonvanishing class $[(\eta_0)_i] = u_i$ in degrees $\tau(K)-(m-1)/2, \tau(K)-(m-1)/2-1, \dots$ On the other hand, when $i$ is \emph{large}, $[(\eta_0)_i]$ lies in degree $-\tau(K)-(i-1)-(m-1)/2$, and we get a nonvanishing class $[(\eta_0)_i] = u_{2\tau(K)+i+(m-1)/2}$ in degrees $-\tau(K)+(m-1)/2, -\tau(K)+(m-1)/2+1, \dots$

We also have a string of $\F$ summands in between, giving us a strip of unstable elements of length $2\tau(K)-m$, as in Theorem \ref{Heddenthm}.

\begin{rmk}
Something more can be said about $\spin$ structures: $\HFK(S^3_m(K),\ltilde K)$ splits as a sum of subcomplexes $\HFK(S^3_m(K),\ltilde K; \s_i)$ corresponding to the $|m|$ different $\spin$ structures on\footnote{This holds as long as $m\neq 0$: if $m=0$ there are infinitely many of them. We'll be mostly considering $m\ll 0$, so we're not really concerned with it.} $S^3_m(K)$. The Alexander grading $\ltilde{A}$ tells us when two horizontal subcomplexes fall into the same $\spin$ structure: as one could expect, if $\ltilde{A}(x) \equiv \ltilde{A}(y) \mod m$, then $x$ and $y$ belong to the same summand $\HFK(S^3_m(K),\ltilde K; \s)$.
\end{rmk}

\subsection{Stabilisation maps}\label{substabmaps}

We're going to study the action of the two stabilisation maps $\sigma_\pm$ of Definition \ref{stabmaps_def} on the sutured Floer homology groups $SFH(-S^3_L)$: it's worth stressing that these maps do not depend on the contact structure on the knot complement or on the particular Legendrian representative, but just on its Thurston-Bennequin number (which determines domain and codomain).

Notice that if $L$ is a Legendrian knot in $S^3$ with $tb(L)=n$, then, as a sutured manifold, $S^3_L$ is just $S^3_{K,n}$. Moreover, if $L'$ is a stabilisation of $L$, then $S^3_{L'}$ is isomorphic to $S^3_{K,n-1}$ as a sutured manifold.

Recall that we have two families (indexed by the integer $n$) of stabilisation maps, $\sigma_\pm: SFH(-S^3_{K,n})\to SFH(-S^3_{K,n-1})$, corresponding to the gluing of the negative and positive stabilisation layer: if the knot $K$ is oriented, these maps can be labelled as $\sigma_-$ or $\sigma_+$. With a slight abuse of notation, we're going to ignore the dependence of these maps on the framing.

\begin{rmk}\label{orientation_rmk}
Notice that orientation reversal of $L$ or $K$ isn't seen by the sutured groups nor by $EH(L)$, but it swaps the r\^oles of $\sigma_-$ and $\sigma_+$.
\end{rmk}

\begin{rmk}
Let's recall that for an \emph{oriented} Legendrian knot $L$ of topological type $K$ in $S^3$ the Bennequin inequality holds:
\[tb(L)+r(L) \le 2g(K)-1.\]
In \cite{Pl} a sharper result was proved:
\begin{equation}\label{olga}
tb(L)+r(L) \le 2\tau(K)-1.
\end{equation}

This last form of the Bennequin inequality, together with Theorem \ref{Heddenthm}, tells us that, whenever we're considering knots in the standard $S^3$, the unstable complex is never trivial in $SFH(-S^3_{K,n})$: more precisely we're always (strictly) below the threshold $2\tau:=2\tau(K)$, so that $2\tau-m$ is always positive; in particular, the dimension of the unstable complex is always positive and \emph{increases} under stabilisations. We'll state the theorem in its full generality anyway, even though this remark tells us we need just half of it when working in $(S^3,\xi_{\rm st})$.
\end{rmk}

We're going to prove the following result:

\begin{thm}\label{stabmaps}
The maps $\sigma_-, \sigma_+: SFH(-S^3_{K,n})\to SFH(-S^3_{K,n-1})$ act as follows:
\[\begin{array}{lcll}\sigma_-:\left\{\begin{array}{l}
d_{i,j}\mapsto d_{i,j}\\
u_\ell\mapsto u_\ell\\
d^*_{i,j}\mapsto d^*_{i,j+1}
  \end{array}\right. , & &
\sigma_+:\left\{\begin{array}{l}
d_{i,j}\mapsto d_{i,j+1}\\
u_\ell\mapsto u_{\ell+1}\\
d^*_{i,j}\mapsto d^*_{i,j}
  \end{array}\right. & {\rm for }\; n\le 2\tau;\\
  \\
\sigma_-:\left\{\begin{array}{l}
d_{i,j}\mapsto d_{i,j}\\
u_\ell\mapsto u_\ell\\
u_{n-2\tau}\mapsto 0 \\
d^*_{i,j}\mapsto d^*_{i,j+1}
  \end{array}\right. , & &
\sigma_+:\left\{\begin{array}{l}
d_{i,j}\mapsto d_{i,j+1}\\
u_\ell\mapsto u_{\ell-1}\\
u_1\mapsto 0\\
d^*_{i,j}\mapsto d^*_{i,j}
  \end{array}\right. & {\rm for }\; n> 2\tau.

\end{array}
\]
\end{thm}

Notice that we're implicitly choosing an appropriate isomorphism between the group $SFH(-S^3_{K,n})$ and the vector space generated by the $d_{i,j}$'s and the $u_i$'s (see Theorem \ref{Heddenthm} and Remark \ref{naturality}).

There's an interpretation of the maps $\sigma_{\pm}: SFH(-S^3_{K,n})\to SFH(-S^3_{K,n-1})$ in terms of Figure \ref{Heddenfigure}, when $n\ll 0$: fix a chain complex $C$ computing $\HFK(S^3,K)$ and call $(\ltilde{C}_n, \ltilde\de)$ and $(\ltilde{C}_{n-1}, \ltilde\de)$ the two complexes defined in the previous section, computing $SFH(-S^3_{K,n})$ and $SFH(-S^3_{K,n-1})$ starting from $C$. We have two ``obvious'' chain maps $s_\pm: \ltilde{C}_n\to \ltilde{C}_{n-1}$: $s_-$ sends $\x_i\in \ltilde{C}_n$ to $\x_i\in \ltilde{C}_{n-1}$, while $s_+$ sends $\x_i\in \ltilde{C}_n$ to $\x_{i+1}\in \ltilde{C}_{n-1}$. The maps $s_\pm$ induce the two stabilisation maps $\sigma_\pm$ at the homology level.

$s_-$ is the inclusion $\ltilde{C}_{n}\hookrightarrow \ltilde{C}_{n-1}$ that misses the leftmost vertical tile (that is, the copy $C_{1-n}$ of $C$ that's in lowest Alexander degree), while $s_+$ is the inclusion that misses the rightmost vertical tile (the copy $C_1$ of $C$ that lies in highest Alexander degree).

As a corollary (of the proof), we obtain a graded version of the result:

\begin{cor}\label{sigma-grading}
The maps $\sigma_\pm$ are Alexander-homogeneous of degree $\mp 1/2$.
\end{cor}

Notice that the maps $\sigma_-$ preserve $S_+$ and eventually kill $S_-$, whereas the maps $\sigma_+$ have the opposite behaviour. Moreover, $\sigma_-$ and $\sigma_+$ are injective on the unstable complex for $n\le2\tau$, while they eventually kill it for $n>2\tau$.

Namely, for $n \le 2\tau$, the subcomplex $S_\pm = \bigcup_{m>0}\ker \sigma_\pm^m = \ker \sigma_{\pm}^N$ for some large $N$ (depending on $K$, but not on the slope $n$: any $N>2g(K)$ works), are canonical. For $n<2\tau$, the unstable subspace is \emph{not} canonical, being a section for the projection map $SFH(-S^3_{K,n})\to SFH(-S^3_{K,n})/(S_++S_-)$.

On the other hand, for $m>2\tau$ the situation is reversed: the unstable complex is defined as the intersection of the kernels of $\sigma_\pm^N$, and $S_\pm$ is defined as a section of the projection map $\ker \sigma_{\mp}^N \to (\ker \sigma_{\mp}^N)/(\ker \sigma_-^N\cap\ker \sigma_+^N)$.

The action of $\sigma_\pm$ on the unstable complex is just by degree shift, as in Theorem \ref{stabmaps} (see also figure \ref{unst_cx} below).

\begin{rmk}\label{naturality}
Let's first consider the case $n\le 2\tau(K)$: as noticed above, we can write $S_\pm = \bigcup_n \ker \sigma_\pm^n$, and we immediately obtain that $S_\pm$ is independent of the isomorphism of Theorems \ref{Heddenthm} and \ref{stabmaps}. In fact, the subspaces $S_\pm\subset SFH(-S^3_{K,n})$ are determined by the action of $\sigma_-$, which is independent of the isomorphism of Theorem \ref{Heddenthm}.

The situation for the unstable complex, on the other hand, is completely different: we have a naturally defined \deff{unstable quotient} (see Remark \ref{unstable_quotient} below) $SFH(-S^3_{K,n})/S$. The unstable complex as we defined it is a section of the quotient map $SFH(-S^3_{K,n})\to SFH(-S^3_{K,n})/S$ that satisfies some additional requirements: namely, we need to choose \emph{any} homogeneous section of the quotient map for $n=2\tau(K)-1$, and then we take the subcomplexes generated by compositions of $\sigma_+$ and $\sigma_-$ to generate the unstable complexes in $SFH(-S^3_{K,n})$ for smaller values of $n$. We give below an example to show that there are actually instances where the choice of the section matters.

Finally, let's consider the case $n>2\tau(K)$. Here the situation is reversed: the unstable complex is the intersection $U=\bigcup_n \ker\sigma^n_- \cap \bigcup_n \ker\sigma_+^n$, and is therefore well-defined and independent of the isomorphism. The two unstable complexes, on the other hand, depend on the choice of suitable sections of the quotient maps $\bigcup_n \ker\sigma_\pm^n \to U$.
\end{rmk}

\begin{ex}
We can give a concrete example to show that the choice of the unstable complex is not unique. There's a recent result of Baldwin, Vela--Vick and V\'ertesi \cite{BVV} that relates the combinatorial Legendrian invariants $\widehat{\lambda}_{\pm}(L), \lambda^-_\pm(L)$ of \cite{OST} and the invariants $\Lhat(\pm L), \L^-(L)$ of \cite{LOSS}: there are two Legendrian representatives $L_1, L_2$ of the pretzel knot $K=P(-4,-3,3)=m(10_{140})$ in $(S^3, \xi_{\rm st})$ that have $tb(L_i)=-1$, $r(L_i)=0$, but $\Lhat(L_1) = 0 \neq \Lhat(L_2)$ (the example is found in \cite{NOT}, where they're distinguished by the combinatorial invariants); notice that $\tau(K)=0$ \cite{knotinfo}, therefore the unstable complex in $SFH(-S^3_{L_i})$ has length $|tb(L_i)-2\tau(K)|=1$, and $EH(L_i)$ has the same degree as the degree of the only nonzero element of the unstable complex (see Proposition \ref{EHgrading} below).

Since the mapping class group of $S^3\setminus N(K)$ relative to the boundary is trivial \cite{KS}, the fact that $\Lhat$ distinguishes these two knots for some parametrisation implies that it distinguishes them for all parametrisations (see the discussion preceding Lemma \ref{EHnaturalitylemma} below).

Neither $EH(L_1)$ nor $EH(L_2)$ gets killed by $\sigma^n_+\circ\sigma^n_-$, since the trivial filling (\emph{i.e.} contact $\infty$-surgery) yields back the standard contact structure on $S^3$ (compare with Proposition \ref{stablevspsiinfty} below). In particular, we can define the generator of the unstable complex to be either of $EH(L_1)$ or $EH(L_2)$, and these two element are distinct by \cite{SV}. Here we're using the fact that $tb(L_i) = 2\tau(K)-1$: if this wasn't the case, and $tb(L_i)<2\tau(K)-1$, we'd need to check that an element that doesn't vanish under $\sigma^n_+\circ\sigma^n_-$ is in fact the stabilisation of at least one element in $SFH(-S^3_{K,tb(L_i)+1})$, because of Theorem \ref{stabmaps}.
\end{ex}

\subsection{The proof of Theorem \ref{stabmaps}}

In this section we're going to give a proof of Theorem \ref{stabmaps}: the main point is the interaction of stabilisation maps with bordered Floer homology. The reader is referred to \cite{LOT} (especially Chapter 11 and Appendix A) for definitions and properties.

Let $\H=(\Sigma\setminus\{p\}, \{\beta_1^a, \beta_2^a\}, \bb^c, \ba)$ be a bordered diagram for $S^3\setminus N(K)$ such that the closures of $\beta_1^a, \beta_2^a$ represent the curves $\lambda-(n+1)\mu$ and $\lambda-n\mu$ respectively, where $\lambda\subset \de N(K)$ is the Seifert longitude for $K$, and $\mu\subset \de N(K)$ is the meridian.

Let's call $W = \CFD(\H)$ and $V = \CFD(\H')$, where $\H'$ is the bordered diagram $(\Sigma,\{\mu,\beta^a_2\},\bb^c,\ba)$; as in \cite{LOT}, we'll use the notation $V^j$, $W^j$ to denote the submodules $\iota_jV$, $\iota_jW$ of $V$ and $W$ in the idempotent $\iota_j$, for $j=0,1$. When talking about coefficient maps, we'll use the superscripts $V$, $W$ to distinguish between the maps acting on $V$ and the ones acting on $W$.

Finally, recall that we have four isomorphisms:
\[
\begin{array}{lcl}
H_*(V^0,D^V) \simeq \HFK(S^3,K) & \quad & H_*(V^1,D^V) \simeq SFH(S^3_{K,-n})\\
H_*(W^0,D^W)\simeq SFH(S^3_{K,-n-1}) & \quad & H_*(W^1,D^W)\simeq SFH(S^3_{K,-n}).
\end{array}
\]
$V^0$ is just $\CFK(S^3,K)$, and we can write down explicitly a chain homotopy equivalence between the model for $V^1$ found in \cite{LOT} and the model $\ltilde{C}$ for Theorem \ref{Heddenthm}.

Since $\ltilde{C}$ computes the sutured Floer group $SFH(-S^3_{K,n})$, which is in turn the cohomology group $SFH(S^3_{K,n})$, we expect $V^1$ to be chain homotopic equivalent to the \emph{dual} complex of $\ltilde{C}$: in fact, $V^1$ \emph{is} the dual to $\ltilde{C}$ at the top and at the bottom, and is a sum of copies of $\F = \HF(S^3)$ for intermediate values of the Alexander grading. The differential $D^V$ is the knot Floer differential $\de_K$ on $V^0$ and it's the adjoint of $\ltilde{\de}$ on $V^1$. The maps $D_1^V$ and $D_3^V$ are adjoint to the projections $\ltilde{C}\to C_n\simeq \CFK(S^3,K)$ and $\ltilde{C}\to C_1\simeq \CFK(S^3,K)$ respectively. The map $D_2^V$ is adjoint to the inclusion $C\simeq C_1\hookrightarrow \ltilde{C}$, and the map $D_{23}^V$ is adjoint to the shift map $\x_i\mapsto\x_{i-1}$.

\begin{prop}[\cite{Zarev2}]
The \emph{adjoints} of the coefficient maps $D^W_1$ and $D^W_3$ induce the two stabilisation maps
\[
SFH(-S^3_{K,-n})\simeq H^*(W^1,D^W)  \stackrel{\sigma_{\pm}}{\longrightarrow}H^*(W^0,D^W) \simeq SFH(-S^3_{K,-n-1}).
\]
\end{prop}

We're now able to prove Theorem \ref{stabmaps}, which is just a computation in light of the previous proposition.

\begin{proof}
It's shown in \cite[Appendix A]{LOT} that $W = \CFDA(\tau_\lambda)\boxtimes V$. $\CFDA(\tau_\lambda)$ is an $\mathcal{A}(T^2)$-bimodule, which is generated over $\F$ by three vectors $\p,\q,\bs$. The action of the idempotents on $\CFDA(\tau_\lambda)$ is so that, as vector spaces:
\[
\begin{array}{lll}
W^0 = \p\boxtimes V^0 \oplus \bs\boxtimes V^1 & \quad &  W^1 = \q\boxtimes V^1.
\end{array}
\]
We want to compute the action of the coefficient maps $D^W$, $D^W_1$, $D^W_3$ on $W$. We refer to the computations of \cite[Chapter A.3.1]{LOT}: there we find that
\[
\begin{array}{lll}
m_{0,1,1}(\p,\rho_3) = \rho_3\otimes \q & \quad & m_{0,1,0}(\bs) = \rho_1\otimes\q\\
m_{0,1,1}(\bs,\rho_2) = \p & \quad & m_{0,1,1}(\bs,\rho_{23}) = \rho_3\otimes\q.
\end{array}
\]
In particular, for all $\x\in V^0$, $\y\in V^1$ we have:
\begin{equation}\label{DW13}
\begin{array}{l}
D^W: \p\boxtimes\x + \bs\boxtimes\y \mapsto \p\boxtimes(D^V\x + D^V_2\y) + \bs\boxtimes D^V\y;\\
D^W: \q\boxtimes \y \mapsto \q\boxtimes D\y \\
D_1^W: \p\boxtimes\x + \bs\boxtimes\y \mapsto \q\boxtimes\y;\\
D_3^W: \p\boxtimes\x + \bs\boxtimes\y \mapsto \q\boxtimes (D_3^V\x + D_{23}^V\y),
\end{array}
\end{equation}
The model for $SFH(-S^3_{K,-n-1})$ given by the dual of $W^0$ agrees with the model of Theorem \ref{Heddenthm}, under the linear isomorphism that identifies the subspace $\p\boxtimes V^0$ with the dual of $C_1$, sitting as the leftmost column in Figure \ref{Heddenfigure}, and the subspace $\bs\boxtimes V^1$ with the dual to $\bigoplus_{k\ge 2} C_k$, consisting of the $n$ rightmost columns.

The adjoint of $D^W$ acts on the dual of $W^0$ so that the dual of $\bs\boxtimes V^1$ is a subcomplex. By equations \ref{DW13}, $D_1^W(\p\boxtimes\x + \bs\boxtimes\y) = \q\boxtimes\x$: in other words, $\p\boxtimes V^0 \subset \ker D_1^W$ and $D_1^W$ is the isomorphism $\bs\boxtimes V^1$ onto $W^1\simeq V^1$ that is the identity on the second factor. In particular, the adjoint of $D_1^W$ is the inclusion of the dual of $W^1$ into the dual of $W^0$ as the subcomplex dual to $\bs\boxtimes V^1$, that is the subcomplex generated by the $C_k$ with $k\ge 2$.

Similarly, the adjoint of $D_3^W$ is seen to act as the inclusion of the dual of $W^1$ into the dual of $W^1$ as the subcomplex generated by the $C_k$ with $k\le n$.

We've given a concrete identification of $H_*(\ltilde{C})$ with the model of Definition \ref{defn3.3}, where the class of $(\eta_{2j})_i$ is identified with $d_{i,j}$ for $i\le \delta(j)$ and with $d_{m'-i,j}$ for $i\ge m'-\delta(j)$, and $(\eta_0)_i$ is identified with $u_i$. The adjoint of $D_1^W$ is just the inclusion map $\x_i\mapsto \x_i$, whereas the adjoint of $D_3^W$ is the inclusion $\x_i\mapsto\x_{i+1}$ for all $i$'s: in particular, the induced maps act on homology as claimed.

The result for arbitrary framing parameter $n$ follows from \cite[Theorem A.11]{LOT}: they prove that $V^1=\iota_1CFD(S^3\setminus K)$ decomposes of the sum of a stable complex (containing the dual to $S_+\oplus S_-$) and an unstable chain (containing the dual to $U$) as follows. We can pick bases $\{\xi_i\}, \{\eta_i\}$ for $\CFK(S^3,K)$ playing the r\^oles of the basis $\mathcal{B}$ used in Theorem \ref{Heddenthm}\footnote{We're going to forget about primed elements, as they don't play any r\^ole in homology.}, one with respect to the basepoint $z$ and the other with respect to the basepoint $w$. We also introduce strings of elements $\{\kappa^k_i\}, \{\lambda^k_i\}$ of length $\delta(i)$ associated to each arrow $\xi_i\stackrel{\de_z}{\longrightarrow}\xi_{i+1}$ and $\eta_i\stackrel{\de_w}{\longrightarrow}\eta_{i+1}$ respectively, both of length $\delta(i)$.

The stable complex in $V^1$ looks like:
\[
\begin{array}{l}
\xi_i\stackrel{D_1^V}{\longrightarrow}\kappa_1^i\stackrel{D_{23}^V}{\longleftarrow}\kappa_2^i\stackrel{D_{23}^V}{\longleftarrow}\dots\stackrel{D_{23}^V}{\longleftarrow}\kappa_{\delta(i)}^i\stackrel{D_{123}^V}{\longleftarrow}\xi_{i+1};\\
\eta_i\stackrel{D_3^V}{\longrightarrow}\lambda^i_{1}\stackrel{D_{23}^V}{\longrightarrow}\lambda^i_{2}\stackrel{D_{23}^V}{\longrightarrow}\dots\stackrel{D_{23}^V}{\longrightarrow}\lambda^i_{\delta(i)}\stackrel{D_{2}^V}{\longrightarrow}\eta_{i+1},
\end{array}
\]
and it's immediate to find an identification of the $\iota_1$ part of the stable complex with the dual of $S_+\oplus S_-$ in Theorem \ref{Heddenthm}, $d_{i,j}$ with the dual of $\kappa^{2i-1}_j$ and $d_{i,j}^*$ with the dual of $\lambda_j^{2i-1}$.

The unstable chain, on the other hand, depends on the framing as follows:
\[
\begin{array}{cll}
\xi_0\stackrel{D_1^V}{\longrightarrow}\mu_1\stackrel{D_{23}^V}{\longleftarrow}\mu_2\stackrel{D_{23}^V}{\longleftarrow}\dots\stackrel{D_{23}^V}{\longleftarrow}\mu_{2\tau(K)-n}\stackrel{D_{3}^V}{\longleftarrow}\eta_0 & & \text{for } n<2\tau(K);\\
\xi_0\stackrel{D_{12}^V}{\longrightarrow}\eta_0 & & \text{for } n=2\tau(K)\\
\xi_0\stackrel{D_{123}^V}{\longrightarrow}\mu_1\stackrel{D_{23}^V}{\longrightarrow}\mu_2\stackrel{D_{23}^V}{\longrightarrow}\dots\stackrel{D_{23}^V}{\longrightarrow}\mu_{n-2\tau(K)}\stackrel{D_{2}^V}{\longrightarrow}\eta_0 & & \text{for } n>2\tau(K),
\end{array}
\]
and we can identify $u_k$ in the unstable complex of $SFH(-S^3_{K,n})$ with the dual of $\mu_k$.

Let's call $W:=\CFDA(\tau_\lambda)\boxtimes V$: then, as above
\[
\begin{array}{lll}
W^0 = \p\boxtimes V^0 \oplus \bs\boxtimes V^1 & \quad &  W^1 = \q\boxtimes V^1,
\end{array}
\]
and the action of the maps $D^W, D^W_1$ and $D^W_3$ is controlled by equations \ref{DW13}. We have an obvious identification of the dual of $W^1$ with $SFH(-S^3_{K,n})$ that respects the stable-unstable decomposition.

The stable complex $S_+\subset SFH(-S^3_{K,n-1})$ is identified with the cohomology of the subcomplex of $W^0$ spanned by $\bs\boxtimes\kappa^j_i$ via the map $d_{i,j}\leftrightarrow (\bs\boxtimes\kappa_j^{2i-1})^*$. $S_-\subset SFH(-S^3_{K,n-1})$ is identified with the cohomology of the subcomplex of $W^0$ spanned by $\bs\boxtimes\lambda_i^j$ and $\p\boxtimes\eta_i$ via the map $d^*_{i,1}\leftrightarrow (\p\boxtimes\eta_{2i-1})^*, d^*_{i,j}\leftrightarrow (\p\boxtimes\lambda^{2i-1}_{j-1})^*$. Notice that equations \ref{DW13} imply that for every odd $i$ there's an arrow $\bs\boxtimes\lambda^i_{\delta(i)}\stackrel{D^W}{\longrightarrow}\p\boxtimes\eta_{i+1}$, so that the homology of the stable complex in $V$ has constant rank. 

The unstable complex of $SFH(-S^3_{K,n})$ is identified with the cohomology of the subspace of $W^0$ spanned by $\bs\boxtimes\mu_\ell$ and $\p\boxtimes\eta_0$ via one of these two maps: if $n\le 2\tau(K)$, we identify $u_k$ with the dual of $\bs\boxtimes\mu_k$ and $u_{2\tau(K)-n+1}$ with the dual of $\p\boxtimes\eta_0$; if $n>2\tau(K)$, we just identify $u_k$ with the dual of $\bs\boxtimes\mu_k$. Notice that in the latter case there's an arrow $\bs\boxtimes\mu_{n-2\tau(K)}\stackrel{D^W}{\longrightarrow}\p\boxtimes\eta_0$ that cancels out the two generators involved, in cohomology, so that both maps are isomorphisms.

Now, the maps $D_1^W, D_3^W$ act on the stable complex as follows:
\[
\xymatrix{
& \bs\boxtimes\kappa^i_1\ar[dl]_{D^W_3}\ar[d]_{D^W_1} & \bs\boxtimes\kappa^i_2\ar[dl]_{D^W_3}\ar[d]_{D^W_1} & \dots & & \bs\boxtimes\kappa^i_{\delta(i)}\ar[dl]_{D^W_3}\ar[d]_{D^W_1}\\
0 & \q\boxtimes\kappa^i_1 & \q\boxtimes\kappa^i_2 & \dots  &\q\boxtimes\kappa^i_{\delta(i)-1} & \q\boxtimes\kappa^i_{\delta(i)};
}
\]
\[
\xymatrix{
\p\boxtimes\eta_i\ar[dr]^{D^W_3} & \bs\boxtimes\lambda^i_1\ar[d]^{D^W_1}\ar[dr]^{D^W_3} & \dots & & \bs\boxtimes\lambda^i_{\delta(i)}\ar[d]^{D^W_1}\ar[r]^{D^W} & \p\boxtimes\eta_{i+1}\\
& \q\boxtimes\lambda^i_1 & \q\boxtimes\lambda^i_2 & \dots  &\q\boxtimes\kappa^i_{\delta(i)}.
}
\]
Finally, the action of the maps on the unstable complex depends on the framing: if $n=2\tau(K)$ or $n=2\tau(K)+1$ there's nothing to prove, since either the unstable complex in the domain or the unstable complex in the range of $D^W_1, D^W_3$ is trivial for these framing.

Let $m=|2\tau(K)-n|$. If $n<2\tau(K)$, the action is as follows:
\[
\xymatrix{
& \bs\boxtimes\mu_1\ar[dl]_{D^W_3}\ar[d]_{D^W_1} & \bs\boxtimes\mu_2\ar[dl]_{D^W_3}\ar[d]_{D^W_1} & \dots & \bs\boxtimes\mu_{m}\ar[d]_{D^W_1} & \p\boxtimes\eta_0\ar[dl]_{D^W_3}\\
0 & \q\boxtimes\mu_1 & \q\boxtimes\mu_2 & \dots  &\q\boxtimes\mu_{m};
}
\]
if $n>2\tau(K)+1$, the action is:
\[
\xymatrix{
\bs\boxtimes\mu_1\ar[d]^{D^W_1}\ar[dr]^{D^W_3} & \bs\boxtimes\mu_2\ar[d]^{D^W_1} & \dots & \bs\boxtimes\mu_{m-1}\ar[d]^{D^W_1}\ar[dr]^{D^W_3} & \bs\boxtimes\mu_{m}\ar[d]^{D^W_1}\ar[r]^{D^W} & \p\boxtimes\eta_0\\
\q\boxtimes\mu_1 & \q\boxtimes\mu_2 & \dots  &\q\boxtimes\mu_{m-1}&\q\boxtimes\mu_{m}.
}
\]
Using the identification discussed above, Theorem \ref{Heddenthm} follows.
\end{proof}

The proof of Corollary \ref{sigma-grading} is straightforward.

\begin{proof}
According to the computations above, $D_1^W$ shifts the degree by $1/2$ and $D_3^W$ shifts the degree by $-1/2$. Their adjoints, $\sigma_-$ and $\sigma_+$ shift the degrees by $-1/2$ and $+1/2$ respectively.
\end{proof}

\subsection{Sutured Legendrian invariants}

Let $L$ be a Legendrian knot in $(S^3,\xi)$, of topological type $K$. Recall that in \cite{me} the following two facts have been proved:

\begin{prop}\label{stablevsOT}
$\xi$ is overtwisted if and only if $\sigma_+^N(\sigma^N_-(EH(L)))=0$ for sufficiently large $N$, that is if and only if $x$ is stable.
\end{prop}

\begin{proof}[Sketch of proof]
The `if' direction is easy, since the union of a stabilisation basic slice and the $\infty$-surgery layer is still an $\infty$-surgery layer.

The `only if' direction follows from the remark that, in the relevant Alexander grading component, there is only one non-vanishing element $x_0$: for sufficiently large $N$, $x_0$ is the contact element of a Legendrian representative of $K$ in the standard $S^3$. We can now argue by contradiction.
\end{proof}

The following lemma is a part of the proof of the proposition above, and turns out to be useful below.

\begin{lemma}\label{stablevspsiinfty}
A homogeneous element $x\in SFH(-S^3_{K,n})$ is stable if and only if $\psi_\infty(x)=0$.
\end{lemma}

\begin{rmk}
If $tb(L)>2\tau(K)-1$, then $L$ violates Plamenevskaya's inequality, automatically implying that $\xi$ is overtwisted. On the other hand, if $tb(L)\le 2\tau(K)-1$, the proposition can be rephrased as follows: $EH(L)$ is stable in $SFH(-S^3_L)$ if and only if $\xi$ is overtwisted.
\end{rmk}

This is basically a rephrasing of the Theorem 1.2 in \cite{LOSS} telling us that $\L^-(L)$ is mapped to $c(\xi)$ if we set $U=1$.

\begin{prop}\label{EHgrading}
Identifying $SFH(-S^3_L) = \HFK(S^3_{tb(L)}(K), \ltilde{K})$ as in Proposition \ref{HFvsSFH}, $EH(L)$ is homogenous of Alexander degree $-r(L)/2$.
\end{prop}

This is a reinterpretation of the fact that $\L^-(L)\in HFK^-(-S^3,K)$ has Alexander degree $2A(\L^-(L)) = tb(L)-r(L)+1$.

\section{Contact surgeries}\label{cont_surg}

Suppose now that $L$ is a Legendrian knot in $(Y,\xi)$ of topological type $K$: contact surgery on $L$ is an operation on $Y_L= Y\setminus {\rm Int}(\nu(L))$ that consists of gluing a solid torus $S^1\times D^2$ with a tight contact structure $\eta$ such that the boundary of the torus is $\eta$-convex.

Such a tight $\eta$ exists on $S^1\times D^2$ as long as the $\eta$-dividing curves on the boundary are not parallel to $S^1\times \{*\}$ (see \cite{Ho}).

We want the gluing to respect the dividing curves on the boundary of $Y_L$ and $S^1\times D^2$, so that we can glue $\xi$ and $\eta$ to get a contact structure on the surgered manifold $Y'$. In particular, this can be done whenever we don't fill in the meridional slope for $L$.

We have a natural basis for $\de Y_L = T^2$, given by the meridian $\mu$ for $L\subset Y$ and the dividing curve $\gamma$ that is homologous to $L$ in $\nu(L)$. The slope of the curve $\{*\}\times \de D^2$ in $\de Y_L$ is measured with respect to this natural basis, and we'll refer to \deff{contact} $p/q$-\deff{surgery} along $L$ to indicate any contact structure obtained with this process on $Y_{p\mu+q\gamma}(L)$.

\begin{rmk}
Up to isotopy, there's only one contact structure $\eta$ on $S^1\times D^2$ that gives contact $\pm1$-surgery along $L$: we denote the resulting contact structure on the surgered manifold with $\xi_{+1}(L)$, or simply $\xi_{+1}$ if $L$ is clear from the context. Similarly, there is only one $\eta$ that gives contact $1/m$-surgery, that we'll denote with $\xi_{1/m}(L)$.
\end{rmk}

\begin{rmk}
Whenever $K$ is nullhomologous in $Y$, \emph{e.g.} when $Y=S^3$, there is another natural framing for $K$, the Seifert framing: one easily checks that doing contact $p/q$-surgery on a Legendrian knot $L$ with $tb(L)=t$ produces a contact structure on $Y_{t+p/q}(K)$, where the surgery coefficient here is measured with respect to the Seifert framing, so that the difference between the contact and the topological surgery framings is just a global shift.
\end{rmk}

Let's recall here Ding and Geiges's algorithm to identify contact $p/q$-surgery along a Legendrian knot $L$ in $(Y,\xi)$ as a sequence of contact $\pm1$-surgeries, when $p/q$ is positive. Pick the minimal integer $k$ such that $q-kp$ is negative, and call $r$ the number $1+p/(kp-q)$. Now consider the negative continued fraction expansion $[a_0,\dots,a_\ell]$ of $r$: inductively, $a_0 = \lceil r \rceil$ is the smallest integer $a_0\ge r$, and $r = a_0-1/[a_1,\dots,a_\ell]$. Notice that $a_i\ge 2$ for each $i$, by construction.

Define the following link $\bbL = \bbL_+\cup \bbL_-$: $\bbL_+$ is the union of $k$ Legendrian pushoffs of $L$; $\bbL_- = L_0\cup L_1\cup\dots\cup L_\ell\subset Y\setminus \bbL_+$ is constructed as follows: $L_0$ is any $(a_0-2)$-th stabilisation of a pushoff of $L$, $L_{j+1}$ is any $(a_{j+1}-2)$-th stabilisation of a pushoff of $L_j$ for $0\le j \le \ell-1$. If we have more fractions floating around, we'll denote the link associated to $p/q$ as $\bbL(p/q)$, and the two sublinks as $\bbL^\pm(p/q)$.

Notice that $\bbL_-$ depends on the choice of the signs of the stabilisations along the way: we suppress this dependence from the notation.

\begin{thm}[\cite{DG}]
Contact $p/q$-surgery is obtained from $Y$ as contact $+1$-surgery along the link $\bbL_+$ and Legendrian surgery along the link $\bbL_-$.
\end{thm}

\begin{ex}
For $n>1$, the algorithm gives us $k=1$, and the continued fraction expansion $[3,2,\dots,2]$, where there are $n-2$ 2's. Thus there are exactly two isotopy classes of contact $+n$-surgeries, depending on the choice of a positive or negative stabilisation of $L$: we'll denote them with $\xi^\pm_n(L)$, or $\xi_n^{\pm}$, sticking to Lisca and Stipsicz's convention \cite{LStrans}.
\end{ex}

\begin{rmk}
When $n=1$, $\xi^+_n = \xi^-_n = \xi_{+1}$, so the distinction between the choice of the two signs disappears.
\end{rmk}

\begin{rmk}
Let's observe here that, since $-L^\pm = (-L)^\mp$, positive contact surgeries on $L$ are dual to contact surgeries on $-L$: doing $p/q$ surgery on $L$, for a given choice of signs and doing $p/q$ surgery on $-L$ with the opposite choice of signs gives isotopic contact structures, since the two links $\bbL^+$ and $\bbL^-$ are isotopic.

In particular, as noted in the introduction, $\xi_n^-(L)$ is isotopic to $\xi^+_n(-L)$.
\end{rmk}

We'll denote with $\xi_{p/q}(L)$ any of the contact structures constructed using the algorithm above.

We want to find an open book decomposition compatible with $\xi_{p/q}(L)$, following Ozbagci \cite{Ozbob}: fix an open book $(F,h)$ for $(Y,\xi)$ compatible with $L$, in the sense that $L$ lives in a page $F$ of the open book, and is nontrivial in $H_1(F)$. The sequence of stabilisations prescribes a sequence of stabilisations of the open book, and a sequence of curves $L, L_0,\dots, L_\ell$ in the resulting page $F'$. Call $h'$ the monodromy given by the sequence of stabilisations on $(F,h)$.

An open book for $\xi_{p/q}(L)$ is given by composing $h'$ with $k$ negative Dehn twists along $L$ and a positive Dehn twist along $L_i$ for each $i$.

\begin{prop}
Let $m\ge 1, n\ge 0$ be integers, and let $p/q$ be a rational number such that\footnote{Here we adopt the convention that $1/0 = +\infty$.} $n+1/m\le p/q < n+1/(m-1)$.

Then $\xi_{p/q}(L)$ is obtained from $\xi_{n+1/m}(L)$ through Legendrian surgeries, if the first $m$ stabilisation choices for $p/q$-surgery coincide with the choices for $n+1/m$.
\end{prop}

\begin{proof}
Let's apply the Ding-Geiges algorithm to $\alpha:=p/q$ and $\beta:=n+1/m$: if $n=0$, the number $k$ given by the algorithm is $m$ for both coefficients, and the continued fraction expansion for $p/q$ has the continued fraction expansion for $1/m$ (which is the empty expansion) as an initial segment; if $n\ge 1$, the number $k$ is 1 for both fractions, and the algorithm tells us to expand the two fractions $r_\beta = 1-(mn+1)/(m-mn-1)$ and $r_\alpha=1-p/(q-p)$.

\begin{lemma}\label{cfe}
The continued fraction expansion for $r_\alpha$ contains the expansion for $r_\beta$ as an initial segment also when $n\ge 1$.
\end{lemma}

Once we have the lemma, together with the previous considerations, we see that the two links $\bbL_+$ associated to $\xi_{\alpha}(L)$ and $\xi_{\beta}(L)$ are equal; if we choose stabilisations carefully, the link $\bbL_-(\alpha)$ associated to $p/q$ and $L$ contains the link $\bbL_-(\beta)$ associated to $m+1/n$ and $L$, so that $\xi_{\alpha}(L)$ is obtained from $\xi_{\beta}(L)$ through Legendrian surgery on $\bbL_-(\alpha)\setminus \bbL_+(\beta)$.
\end{proof}

Before proving Lemma \ref{cfe}, let's analyse what happens with contact $n+1/m$-surgery, via Ding-Geiges algorithm.

\begin{rmk}
Contact $1/m$-surgery is just a sequence of $+1$-surgeries; on the other hand, when $n\ge 1$, the link $\bbL_+$ consists of $L$ only, and $\bbL_-$ is non-empty: let's now distinguish between $n=1$ and $n\ge 2$.

The fraction to expand, when $n=1$, is just $1+(m+1)$, so the expansion is $[m+2]$: in this case, $\bbL_-$ consists of an $m$-th stabilisation of a pushoff of $L$.

For larger values of $n$, the fraction to expand is $1+(mn+1)/(m-mn-1)$: by induction on $n$, its continued fraction expansion is $[3,2,\dots,2,m+1]$, where the sequence of 2's has length $n-2$ (but if $m=1$ there are $n-1$ 2's in total).

In any case, there are $m$ stabilisations to be chosen if $n\ge 1$.
\end{rmk}

\begin{proof}[Proof of Lemma \ref{cfe}]
As before, we let $\alpha := p/q$.

The statement is trivial if $\alpha = p/q = n+1/m$, so we can assume that both inequalities in the statement of the proposition are strict.

We'll prove the nontrivial case by induction on $n$. When $n=1$, the fraction associated to $1+1/m$ is $1+(1+m)/1$, whose continued fraction expansion is $[m+2]$. We need to expand the fraction $1+p/(p-q) = (2p-q)/(p-q) = (2\alpha-1)/(\alpha-1)$: the inequality $1+1/m < \alpha < 1+1/(m-1)$ can be read as $1/m < \alpha-1 < 1/(m-1)$, so that
\[
2+m-1< 2+\frac1{\alpha-1} = \frac{2\alpha-1}{\alpha-1} < 2+m:
\]
in particular, the first element of the continued fraction expansion we're looking at is $\lceil\frac{2\alpha-1}{\alpha-1}\rceil=m+2$, as we wanted.

Let's suppose that the statement holds for $n\ge 1$, and prove that it holds for $n+1$: the fraction associated to $n+1+1/m$ is $(2mn+m+2)/(mn+1) = 3-((mn-m+1)/(mn+1))^{-1}$, and the one associated to $p/q+1$ is $(2p-q)/(p-q) = 3-((p-2q)/(p-q))^{-1}$. In particular, both expansions start with a 3, and the first one continues with the expansion of $(mn+1)/(mn-m+1)$; in order to prove the statement, it's enough to show that the continued fraction expansion of $(mn-m+1)/(mn+1)$ is an initial segment of the expansion of $(p-2q)/(p-q)$.

Now, the algorithm for $n+1/m$ and $p/q - 1$ tells us to expand the two fractions $1+(mn+1)/(mn-m-1)$ and $1+(p-q)/(2p-q)$, and by the inductive hypothesis the expansion of the first one is the initial segment of the expansion of the second one. The result follows.
\end{proof}

We can immediately draw two corollaries:

\begin{cor}\label{ge-n}
If $n$ is a positive integer and $c(\xi^-_n(L))\neq 0$, then for every $p/q\ge n$, $c(\xi_{p/q}(L))\neq 0$ whenever the first stabilisation for $p/q$-surgery is a negative.
\end{cor}

\begin{cor}\label{ge-1m}
If $c(\xi_{n+1/m}(L))\neq 0$ for a positive integer $n$ and all positive integers $m$, then for all $p/q > n$ there is a sign choice for the Ding-Geiges algorithm such that $c(\xi_{p/q}(L))\neq 0$.
\end{cor}

For the following proposition, let's introduce the following notation: for $n\ge 0,m\ge 1$ integers, we denote with $\xi^-_{n+1/m}(L)$ any contact $(n+1/m)$-surgery on $L$ such that all the $m$ stabilisations are chosen to be negative. In particular, when $m=1$, this is consistent with Lisca and Stipsicz's notation for integral surgeries; it is understood that $\xi^-_{1/m}(L)$ is just $\xi_{1/m}(L)$, since there are no stabilisations involved.

\begin{prop}\label{stabvsQsurg}
For $n\ge 0, m\ge 1$ integers, the two contact structures $\xi^-_{n+1/m}(L)$ and $\xi^-_{n+1+1/m}(L^-)$ are isotopic.
\end{prop}

\begin{rmk}
The case $m=1$ in the proposition is proved by Lisca and Stipsicz \cite{LStrans}: the proof we present here is a refinement of their first proof.
\end{rmk}

\begin{proof}
We'll prove the result by induction on $n$.

When $n=0$, we're comparing $\xi_{1/m}(L)$ with $\xi_{1+1/m}(L^-)$. Suppose we have an open book $(F, h, L)$ for $(Y,\xi)$ compatible with $L$. According to Ozbagci \cite{Ozbob}, the open book $(F,D_L^{-m}\circ h)$ supports the contact structure $\xi_{1/m}(L)$.

We can construct an open book supporting $\xi^-_{1+1/m}$ as follows: the Ding-Geiges algorithm tells us that we need to push-off and stabilise (negatively, according to our choice) $L$ $m$ times, and do $+1$-surgery on $L$ and $-1$ on the push-off. We can realise $L$ and the push-off on the page of the same open book by doing $m$ positive stabilisations (using boundary-parallel arcs for the Murasugi sum inside $F$); the push-off is represented by a curve $L_1$ on the page, parallel to $L$ except that it runs once along each of the $m$ handles (see figure \ref{1miso}). Call $(F', h')$ the monodromy for $\xi^-_{1+1/m}(L^-)$, as shown in Figure \ref{1miso}.

\begin{figure}[h!]
\begin{center}
\includegraphics[scale=.75]{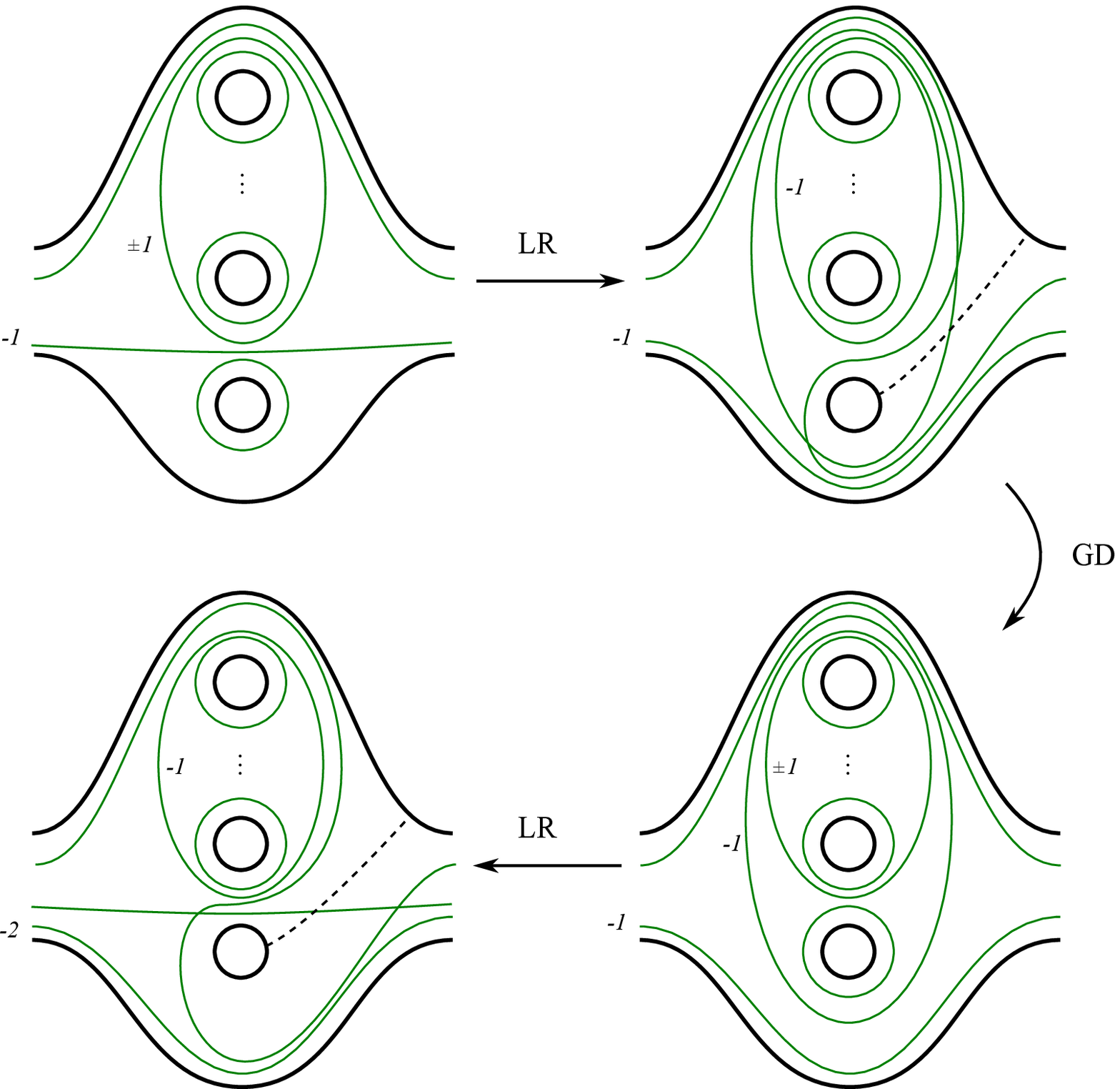}
\end{center}
\caption{The sequence of moves in Claim \ref{11m1mclaim}.}\label{1miso}
\end{figure}

\begin{claim}\label{11m1mclaim}
$\xi^-_{1+1/m}(L^-)$ is isotopic to $\xi_{1/m}(L)$.
\end{claim}

\begin{proof}
We now apply the lantern relation to the monodromy $h'$, where we insert a pair of canceling Dehn twists along the curve labelled with a $\pm1$ in Figure \ref{1miso}: after applying the relation, we see a destabilisation arc (dashed in the figure). After Giroux destabilising, we decrease the number of boundary components and we obtain the monodromy at the bottom right of the figure. If we insert another pair of opposite Dehn twist along the $\pm1$ curve, we can apply the lantern relation once again, and we see another destabilisation arc. The resulting open book looks now exactly like the one we had in the previous step (bottom right in the figure), with one less boundary component. After $m-1$ application of the lantern relation-destabilisation process, we end up with the open book we described for $\xi_{1/m}(L)$.
\end{proof}

Notice how, in this process, we always destabilised without any need for conjugation, so we actually proved that the two contact structures are isotopic rather than isomorphic. On the contrary, for the inductive step we'll first show:

\begin{claim}\label{n11mn1mclaim}
$\xi_{n+1+1/m}^-(L^-)$ is \emph{contactomorphic} to $\xi^-_{n+1/m}(L)$.
\end{claim}

\begin{proof}
We now refer to Figure \ref{n1miso}: the open book at the top left corresponds to the surgery $\xi_{n+1+1/m}^-(L^-)$; after applying the lantern relation once and conjugating, we find the destabilisation arc (dashed in the figure). The destabilisation arc intersects a single curve $d$ such that the monodromy factorises as $h_1\circ D_d\circ h_2$: in order to destabilise, we need to have a monodromy of the form $D_d \circ h_3$, so we need to conjugate $h$ with $h_1$; by conjugating we lose the isotopy result. After Giroux destabilising, we obtain the open book at the bottom, that represents $\xi_{n+1/m}^-(L)$.
\end{proof}

\begin{figure}[h!]
\begin{center}
\includegraphics[scale=.6]{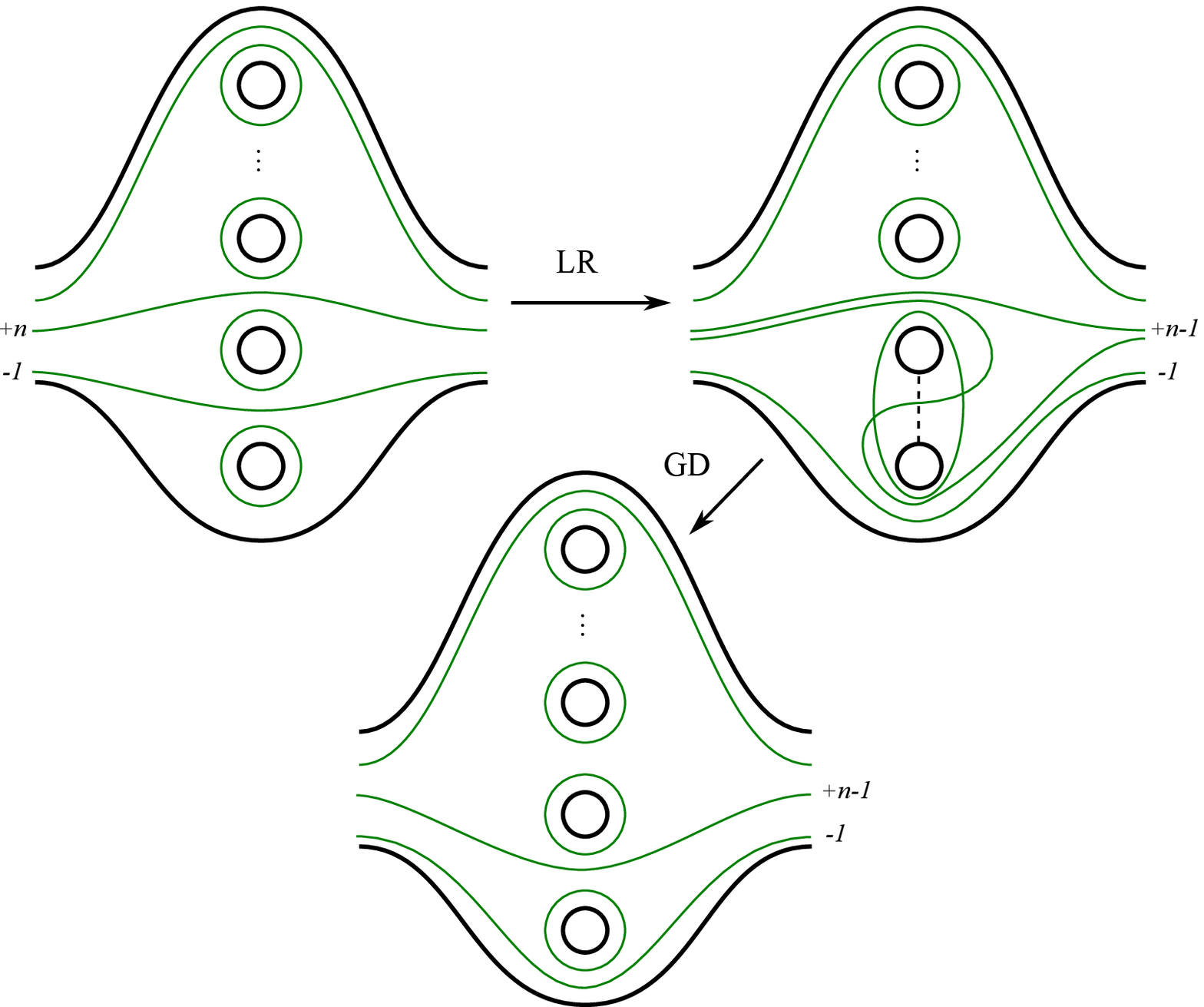}
\end{center}
\caption{The sequence of moves in Claim \ref{n11mn1mclaim}}\label{n1miso}
\end{figure}

\begin{rmk}\label{mix_stab}
In the proof of the inductive step, we never used the fact that all the last $m-1$ stabilisations are negative, but just that the first one is: in other words, $\xi_{n+1+1/m}(L^-)$ is isomorphic to $\xi_{n+1/m}(L)$ if the first stabilisation is negative for both surgeries, and the number of positive stabilisations on the last push-off is the same.

Moreover, using Lisca and Stipsicz's trick (see \cite[Lemma 2.3]{LStrans}), one proves that $\xi_{n+1+1/m}(L^-)$ is overtwisted if the first (respectively any) stabilisation is positive for all $n\ge 1$ (resp. for $n=0$): they prove the result by exhibiting an overtwisted disc in the surgered manifold, which is isotopic to the core of the first surgery handle relative to $L$ (not relative to $L^-$).
\end{rmk}

Using the remark, we can now complete the proof:

\begin{claim}
$\xi_{n+1+1/m}^-(L^-)$ is \emph{isotopic} to $\xi^-_{n+1/m}(L)$.
\end{claim}

\begin{proof}
We first argue that $\xi_{n+1+1/m}^-(L^-)$ is isotopic to one of the contact $(n+1/m)$-surgeries on $L$, by showing that the relevant solid torus is tight: if $L_0$ is the Legendrian representative of the positive trefoil with $tb(L_0)=1$, by the main result of \cite{LS1}\footnote{It also follows from Theorem \ref{mainthm} of this paper, and precisely from the implication whose proof is independent of this discussion.} we know that $\xi^-_{n+1+1/m}(L_0^-)$ is tight; this implies that the solid torus we attach to $S^3_L$ (not to $S^3_{L^-}$) to obtain $\xi_{n+1+1/m}^-(L^-)$ is tight. This solid torus is the union of a negative stabilisation layer and a surgery layer, and, since it's tight, it's one of the solid tori that we glue in to $S^3_L$ to get one of the contact $(n+1/m)$-surgeries: let's call this surgery $\overline{\xi}_{n+1/m}(L)$.

The argument above also shows that the choice of the signs of the stabilisations in the algorithm is independent of the Legendrian knot $L$: to get the claim, it suffices to prove the result in a particular case.

Consider the case $L=L_1$, where $L_1$ is the $n$-th negative stabilisation of a Legendrian positive torus knot $L_0$ with maximal Thurston-Bennequin number. By induction, $\overline{\xi}_{n+1/m}(L_1)$ is contactomorphic to $\xi_{1/m}(L_0)$, which in turn has nonvanishing contact invariant, so $c(\overline{\xi}_{n+1/m}(L_1))\neq 0$. Using the Lisca-Stipsicz trick mentioned above, we immediately see that the first push-off of $L_1$ has to be \emph{negatively} stabilised, and this already concludes the proof in the case $m=1$ (since there are no more stabilisation choices).

If $m>1$, the algorithm tells us that there are $m-1$ further stabilisations to do, and these latter stabilisations commute: let's suppose that $p\ge0$ of them are positive.

Thanks to Remark \ref{mix_stab} above, $\overline{\xi}_{n+1/m}(L_1)$ is contactomorphic the ${1+1/m}$-surgery on $L_0^-$ where the first (only) push-off of $L_0^-$ has been positively stabilised $p$ times and negatively stabilised $m-p$ times; by the second part of the remark, this latter contact structure is overtwisted if $p>0$. Since $\overline{\xi}_{n+1/m}(L_1)$ is isotopic to $\overline{\xi}_{n+1+1/m}(L_1^-)$, which in turn is contactomorphic to $\xi_{1/m}(L_0)$, and the latter is tight, we get $p=0$, \emph{i.e.} $\overline{\xi}_{n+1/m}(L_1)$ is isotopic to $\xi^-{n+1/m}(L_1)$ and is not isotopic to any other $(n+1/m)$-surgery on $L_1$.
\end{proof}

This also concludes the proof of the proposition.
\end{proof}

\begin{rmk}
Doing contact surgery along $L$ corresponds to gluing a solid torus with a tight contact structure to $S^3_L$. In particular, every contact $p/q$-surgery induces a map $\psi_{p/q}$ between $SFH(-S^3_{K,t})$ and $SFH(-S^3_{t+p/q}(K))$. When $p/q = n+1/m$, we'll denote the map corresponding to $\xi^-_{n+1/m}$ as $\psi^-_{n+1/m}$.

Notice that, when $n=0$, the sign choice is immaterial, and the map corresponds to $1/m$-surgery on $L$.
\end{rmk}

\section{Cables}\label{cont_cabling}

\subsection{Topological cabling}

Let $K$ be a nullhomologous knot in a 3-manifold $Y$. Take a tubular neighbourhood $K\subset N(K) \subset Y$, where we identify $N(K)$ with $\{z\in \C\mid |z|\le 1\}\times S^1$ in such a way that $K=\{0\}\times S^1$ and $\lambda = \{1\}\times S^1$ is nullhomologous in $Y$. Together with the meridian $\mu = \{|z|=1\}\times \{*\}$, $\lambda$ gives a parametrisation of $\de N(K)$.

\begin{defn}
Given $p>0$ and $q$ relatively prime integers, we define the $(p,q)$-\deff{cable} $K_{p,q}$ of $K$ to be any simple closed curve in $\de N(K)$, homologous to $p\lambda+q\mu$.
\end{defn}

\begin{rmk}
Here we adopt Hedden's and Hom's convention for the labelling of $p$ and $q$; Etnyre and Honda use the opposite convention.
\end{rmk}

Let's recall the following classical result:

\begin{prop}
The manifold $S^3_{pq}(K_{p,q})$ obtained with $pq$ surgery on $S^3$ along $K_{p,q}$ is diffeomorphic to the connected sum $S^3_{q/p}(K)\# L(p,q)$.
\end{prop}

\begin{rmk}
It's an open problem, called the cabling conjecture, whether the only rational surgeries that are reducible are surgeries along cables, as in the lemma.
\end{rmk}

We're interested in the behaviour of $\tau$ and $\epsilon$ under cabling. Hom answered precisely this question:

\begin{thm}[\cite{Hom}]\label{homthm}
$\tau(K_{p,q})$ and $\epsilon(K_{p,q})$ are determined by $p,q,\tau(K),\epsilon(K)$ in the following way:
\begin{itemize}
\item[1.] if $\epsilon(K)=0$, then $\tau(K_{p,q}) = \frac{(p-1)(q-\sgn(q))}2$; if $|q|\le1$, $\epsilon(K_{p,q})=0$, and $\epsilon(K_{p,q})=\sgn(q)$ otherwise;
\item[2.] if $\epsilon(K)\neq 0$, then $\tau(K_{p,q}) = p\tau(K)+\frac{(p-1)(q-\epsilon(K))}2$ and $\epsilon(K_{p,q}) = \epsilon(K)$.
\end{itemize}
\end{thm}

\subsection{Legendrian cabling}

We want to construct Legendrian cables of Legendrian knots through a ``standard'' construction: similar ideas appeared in \cite{Ru} (for Whitehead doubles), \cite{EH}, \cite{DGls} and, more recently, in \cite{CFHH}.

Consider an oriented Legendrian knot $L\subset (S^3,\xi_{\rm st})$ and its front projection; take $m$ push-offs of $L$ under the flow of $\de/\de z$, and twist them away from cusps, as in Figure \ref{L_mn}: notice that the twists are performed on strands that point \emph{to the right}. When $n\ge0$ is coprime with $m$, this is still the front projection of an oriented Legendrian knot, so the following definition makes sense:

\begin{figure}
\begin{center}
\includegraphics[scale=0.8]{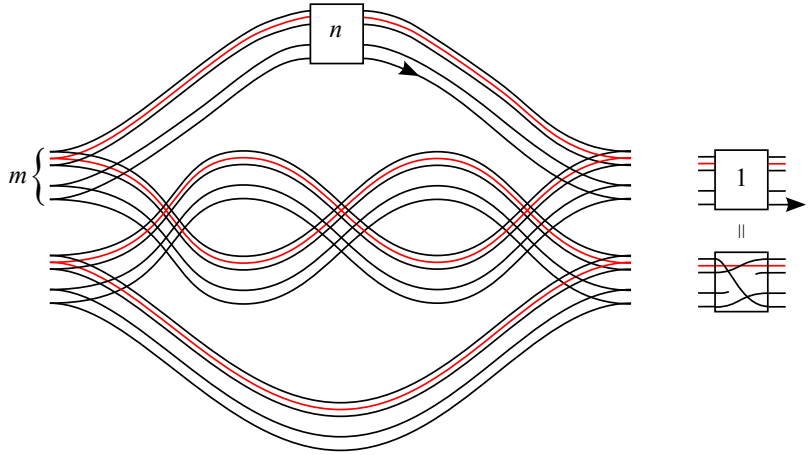}
\end{center}
\caption{A Legendrian cable (in black) of L (in red).}\label{L_mn}
\end{figure}

\begin{defn}
We'll call \deff{Legendrian} $(m,n)$-\deff{cable} the Legendrian knot $L_{m,n}$ obtained by the procedure we just described.
\end{defn}

\begin{rmk}
We defined $L_{m,n}$ starting from the front projection of $L$, so \emph{a priori} $L_{m,n}$ depends on the diagram and on the position of the twists. On the other hand, one can see that we can trade a twist (or better, one $m$-th of a twist) for a cusp displacement (see the first two diagrams in Figure \ref{L21L23} for an example when $m=2$), and vice-versa. In this trade, though, we need to change the box in Figure \ref{L_mn} with its horizontal reflection (denoted with $-n$ in Figure \ref{Lmn-wd}); however, moving through two cusps we recover the original picture. So we can move the twists across two consecutive cusps: according to our orientation convention when placing the twists, $L_{m,n}$ is independendent of their position in the diagram.

As for the independence of the diagram of $L$, one can check that there is a ``multiple strand'' version of the Legendrian Reidemeister moves \cite{EtnLTK} (\emph{i.e.} $L_{m,0}$ is well-defined): since these ``multiple strand'' moves are local, and we can move the twists away from the region involved, we get that $L_{m,n}$ is well-defined.
\end{rmk}

\begin{figure}
\begin{center}
\includegraphics[scale=0.43]{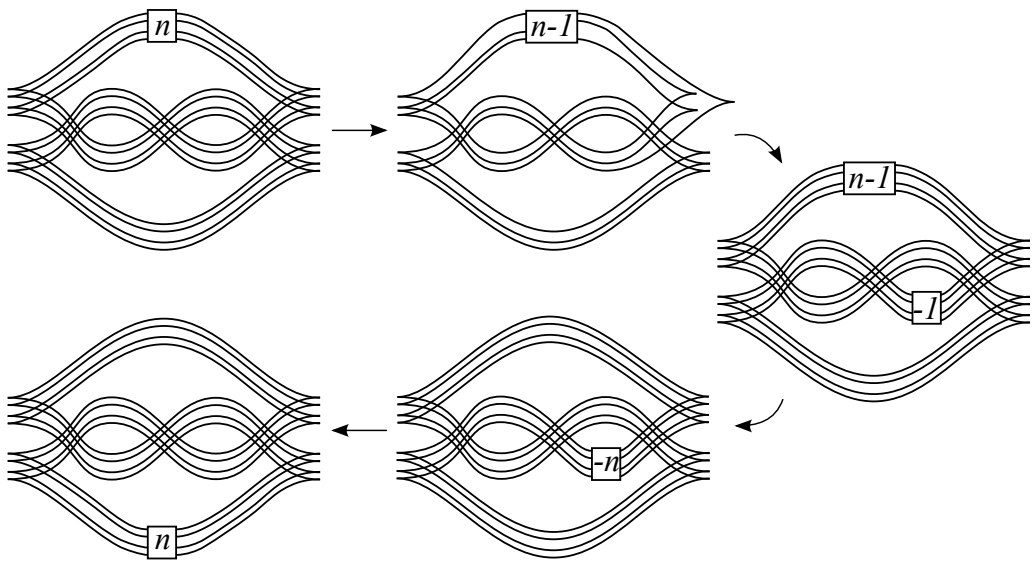}
\end{center}
\caption{A sequence of moves showing that $L_{m,n}$ doesn't depend on the position of the twists}\label{Lmn-wd}
\end{figure}

\begin{rmk}\label{uniquecabling}
Legendrian representatives of cables have proven to be a remarkable source of examples: for example, some (3,2)-cables of the trefoil (and, more generally, cables of positive torus knots) are not Legendrian or transversely simple (see \cite{EHcables, ELT}). On the other hand, this is the only possible \emph{local} construction, \emph{i.e.} a cabling construction obtained by attaching a cabling layer to $S^3_L$ (see \cite{DGls}). In particular, the main theorem in \cite{DGls} gives another proof of the well-definedness of $L_{m,n}$, which also gets rid of the orientation choice we made: we can place the box on a left-oriented strand, and this still gives a Legendrian representative of the same knot (see Proposition \ref{top-type} below) with the same classical invariants. Ding and Geiges' result now imply that the two (\emph{a priori} different) constructions give the same Legendrian knot.
\end{rmk}

We want to compute the classical invariants of $L_{m,n}$, given the classical invariants of $L$: say that $L$ is of topological type $K$ and has Thurston-Bennequin and rotation numbers $t=tb(L)$ and $r=r(L)$ respectively. From now on, we denote with $L$ and $L_{m,n}$ both the knots and their front projections.

\begin{prop}\label{top-type}
$L_{m,n}$ is of topological type $K_{m,mt+n}$.
\end{prop}

\begin{proof}
The curve $L_{m,n}$ is topologically isotopic to a curve on the boundary $\de N(L)$ of a tubular neighbourhood of $L$, and, by counting intersections with a disc section of $N(L)$, we get $[L_{m,n}] = m[L]\in H_1(N(L))$. So $L_{m,n}$ is of topological type $K_{m,q}$ for some $q$.

To determine this last parameter, we compute the linking number of $L$ and $L_{m,n}$, using the red curve in Figure \ref{L_mn}: $\lk(L_{m,n}, L)$ is the algebraic sum of the crossings of $L$ and $L_{m,n}$, divided by 2.

The crossings of $L$ and $L_{m,n}$ are of one of two kinds: the crossings coming from crossings of the diagram of $L$, and the ones coming from cusps of $L$. The former contribute for $2m\wr(L)+2n$, and the latter for $-mc(L)$, where $\wr(L), c(L)$ are the writhe and the number of cusps of $L$.

Summing up, we get
\[q=\lk(L_{m,n},L) = m(\wr(L)-c(L)/2)+n = m\cdot t+n.\]
\end{proof}

\begin{prop}
The classical invariants for $L_{m,n}$ are:
\[
\begin{array}{l}tb(L_{m,n}) = m^2t+(m-1)n;\\r(L_{m,n}) = mr;\\ sl(L_{m,n}) = m^2t-mr+(m-1)n.\end{array}\]
\end{prop}

\begin{proof}
This is a straightforward computation: the only cusps in the front projection for $L_{m,n}$ are the ones coming from the cusps for $L$, and they're $m$ times as many, of the same sign, so in particular
\[
r(L_{m,n}) = mr;\]
on the other hand, there are three kinds of crossings, coming from the crossings of $L$ ($m^2$ times as many), from the cusps of $L$ ($\binom{m}{2}$ of them for each cusp), and $(m-1)n$ of them coming from the $n$ twists. 

Summing them up with the appropriate signs, we get
\[tb(L_{m,n}) = m^2\wr(L)-\binom{m}{2}c(L)+(m-1)n-\frac{m}2c(L) = m^2t+(m-1)n.\]
Combining the two results, we get the self-linking number:
\[
sl(L_{m,n}) = m^2t+(m-1)n-mr.
\]
\end{proof}

One can now compare Hom's formulae for $\tau$ and $\epsilon$ with the proposition above: using Plamenevskaya's inequality (\ref{olga}), one checks:

\begin{prop}\label{homvsmain}
Let $L$ be a Legendrian knot in $(S^3,\xi_{\rm st})$ of topological type $K$, such that $\epsilon(K)\neq 0$. Then:
\begin{itemize}
\item[(i)] (SL) holds for $L_{m,n}$ if and only if (SL) and (TN) hold for $L$;

\item[(ii)] suppose that (SL) holds for $L_{m,n}$: (SC) holds for the pair $(L_{m,n},p)$ if and only if $p\ge 1-m\cdot r(L)$;

\item[(iii)] (TN) holds for $L_{m,n}$ if and only if (TN) holds for $L$.
\end{itemize}

On the other hand, if $\epsilon(K)=0$ (and therefore $\tau(K)=\nu(K)=0$):
\begin{itemize}
\item[(i')] (SL) holds for $L_{m,n}$ if and only if (SL) holds for $L$ and $n\ge 1-m\cdot tb(L)$;

\item[(ii')] suppose that (SL) holds for $L_{m,n}$: (SC) holds for the pair $(L_{m,n},p)$ if and only if $p\ge 1-m\cdot r(L)$;

\item[(iii')] (TN) holds for $L_{m,n}$ if and only if $n\ge -1-m\cdot tb(L)$.
\end{itemize}
\end{prop}

Before stating the following lemma, let's introduce some notation: given a Legendrian knot $L$, we denote with $L^{(1)}=L^-$ its negative stabilisation, and recursively $L^{(n+1)} = (L^{(n)})^-$. For small values of $n$, we may use the `differential' notation $L^{(1)}=L'$, $L^{(2)}=L''$...

\begin{lemma}
$(L^{(k)})_{m,km+n}$ is isotopic to $(L_{m,n})^{(km)}$.
\end{lemma}

\begin{rmk}
Notice that, if such an identity exists, then the triple of numbers $(m, n, km)$ defining the second knot is uniquely identified by the classical invariants of the first knot. Also, since the cabling operation commutes with orientation-reversal, the same result holds for positive stabilisations (on both knots).
\end{rmk}

\begin{proof}
It's enough to prove the result for $k=1$: this case is an easy induction on $m$. For $m=2$, we can apply the second Legendrian Reidemeister (LR2, see \cite{EtnLTK} for details) move twice as in Figure \ref{L21L23}.

Let's now suppose we want to prove the result for $m+1$: we can apply LR2 $2m$ times as in the base case, and reduce to the inductive assumption.
\end{proof}

\begin{figure}
\begin{center}
\includegraphics[scale=.6]{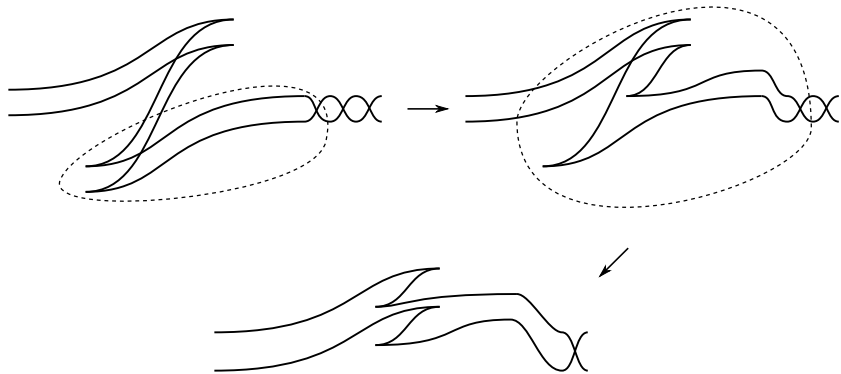}
\caption{A Legendrian isotopy from $(L')_{2,3}$ (top left) to $(L_{2,1})''$ (bottom).}\label{L21L23}
\end{center}
\end{figure}

\begin{rmk}
The cabling construction, stabilisations and the proof of the lemma above are all \emph{local}, in the sense that they all take place in a neighbourhood of $L$.

In particular, there exists a \deff{cabling layer} $(\bfT_{m,n}, \xi_{m,n})$, that is topologically a difference of solid tori: the core of the inner torus winds $m$ times around the outer solid torus, there are two $\xi_{m,n}$-dividing curves on the `outer' (respectively `inner') boundary component are homologous to the longitude of the bigger (resp. smaller) solid torus.

What the previous lemma says at the level of contact layers is that we have two isotopic layers: one is obtained by gluing a stabilisation layer $(\T^2\times I,\eta_-)$ (\emph{i.e.} a specific basic slice) from the \emph{back} to the \emph{outer} boundary of $\bfT_{m,n}$; the other is obtained by gluing $m$ stabilisation layers $(\T^2\times I,\eta_-)$ from the \emph{front} to the \emph{inner} boundary of $\bfT_{m,m+n}$.

This also allows us to generalise the notion of Legendrian cabling to Legendrian knots in any contact manifold.
\end{rmk}

This cabling layer induces a gluing map
\[\kappa_{m,n}:=\Psi_{\xi_{m,n}}: SFH(-Y_{K,t})\to SFH(-Y^3_{K_{m,mt+n},m^2t+n}).\]
Thanks to the remark, the lemma above can be translated to the sutured world as follows:

\begin{cor}\label{cablingvsstab}
$\kappa_{m,km+n}\circ\sigma_\pm^k = \sigma_\pm^{km}\circ\kappa_{m,n}$.
\end{cor}

\begin{rmk}
It's clear that $\psi_\infty\circ\kappa_{m,n} = \psi_\infty$, since the corresponding layers are both $\infty$-surgery layers, so they're isotopic.
\end{rmk}

We conclude the subsection with a remark on the action of $\kappa_{m,n}$ on sutured Floer homology, when $K\subset S^3$. 

\begin{prop}
$\kappa_{m,n}$ sends stable elements to stable elements, and therefore descends to a map of unstable complexes, still denoted with $\kappa_{m,n}$,
\[\kappa_{m,n}: SFH(-S^3_{K,t})/S\to SFH(-S^3_{K_{m,mt+n},m^2t+n})/S.\]
\end{prop}

\begin{proof}
Recall from Section 3 that the set of stable elements is $\ker (\sigma_+^N\circ \sigma_-^N)$ for some sufficiently large $N$, and by Theorem \ref{stabmaps} it's also spanned by $\ker(\sigma_-^N)$ and $\ker(\sigma_+^N)$. It therefore suffices to show that the proposition holds for an element $x\in SFH(-S^3_{K,t})$ such that $\sigma_\pm^k(x)=0$ for some $k$. It follows from Corollary \ref{cablingvsstab} that $\sigma_\pm^{km}(\kappa_{m,n}(x)) = \kappa_{m,mk+n}(\sigma_\pm^k(x)) = 0$.
\end{proof}

\subsection{Contact surgeries and cabling}

Let's start with an easy, general observation:

\begin{rmk}
Let $L$ be a nullhomologous Legendrian knot in $(Y,\xi)$, of topological type $K$. The Legendrian knot $L_{m,n}$ is of topological type $K_{m,m\cdot tb(L)+n}$. If we do contact $+n$-surgery on $L_{m,n}$, the topological surgery coefficient is $tb(L_{m,n})+n = m(m\cdot tb(L)+n)$, so the underlying manifold is reducible: more precisely it splits as $Y_{tb(L)+n/m}(K)\# L(m,n)$.
\end{rmk}

It's natural to ask whether there are natural contact structures on the two factors that realise this topological decomposition as a contact connected sum: in fact, this happens to be true (regardless of the homological assumption $[K]=0\in H_1(Y)$), when $n=1$.

\begin{prop}\label{1overm}
Contact $+1$-surgery on $L_{m,1}$ is isotopic to a contact connected sum of a contact $1/m$-surgery and a tight contact structure on $L(m,1)$.
\end{prop}

Before getting to the proof, let's find an open book for $(S^3,\xi_{\rm st})$ compatible with $L_{m,1}$. Start off with an open book $(F,h)$ for $(S^3,\xi_{\rm st})$ for which $L$ sits on a page, and is not nullhomologous in that page; we can assume that $L\subset F$ is a simple closed, not nullhomologous curve. Consider a properly embedded arc $c\in F$ that intersect $L$ in a single point, and consider the positive (Giroux) stabilisation $(F',h')$ of $(F,h)$ along $c$: the situation is depicted in Figure \ref{OBcable}. 

\begin{lemma}\label{Lm1ob}
The curve depicted on the right hand side of the figure represents $L_{m,1}$.
\end{lemma}

Before giving the proof of the lemma, let's recall a result of Ozsv\'ath and Stipsicz (see \cite[Section 4]{OzvSt}):

\begin{prop}
If $(F,h,L)$ is an open book decomposition with connected binding for $(Y,\xi)$ compatible with the nullhomologous Legendrian knot $L$, then:
\begin{itemize}
\item[(i)] there exists $[Z] = \zeta \in H_1(F)$ such that $[L] = h_*(\zeta)-\zeta$;
\item[(ii)] $r(L)$ is the Euler class of a 2-chain $P$ such that $\de P = L+Z-h(Z)$;
\item[(iii)] $tb(L)$ is $Z\cdot L$.
\end{itemize}
\end{prop}

This has some very interesting consequences, whose proof is straightforward:

\begin{cor}
If $L,L_1,L_2$ are three embedded, homologically nontrivial curves in $F$, such that $[L] = [L_1]+[L_2]$, and $Z_i$ is associated to $L_i$ as in \emph{(i)}, then:
\begin{itemize}
\item $tb(L) = tb(L_1)+tb(L_2) + Z_1\cdot L_2 + Z_2\cdot L_1$;
\item $r(L) = r(L_1)+r(L_2)$.
\end{itemize}
\end{cor}

\begin{proof}[Proof of Lemma \ref{Lm1ob}]
Let's remark that $L_{m,n}$ is Legendrian isotopic to a torus knot in a standard Legendrian neighbourhood $\nu(L)$ of $L$.

Call $L'$ the Legendrian knot represented by the curve on the right hand side of Figure \ref{OBcable}, and suppose that $\de F'$ is connected; in particular $\de F$ has two connected components.

Stabilising along $c$ corresponds to a connected sum of $(Y,\xi)$ with $(S^3,\xi_{\rm st})$. In particular, we can suppose that a neighbourhood of $c$ in $Y$ is contained in $\nu(L)$; call $b$ the core of the annulus we do Murasugi sum with (which is the closure of $c$ inside the new 1-handle of $F'$). In this case, the connected sum can be performed inside $\nu(L)$. In particular, $b$ is nullhomologous in $S^3$. Observe now that the curve $L'$ is isotopic in $S^3$ to a curve on the boundary of $\nu(L)$; $L'$ is homologous to $m\lambda_\xi + b$, where $b$ is given the orientation such that $b\cdot L = 1$, and in particular $L'$ represents a $(m,m\cdot tb(L)+1)$-cable of $K$.

Since $L'$ and $L_{m,1}$ are both local modifications of $L$, and torus knots in the standard solid torus are Legendrian simple \cite{EH}, the knot represented by $L'$ is isotopic to $L_{m,1}$ provided they have the same classical invariants.

But the corollary above provides us the tools we need to compute $tb(L')$ and $r(L')$: we know that $[L'] = m[L]+[b]$, and we know that for some $Z_0\subset F$, $[h(Z_0)]-[Z_0] = [L]$.

\begin{claim}
If $Z_b$ is parallel to a boundary component of $F\subset F'$, oriented so that $b\cdot Z_b = 1$, then $[h'(Z_b)]-[Z_b] = [b]$.
\end{claim}

\begin{proof}
The action of $D_b$ on homology is given by $[D_b(\gamma)] = [\gamma] + (b\cdot\gamma)[b]$. Since $[h(Z_b)] = [Z_b]$ and $h' = D_b\circ h$, it follows that $[h'(Z_b)] = [Z_b]+(b\cdot Z_b)[b]$.
\end{proof}

\begin{claim}
$[h'(Z_0)] - [Z_0] = [L]+(b\cdot Z_0+1)[b]$.
\end{claim}

\begin{proof}
We know that $[h(Z_0)] = [Z_0]+[L]$, and that $[D_b(c)] = [c]+(b\cdot c)[b]$:
\[[h'(Z_0)] = [D_b(h(Z_0))] = [D_b(Z_0)] + [D_b(L)] = [Z_0] + [L] + (b\cdot Z_0+b\cdot L)[b],\]
that is what we wanted, since, by assumption, $b\cdot L = 1$.
\end{proof}

In particular, if we let $Z$ be any curve in the homology class $[Z_0] - (b\cdot Z_0+1)[Z_b]$, then $[h(Z)]-[Z] = [L]$.

We can now compute $tb(L')$:
\[
tb(L') = m^2tb(L) + tb(b) + m(Z\cdot b + Z_b\cdot L) = m^2tb(L)-1+m,
\]
where we used that, $b\cdot Z = b\cdot Z_0$, $Z_b\cdot L = 0$ and $tb(b) = -1$. This last identity comes from the fact that $b$ represents a Legendrian unknot with Thurston-Bennequin number $-1$ in the $(S^3,\xi_{\rm st})$ connected summand, corresponding to the stabilisation made on $F$ to get $F'$.

As we said, the rotation number is linear, so $r(L') = mr(L)+r(b) = mr(L)$, and in particular $L'$ and $L_{m,1}$ have the same classical invariants, and Ding and Geiges' results \cite{DGls} (see Remark \ref{uniquecabling} above) imply that they're isotopic.
\end{proof}

Recall that $b$ is the core of the annulus we do Murasugi sum with; let's also denote with $\beta$ the positive Dehn twist along $b$, and with $\lambda$ the positive Dehn twist along $L\subset F'$.

\begin{figure}
\begin{center}
\includegraphics[scale=.55]{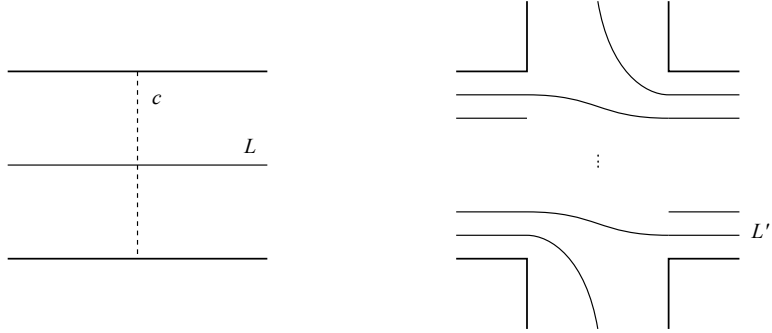}
\end{center}
\caption{The figure on the left shows the open book $(F,\phi,L)$ and the stabilisation arc $c$. The figure on the right represents $(F',\phi',L_{m,1})$: $L_{m,1}$ runs $m$ times along $L$ and once along the new 1-handle.}\label{OBcable}
\end{figure}

\begin{proof}[Proof of Theorem \ref{1overm}]
An open book for $+1$ surgery on $L_{m,1}$ has $F'$ as a page, and the monodromy is the composition of $h'$ with a negative Dehn twist along the curve representing $L_{m,1}$; this last curve is isotopic to $\lambda^{-m}b$, so the monodromy can be written as $(\lambda^{-m}\beta\lambda^m)^{-1} h' = \lambda^{-m}\beta^{-1}\lambda^m\beta h$.

Fix a contact structure $\xi_m$ on $L(m,1)$ supported by an open book with an annular page $A$, and monodromy given by $m$ positive twists along the core of the annulus. Fix also a properly embedded arc $d$ connecting the two boundary components of the annulus.

Similarly, an open book for $+1/m$-surgery along $L$ is given by $(F,\lambda^{-m}h)$; an open book for the connected sum of this with $\xi_m$ can be realised via a Murasugi sum of $F$ and $A$ along the arcs $d\subset A$ and an arc $c'\subset F$. Instead of using $c'=c$, though, we'll use the arc $c' = \lambda^{1-m}c$, so to simplify the proof: heuristically, we're sort of forced to choose $c' = \lambda^pc$ for some integer $p$, since, being the result local, we expect the proof to be local around $L$, and the only action that the monodromy can have on $c$, local around $L$, is by powers of $\lambda$.

The new surface is diffeomorphic to $F'$, and the monodromy, under the obvious identification, is $\lambda^{1-m}\beta^m\lambda^{-(1-m)} \cdot \lambda^m \cdot h = \lambda^{1-m}\beta^m\lambda^{-1}h$.

Notice that $b$ and $L$ intersect exactly once by assumption.

To prove isotopy of the two contact structures, we prove that the two monodromies are isotopic, and this is an easy computation in the mapping class group of $F'$:
\[
\lambda^{-m}\beta^{-1}\lambda^m\beta h = \lambda^{1-m}\beta^m\lambda^{-1} h \,\Longleftrightarrow\, \beta^{-1}\lambda^m \beta = \lambda\beta^m\lambda^{-1},
\]
and this last equation follows from taking the $m$-th power of the braid relation $\beta^{-1}\lambda\beta = \lambda\beta\lambda^{-1}$.
\end{proof}

\begin{rmk}
There is also a less direct proof of Proposition \ref{1overm}: the idea is that for some $L$, \emph{e.g.} the right-handed trefoil with maximal Thurston-Bennequin, every Legendrian cable has $tb(L_{m,1}) = 2g(L_{m,1})-1$, and therefore the contact structure $\xi_{+1}(L_{m,1})$ has nonvanishing contact invariant \cite{LS1}. This implies that the layer $\kappa_{m,1}\cup\bfT^-_n$ has nonvanishing contact invariant, and in particular is tight. But $\kappa_{m,1}\cup\bfT^-_n$ \emph{topologically} decomposes as $\bfT_{1/m}\# L(m,1)$, and the contact structure on the $\bfT_{1/m}$ factor has to be tight, and is the only tight solid torus that gives contact $1/m$-surgery.

Notice that, if $m\ge 4$, this doesn't immediately say anything about the contact structure on the lens space factor: it's not unreasonable that one can extract this information from the $\spin$-structures of the contact invariants $c(\xi_{1/m}(L))$ and $c(\xi_{+1}(L_{m,1}))$.
\end{rmk}

\vskip 0.2 cm

We now want to turn to the case of $L_{m,n}$ with $n\equiv 1 \pmod{m}$: we start off with an example.

\begin{ex}
Let's consider the case of the Legendrian positive trefoil $L_0\subset (S^3,\xi_{\rm st})$ with $tb(L_0)=1$ and $r(L_0)=0$. $L_0$ satisfies all three conditions in Theorem \ref{mainthm}; more precisely, since $\epsilon(T_{3,2}) = 1$, proposition \ref{homvsmain} tells us that all of its Legendrian cables $(L_0)_{m,n}$ satisfy the three hypotheses in Theorem \ref{mainthm} for any positive surgery coefficient.

Therefore, if we accept the `if' direction of our main theorem (whose proof won't rely on these facts), we know that $\xi^-_n((L_0)_{m,n})$ has nonvanishing contact invariant\footnote{This follows also from Lisca and Stipsicz's main theorem in \cite{LS1}, since positive cables of the trefoil have $\tau = g$.} (in fact, also $\xi^+_n((L_0)_{m,n})$ does).
\end{ex}

Thanks to the example, we know that $c(\xi^-_n(L_{m,n}))\neq 0$ for some Legendrian $L$: in particular, the contact invariant $EH(\bfT_{m,n}\cup \bfT^-_n, \xi_{m,n}\cup \xi^-_n)$ of the union of the cabling layer and the surgery layer has nonvanishing contact invariant (see Corollary \ref{nonvanishing}), and therefore is tight.

The layer $\bfT_{m,n}\cup \bfT^{\pm}_n$ splits as a connected sum $\bfT_{n/m}\# L(m,n)$ of tight manifolds, therefore we proved:

\begin{prop}
The contact structure $\xi^-_n(L_{m,n})$ splits as the connected sum $\xi_{n/m}(L)\#\eta_{m,n}$ for some choice of $n/m$ surgery along $L$ and a contact structure $\eta_{m,n}$ on $L(m,n)$; both the surgery layer and the contact structure on the lens space are independent of $L$.
\end{prop}

We want to pin down the choice of the contact structures $\xi_{n/m}$ and $\eta_{m,n}$ in the statement above, when $n\equiv 1\pmod{m}$.

\begin{prop}\label{cablingvssurg}
Suppose $m\ge 1$. Contact $n$-surgery on $L_{m,n}$ yields the connected sum $\xi^-_{n/m}(L)\#\eta_m$, where $\eta_m$ is obtained by $-1$-surgery on the Legendrian unknot with $(tb,r) = (1-m, 2-m)$.
\end{prop}

This proof is similar in spirit to the proof of Proposition \ref{stabvsQsurg}: the key point of the proof is Remark \ref{mix_stab}. Recall that when $n\equiv 1 \pmod{m}$ is larger than 1, we have $m$ stabilisations to choose, and the last $m-1$ choices commute (and all choices commute if $n=m+1$). In the following, $L_0$ will be the Legendrian right-handed trefoil with $tb(L_0)=0$.

\begin{proof}
We first take care of the lens space summand.

\begin{claim}
The contact structure on the lens space is $\eta_m$.
\end{claim}

\begin{proof}
Let $k= \lfloor n/m \rfloor = (n-1)/m$; in particular, $mk=n-1$.

Since the contact structure on $L(m,1)$ doesn't depend on the particular Legendrian knot, we can pick any $L$: let $L = L_0^{(k)}$ be a $k$-th negative stabilisation of the trefoil $L_0$. Then:
\[\xi^-_n(L_{m,n}) = \xi^-_n((L_0)_{m,1}^{(mk)}) = \xi_{+1}((L_0)_{m,1}) = \xi_{1/m}(L_0)\#\eta_m,\]
where the first equality follows from Lemma \ref{cablingvsstab}, the second from Proposition \ref{stabvsQsurg}, and the third from Proposition \ref{1overm}.
\end{proof}

\begin{rmk}
In the proof of the previous claim, we need to use a knot $L_0$ such that $\xi_{+1}(L_0)_{m,1}$ is tight in order to have uniqueness (up to isomorphism) of the connected sum decomposition (see \cite{DGpd}). As a byproduct of the proof, we obtain that $\xi_{n/m}(L)$ has to be tight for a $k$-th stabilisation of $L_0$.
\end{rmk}

As in the proof of Proposition \ref{stabvsQsurg}, we now rule out all other possibilities for $\xi_{n/m}(L)$.

\begin{claim}
$\xi^-_{m+1}(L_{m,m+1}) = \xi^-_{1+1/m}(L)\#\eta_m$.
\end{claim}

\begin{proof}
Suppose that the contact structure on $\xi^-_{1+1/m}(L)$ was obtained by doing at least one positive stabilisation on the (only) push-off of $L$, as dictated by the Ding-Geiges's algorithm.

Suppose that $L=L_0'$: by the remark above, we know that the contact structure $\xi_{1+1/m}(L)$ is tight, and by Remark \ref{mix_stab} we have that if there is one positive stabilisation, then $\xi_{1+1/m}$ is overtwisted. Therefore, in this particular case, the surgery layer is $\bfT^-_{1+1/m}$.

But this layer is independent of $L$, concluding the proof.
\end{proof}

\begin{claim}
$\xi^-_{km+1}(L_{m,km+1}) = \xi^-_{k+1/m}(L)\#\eta_m$.
\end{claim}

\begin{proof}
As above, let $L=L_0^{(k)}$. On one hand, we know that $\xi_{k+1/m}(L)$ is tight, and on the other hand (see Remark \ref{mix_stab}, and the proof of the claim above) every $\xi_{k+1/m}(L)$ that involves a positive stabilisation in the algorithm is overtwisted.

Again, the surgery layer is independent of the particular choice of $L$.
\end{proof}

Since we had proved the statement for $n=1/m$ in the previous section, we've exhausted all cases.
\end{proof}

\section{The main theorem}\label{mainproof}

\subsection{Technical lemmas}

We introduce here three technical lemmas, whose proofs will be given in the next section. The first one is due to Honda (unpublished), and is implicit in \cite{HKM2}. We'll give our own proof for convenience.

\begin{prop}\label{ass1}
The gluing map $\Psi_\xi: SFH(-M,-\Gamma) \to SFH(-M', -\Gamma')$ associated to an overtwisted contact structure $\xi$ on $N=M'\setminus {\rm Int}(M)$ is trivial.
\end{prop}

\begin{rmk}
Proposition \ref{ass1} is easily seen to be true if we restrict $\Psi_\xi$ to the subspace $\mathcal{EH}(M,\Gamma)$ of $SFH(-M,-\Gamma)$ generated by contact invariants $EH(M,\xi')$ such that $\de M$ is $\xi'$-convex and is divided by $\Gamma$.

In general $\mathcal{EH}(M,\Gamma)$ is a proper subset of $SFH(-M,-\Gamma)$: whenever the maximal Thurston-Bennequin number $\overline{tb}(K)$ of a knot $K\in S^3$ is strictly smaller than $2\tau(K)-1$, no unstable element in $SFH(-S^3_{K,2\tau(K)-1})$ can belong to $\mathcal{EH}(S^3_{K,2\tau(K)-1})$. 
\end{rmk}

Recall that a framed knot $(K,f)$ gives a surgery cobordism $W_f$, and the latter induces a map $F_{-W_f}: \HF(-S^3)\to \HF(-S^3_f(K))$. In Definitions \ref{psiinfty_def} and \ref{psin_def} we introduced the maps $\psi_\infty$ and $\psi_{+1}$ associated to contact $\infty$- and $+1$-surgery. In the last section we'll prove the following:

\begin{prop}\label{ass2}
The following diagram commutes:
\[
\xymatrix{
SFH(-S^3_{K,f})\ar[rr]^{\psi_{+1}}\ar[drr]_{\psi_\infty} & & SFH(-S^3_{f+1}(K)(1))\ar[r]^{\sim} &\HF(-S^3_{f+1}(K))\\
& & SFH(-S^3(1))\ar[r]^{\sim} & \HF(-S^3)\ar[u]^{F_{-W_{f+1}}}.
}
\]
\end{prop}

\begin{rmk}
As it happens for Proposition \ref{ass1}, it's easy to see that this triangle has to be commutative whenever we restrict the domain of $\psi_{+1}$ and $\psi_\infty$ to the subspace $\mathcal{EH}(S^3_{K,f})\subset SFH(-S^3_{K,f})$.
\end{rmk}

There is one more lemma, which is implicit in \cite{OSQsurg}:

\begin{prop}[\cite{OSQsurg}]\label{cobnon0}
The surgery cobordism map
\[
F_{-W_f}: \HF(-S^3)\to\HF(-S^3_f(K))
\]
is injective for $f=2\nu(K)$ and is zero for $f=2\nu(K)-2$.
\end{prop}

\subsection{Algebraic identities}

Recall that in Definition \ref{stabmaps_def} we introduced the notation $\sigma_{\pm}$ to denote the two gluing maps associated to the two types of stabilisation of an oriented knot $K$. In Definition \ref{psin_def} we introduced the notation $\psi_n^\pm$ for gluing maps associated to contact $n$-surgery. Lemma \ref{stabvsQsurg} (or rather, its proof) has a translation to the sutured world:

\begin{prop}\label{stabvssurg}
We have the following identities:
\begin{itemize}
\item[(i)] $\psi^{\pm}_{n+1}\circ \sigma_\pm = \psi^{\pm}_n$;
\item[(ii)] $\psi^{\pm}_n\circ\sigma_{\mp}=0$.
\end{itemize}
\end{prop}

These properties will be used along all the proof, and will be the tool that allow us to switch from small to very large surgery coefficients and back.

\begin{proof}
The first part follows directly from Proposition \ref{stabvsQsurg}; the second part is Lisca and Stipsicz's trick of opposite stabilisations (see Remark \ref{mix_stab}), coupled with Proposition \ref{ass1}.
\end{proof}

The following two propositions allow us to simplify the proof.

\begin{prop}\label{kills}
If $x$ is stable in $SFH(S^3_{K,f})$ then $\psi^{\pm}_n(x)=0$ for any $n\ge1$.
\end{prop}

\begin{proof}
Suppose first that $x\in S_- = \ker \sigma_-^N$. As a warm up, let's prove the theorem in the case $n=1$: for some sufficiently large $N$, we have
\[
\psi_1(x) = (\psi^-_{N+1}\circ\sigma^N_-)(x) = \psi^-_{N+1}(0) = 0.
\]
Theorem \ref{stabmaps} tells us that $\sigma_\pm$ is an isomorphism on $S_\mp$, so each $x\in S_-\subset SFH(S^3_{K,f})$ is the image of an element $x'\in S_-\subset SFH(S^3_{K,f-n+1})$, in the sense that we have $x=\sigma_+^{n-1}(x')$.

Using this and the warm-up, we have that for all $x\in S_-$ we can write:
\[
\psi^+_n(x) = (\psi^+_n\circ\sigma_+^{n-1})(x') = \psi_1(x'),
\]
and the latter vanishes, thanks to the warm up.

On the other hand,
\[
\psi^-_n(x) = (\psi^-_{n+N}\circ\sigma_-^N)(x) = 0,
\]
so we've proven that $S_-\subset\ker \psi^\pm_n$.

We can now exchange the roles of $+$ and $-$ signs in the proof, and obtain that the results holds also for $x\in S_+$, and this concludes the proof.
\end{proof}

Recall that in Lemma \ref{stablevspsiinfty} we proved an analogous result for $\psi_\infty$: we can therefore draw the following corollary.

\begin{cor}
Positive and $\infty$-surgery maps factor through the unstable complex.
\end{cor}

\begin{rmk}\label{unstable_quotient}
In the following we will often implicitly replace $SFH(-S^3_{K,f})$ with the \deff{unstable quotient} $SFH(-S^3_{K,f})/(S_++S_-)$ (see Remark \ref{naturality}). This latter group, as said, is a direct sum of copies of $\F$, sitting in different Alexander gradings.
\end{rmk}

In particular, we'll replace the maps $\psi^-_n$ and $\psi_\infty$ with their compositions with the projection $SFH(-S^3_L)\to SFH(-S^3_L)/(S_++S_-)$: we'll continue to call these maps $\psi_n^-$ and $\psi_\infty$, keeping in mind the domain change. The maps $\sigma_\pm$ will be considered as maps between unstable complexes, too. In what follows, this will be used without further mention.

We find it very convenient to organise all the unstable complexes (as $f$ varies among integers smaller than $2\tau(K)$) in a picture, like in Figure \ref{unst_cx}.

\begin{figure}
\begin{center}
\includegraphics[scale=1.8]{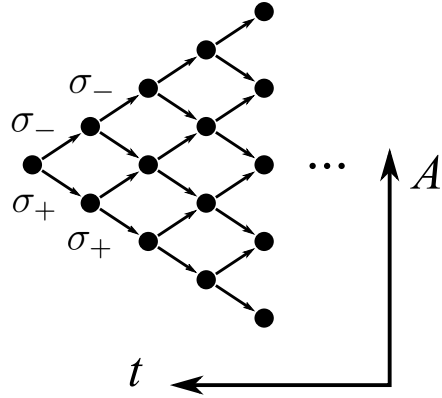}
\end{center}
\caption{Every dot in the picture represents a generator in the unstable complex of some $SFH(-S^3_{K,f})$. Each column is a (Alexander-)homogenous spanning set for the unstable comples for $SFH(-S^3_{K,f})$ for a \emph{fixed} slope $f$; this slope \emph{decreases} when moving to the right. The vertex represents the generator of $SFH(-S^3_{K,2\tau(K)-1})$. The vertical direction gives the Alexander grading induced by the identification with $\HFK(S^3_{f}(K),\ltilde{K})$. The arrows represent the action of $\sigma_\pm$ on the unstable complexes. The (unstable projection of the) invariant $EH(L)$ lands inside this triangle by Plamenevskaya's inequality (\ref{olga}) and Proposition \ref{EHgrading}.}\label{unst_cx}
\end{figure}

\subsection{The proof -- Independence and sufficiency}

We start by proving the last statement in the theorem.

\begin{prop}\label{independenceprop}
The contact invariant $c(\xi^{\pm}_n(L))$ is independent of the Legendrian isotopy class of $L$, when the classical invariants are fixed.
\end{prop}

Recall (see Remark \ref{naturalityrmk}) that what we mean by the statement above is that the \emph{orbit} of the contact invariant $c(\xi^{\pm}_n(L))$ is independent of the Legendrian isotopy class.

We need to make a small digression about a similar issue affecting the $EH$ invariants: suppose that $L_1$ and $L_2$ are two Legendrian representatives of $K$ in $(S^3,\xi_{\rm st})$. Then $EH(L_i)$ is an element in $SFH(S^3_{L_i})$, and if we want to compare $EH(L_1)$ with $EH(L_2)$, we need to give an identification $\alpha$ of the two knot complements. Any two such identifications differ by an element of $MCG(S^3_{L_i},\de S^3_{L_i})$, and this mapping class group acts on $SFH(-S^3_{L_i})$ \cite{JOT}.

We start with a lemma to say how much this identification matters, if $L_1$ and $L_2$ also have the same rotation number. Whenever possible, we'll implicitly assume that a specific identification has been made, without keeping track of it in the notation.

\begin{lemma}\label{EHnaturalitylemma}
Regardless of the identification of $S^3_{L_1}$ with $S^3_{L_2}$, $EH(L_1)-EH(L_2)$ is stable.
\end{lemma}

\begin{proof}
Fix a diffeomorphism $\alpha: S^3_{L_1}\to S^3_{L_2}$. As a notational shorthand, call $\xi_2$ the restriction of $\xi_{\rm st}$ to $S^3_{L_2}$.

The diffeomorphism $\alpha$ induces a map $\alpha^*: SFH(-S^3_{L_2})\to SFH(-S^3_{L_1})$ that preserves the Alexander grading and carries $EH(L_2)$ to $EH(\alpha^*\xi_2)$. Since $L_1$ and $L_2$ also have the same rotation number, $EH(L_1)$ and $EH(L_2)$ have the same degree, and so does $EH(\alpha^*\xi_2)$. To prove the lemma, it's enough to show that $EH(\alpha^*\xi_2)$ is not stable: if so, it has the same degree as $EH(L_1)$, and $\psi_\infty(EH(\alpha^*\xi_2)) = \psi_\infty(EH(L_1))$, and their difference is stable by Lemma \ref{stablevspsiinfty}.

Consider any extension $\ltilde{\alpha}: S^3_1\to S^3_2$ of $\alpha$ from $S^3_1 = S^3_{L_1}\cup\nu(L_1)$ to $S^3_2 = S^3_{L_2}\cup\nu(L_2)$. Both $S^3_1$ and $S^3_2$ are diffeomorphic to $S^3$, but we keep the indices to keep them distinct.

If we do $\infty$-surgery on $L_2$, we get the standard contact structure $\xi_{\rm st; 2}$ on $S^3_2$. When we pull-back using $\ltilde\alpha$ we get a tight contact structure $\alpha^*\xi_{\rm st; 2}$ on $S^3_1$.

Since $\ltilde\alpha$ maps $\nu(L_1)$ diffeomorphically to $\nu(L_2)$, the pull back of the contact structure on $\nu(L_2)$ is a tight contact structure on $\nu(L_1)$: since such a contact structure is unique up to isotopy \cite{Ho}, it means that $\ltilde\alpha^*\xi_{\rm st; 2}$ is obtained from $\alpha^*\xi_2$ by contact $\infty$-surgery.

In other words, $\psi_\infty(EH(\alpha^*\xi_2)) = \ltilde\alpha^*(c(\xi_{\rm st; 2})) \neq 0$, and by Lemma \ref{stablevspsiinfty} $EH(\alpha^*\xi_2)$ is not stable.
\end{proof}

We now turn back to the independence statement.

\begin{proof}[Proof of Proposition \ref{independenceprop}]
If $L_1$ and $L_2$ have the same classical invariants, $EH(L_1)$ and $EH(L_2)$ both belong to $SFH(-S^3_{L_1})$ and are homogeneous of the same degree $r(L_1)$. Moreover, since both knots are Legendrian in the standard contact $S^3$, we have that $\psi_{\infty}(EH(L_1)) = \psi_{\infty}(EH(L_2)) = c(\xi_{\rm st})$, and in particular $EH(L_1)-EH(L_2)\in \ker \psi_\infty$, and since this difference is homogeneous, it's stable by Lemma \ref{stablevspsiinfty}.

The statement now is a straightforward consequence of Proposition \ref{kills}: since $\psi^{\pm}_n$ kills the stable subspace,
\[c(\xi^{\pm}_n(L_1)) - c(\xi^{\pm}_n(L_2)) = \psi^\pm_n(EH(L_1)-EH(L_2)) = 0.\]
\end{proof}

\begin{rmk}
Notice that the previous statement, together with Proposition \ref{stablevsOT} also proves that, if $(S^3,\xi)$ is overtwisted and $L$ is $\xi$-Legendrian, then $c(\xi^-_n(L))=0$ for all $n\ge 0$. In particular, this justifies our focus on surgeries on the standard contact $S^3$.
\end{rmk}

We now prove the sufficiency of the three conditions:

\begin{prop}
If (SL), (SC) and (TN) hold, $c(\xi^-_n(L))\neq 0$.
\end{prop}

\begin{proof}
Call $E:=EH(L), \tau:=\tau(K)=\nu(K), t:=tb(L)$ and $c:=c(\xi^-_n(L))$.

Recall that $\xi^-_{n+1}(L)$ is obtained from $\xi_n^-(L)$ through a Legendrian surgery, and if the latter has nonvanishing contact invariant, so does the former. Therefore, it's enough to prove the result when we have equality in \emph{(SC)}, that is when $n=2\tau-t$.

Graphically, the condition \emph{(SL)} means that $EH(L)$ lands on the top edge of the infinite triangle of Figure \ref{unst_cx}. In particular we can find $x$ in $SFH(S^3_{K,2\tau-1})/S$ such that $\sigma_-^{n-1}(x) = E$, and Proposition \ref{stabvssurg} tells us that $c = \psi^-_n(E) = \psi_1(x)$. By Proposition \ref{ass2}, $\psi_1(x) = F_{-W_{2\tau}}(\psi_\infty(x))$, and the latter is nonzero by the injectivity of $F_{-W_{2\nu(K)}}$ (Proposition \ref{cobnon0}), the condition \emph{(TN)}, and the nonvanishing of $c(\xi_\textrm{st}) = \psi_\infty(x)$.
\end{proof}

\subsection{The proof -- Necessity}

We now turn to the necessity of the three conditions: first we're going to prove that \emph{(SL)} is necessary, then we're going to prove that if \emph{(SL)} holds then \emph{(SC)} is necessary, and finally we're going to prove that if both \emph{(SL)} and \emph{(SC)} hold, then also \emph{(TN)} is necessary.

In the following, we'll call $E:=EH(L), \tau:=\tau(K), \nu:=\nu(K), t:=tb(L)$ and $c:=c(\xi^-_n(L))$.

\begin{prop}\label{SLnecessary}
If (SL) doesn't hold, then $c(\xi^-_n(L))=0$.
\end{prop}

\begin{proof}
Suppose that \emph{(SL)} doesn't hold. In this case, $A(E) = -r(L)$ is not maximal in the unstable complex, since the top generator in the unstable complex for $SFH(-S^3_{K,t})$ has Alexander degree $t-2\tau+1$; pictorially, the unstable component of $E$ is not the top generator in the relevant column on the left hand side of Figure \ref{onlyif}. This implies that $E$ is in the image of $\sigma_+$, and we can write $E = \sigma_+(x)$. So, by Proposition \ref{stabvssurg}:
\[
c = \psi_n^-(E) = \psi_n^-(\sigma_+(x)) = 0.
\]
\end{proof}

\begin{prop}
If (SL) holds but (SC) doesn't, then $c(\xi^-_n(L))=0$.
\end{prop}

\begin{proof}
Call $s$ the surgery index $s=t+n<2\tau$.

\begin{claim}
We can assume $n=1$.
\end{claim}

\begin{proof}
Since \emph{(SL)} holds, but \emph{(SC)} doesn't, the unstable part of $E$ lies in the top (slanted) row of Figure \ref{unst_cx}, and $E=\sigma_-^{n-1}(y)$ for some $y$ in the unstable part of $SFH(S^3_{K,s-1})$. By Proposition \ref{stabvssurg}, $\psi_n^-(E) = \psi_n^-(\sigma^{n-1}_-(y)) = \psi_1(y)$, and $y$ also fails to satisfy the condition \emph{(SC)}, in the sense that the slope $s-1$ of the sutures (that is, the algebraic counterpart of the Thurston-Bennequin number for $L$) satisfies $s-1+1 = s < 2\tau$.
\end{proof}

So, let's assume $n=1$, that is $E=y$ in the proof of the claim. Since $s=1+t<2\tau$, the unstable part of $E$ is not the left vertex of the triangle of the right hand side of Figure \ref{onlyif}. In particular, $\sigma_+(E)$ is in the image of $\sigma_-$, so we can write $\sigma_+(E)=\sigma_-(x)$ for some $x$ in the unstable part of $SFH(S^3_{K,t})$.

By Proposition \ref{stabvssurg}, $c = \psi_1(E) = \psi_2^+(\sigma_+(E)) = \psi_2^+(\sigma_-(x)) = 0$.
\end{proof}

\begin{figure}
\begin{center}
\includegraphics[scale=1.85]{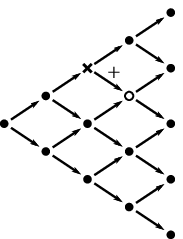}
\hskip 2.5 cm
\includegraphics[scale=1.85]{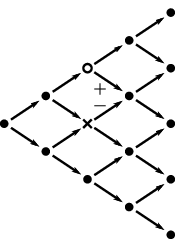}
\caption{The figure on the left represent the situation when \emph{(SL)} doesn't hold: the hollow circle is (the stable part of) $EH(L)$, and the $\times$ is the element $x$ in the proof. The figure on the right represent the situation when \emph{(SL)} holds but \emph{(SC)} doesn't: the hollow circle represents the element $y$, and the $\times$ represents the element $x$ in the proof.}\label{onlyif}
\end{center}
\end{figure}

Finally, for the last part we call into play the Legendrian cabling and surgeries along them.

\begin{prop}
Suppose (SL) and (SC) hold, but (TN) doesn't, then $c(\xi^-_n(L))=0$.
\end{prop}

\begin{proof}
By Corollary \ref{ge-n}, it's enough to show that $c(\xi^-_{n+1/2}(L))=0$.

By Proposition \ref{cablingvssurg}, $c(\xi^-_{n+1/2}(L)) \otimes c(\eta_2)= c(\xi^-_{2n+1}(L_{2,2n+1}))$. Since \emph{(TN)} doesn't hold, though, $\epsilon(K)=-1$, and Proposition \ref{homvsmain} tells us that $L_{2,2n+1}$ fails to satisfy \emph{(SL)}, so the right hand side vanishes by Proposition \ref{SLnecessary}. Since $\eta_2$ is obtained from $\xi_{\rm st}$ through Legendrian surgery, $c(\eta_2)$ is nonvanishing, and therefore $c(\xi^-_{n+1/2}(L))=0$, concluding the proof.
\end{proof}

\subsection{Corollaries}

We're going to prove the following statement, which is slightly stronger than Corollary \ref{qtight}:

\begin{prop}
If $\epsilon(K)=1$ (respectively $\epsilon(K)=0$) and there's a Legendrian representative $L$ of $K$ that satisfies (SL), then for all $q>2\tau(K)-1$ (resp. $q\ge 0$) the manifold $S^3_q(K)$ supports a tight contact structure.
\end{prop}

\begin{rmk}
This is in fact stronger than Corollary \ref{qtight}, since Hom \cite{Hom} proved that $\epsilon(K)=0$ implies $\tau(K)=0$. Observe also that if $\tau(K)=g(K)>0$, then automatically $\epsilon(K)=1$ (see \cite{Hom}); in fact, we have the same implication under the weaker hypothesis $\tau(K)=g_*(K)>0$: this is obtained as a combination of Proposition 2.1 in \cite{LS1} and Proposition \ref{cobnon0}. In particular, if $K$ is not slice and its maximal self-linking number $\overline{sl}(K)$ is $2g_*(K)-1$, every manifold $S^3_q(K)$ with $q>2g_*(K)-1$ supports a tight contact structure (compare with the main results in \cite{LS1} and \cite{LStrans}).
\end{rmk}

\begin{proof}
Suppose that $L$ and $K$ are as in the statement, with $L$ satisfying \emph{(SL)} and $K$ satisfying \emph{(TN)}. Theorem \ref{mainthm} tells us that for all integers $m\ge 2\tau(K)$, we have a contact structure on $S^3_m(K)$ with nonvanishing contact invariant, and therefore tight.

Using Corollary \ref{ge-n}, we obtain contact structures on $S^3_q(K)$ with nonvanishing contact invariants for all $q\ge 2\tau$.

Now, let's call $t=tb(L), r=r(L)$, and recall that \emph{(SL)} implies $r\le0$. There is nothing left to prove when $\epsilon(K)=0$, so we can suppose that $\epsilon(K)=1$.

Proposition \ref{homvsmain} tells us that for every $n\ge1-mr\ge 1$ the Legendrian cable $L_{m,n}$ satisfies the three hypotheses in our main theorem, so that $c(\xi_{1-mr}^-(L_{m,1-mr}))\neq 0$. But, by Proposition \ref{cablingvssurg}, $\xi^-_{1-mr}(L_{m,1-mr})$ splits as a connected sum $\xi^-_{-r+1/m}(L)\#\eta_{m}$, and in particular $c(\xi^-_{1/m-r}(L))\neq 0$. This is a contact structure on $S^3_{t-r+1/m}(K) = S^3_{2\tau(K)-1+1/m}(K)$.

Appealing to Corollary \ref{ge-1m} concludes the proof in the case $\epsilon(K) = 1$.
\end{proof}

\begin{rmk}
Notice that the same trick doesn't work if $\epsilon(K)=0$, because of the odd behaviour of $\tau$ and $\epsilon$ for cables when the cabling coefficient $q$ in Theorem \ref{homthm} goes from positive to negative: these values are the `critical' values that allow us to reach slopes below $2\tau(K)$ when $\epsilon(K)=1$.
\end{rmk}

Let's recall now the definition of the transverse invariant $\tilde{c}$ \cite{LStrans}. Fix a topological knot $K\subset S^3$: the sequence of groups $\left(\HF(-S^3_n(K))\right)$ comes with a collection of maps $F_{\overline{W}_n}: \HF(-S^3_{n}(K))\to \HF(-S^3_{n-1}(K))$, and together they give rise to an inverse system $\left\{\HF(-S^3_{n}(K)), \phi_{f,g}\right\}_{g<f}$, where $\phi_{f,g}$ is the composition $F_{\overline{W}_f}\circ\dots\circ F_{\overline{W}_{g+1}}$. Lisca and Stipsicz call this inverse limit $H(S^3,K)$.

\begin{defn}
Given a transverse knot $T$, the invariant $\stilde{c}(T)$ is the class of the sequence $(c(\xi^-_n(L)))_{n\in\mathbb{N}}$ in $H(S^3,K)$, where $L$ is a Legendrian approximation of $T$.
\end{defn}

There's an ambiguity in the definition of $\stilde{c}$, coming from the ambiguity in the definition of $c$: once we fix a Legendrian approximation $L$ of $T$ and an identification of $S^3_L$ with the ``abstract'' sutured manifold $S^3_{K,tb(L)}$, though, $\stilde{c}$ is well-defined. The equality in the statement of Corollary \ref{ctilde} has to be understood in the sense that the two elements are the same up to fixing the two identifications.

It's proved in \cite{LStrans} that the invariant above is non-trivial (in the sense that it's not identically zero). On the other hand, we prove here that it doesn't detect more than the classical invariants:

\begin{proof}[Proof of Corollary \ref{ctilde}]
We know that $c(\xi^-_n(L))=0$ if $c(\xi)=0$, since $S_\pm \subset \ker\psi^\pm_n$, and we know that if $\xi=\xi_{\rm st}$, $c(\xi^-_n(L))=0$ unless $sl(T)=2\tau(K)-1$ and $\tau(K)=\nu(K)$.

Suppose therefore that $sl(T)=2\tau(K)-1 = 2\nu(K)-1$, and let $L'$ be \emph{any} Legendrian knot of topological type $K$ such that $tb(L')-r(L') = 2\tau(K)-1$ ($L'$ doesn't need to be a Legendrian approximation of $T$). Call $d$ the difference $d=tb(L)-tb(L')$, and suppose that $d>0$. Then for every $n>|d|$, and for every two identifications of $S^3_{L^{(d)}}$ and $S^3_{L'}$ with $S^3_{K,tb(L)}$ we have $c(\xi^-_{n}(L)) = c(\xi^-_{n+d}(L'))$, by Theorem \ref{mainthm}: as a consequence, the classes of the two sequences in $H(S^3,K)$ coincide.

Therefore $\tilde{c}$ can only see whether the two equalities $sl(T)=2\tau(K)-1$ and $\tau(K)=\nu(K)$ hold, and these are equalities in the classical invariants for $T$.
\end{proof}

\section{Proofs of technical lemmas}\label{technical}

This section will be rather dry, and is a detailed account of the various technical ingredients used in the proof.

\subsection{The Heegaard Floer lemma}

Recall that we want to prove that the surgery cobordism map $F_{-W_f}$ induced by the surgery cobordism from $-S^3$ to $-S^3_f(K)$ is injective for $f=2\nu(K)$ and vanishes for $f=2\nu(K)-2$.

Similar results appeared in \cite[Proposition 3.1]{OSHF} and \cite[Proposition 3.1]{He2}; this refined result follows from a computation in \cite{OSQsurg}.

\begin{proof}[Proof of Proposition \ref{cobnon0}]
The map $F_{-W_f}$ fits into the surgery exact triangle
\[
\xymatrix{
\HF(-S^3_f)\ar[rr] & & \HF(-S^3_{f-1})\ar[dl]\\
& \HF(-S^3)\ar[ul]
}
\]
Recall that if in an exact triangle of vector spaces $(U,V,W)$ we have $\dim U + \dim V = \dim W$, then the map between $U$ and $V$ is the zero map.

Having this in mind, we can prove by direct computation, using the `mapping cone' construction of \cite{OSint} (see also \cite{Jake}), that:
\[
\dim \HF(-S^3_f(K)) - \dim \HF(-S^3_{f-1}(K)) = \pm\dim \HF(-S^3),
\]
where the sign is a plus if $f = 2\nu(K)$ and is a minus if $f=2\nu(K)-2$. In fact, in \cite[Proposition 9.1]{OSQsurg}, Ozsv\'ath and Szab\'o compute the ranks of the two groups on the left hand side when $\tau(K)\ge 0$:
\[
\dim \HF(-S^3_f(K)) = |f|+2\max\{0,2\nu(K)-1-f\}+D,
\]
where $D$ is a constant, depending only on $K$.

The condition $\tau(K)\ge0$ can be always achieved by taking the mirror of the knot, if needed. If $\tau(K)=\nu(K)=0$, this dimension has two minima at $f=\pm1$, and therefore the map $F_{-W_f}$ is injective if $f=0$ and $f\ge 2$, and zero otherwise. If $\nu(K)\ge 1$, on the other hand, the dimension has a single minimum at $f=2\nu(K)-1$ (in fact, the graph of the dimension is a traslation of the graph of the absolute value), therefore $F_{-W_f}$ is injective if and only if $f\ge 2\nu(K)$.

We can now use Hom's results \cite{Hom} to recover what happens when $\tau(K)<0$: in that case, $\tau(\overline{K}) > 0$, and in particular $\epsilon(\overline{K}) = -\epsilon(K)\neq 0$. If $\epsilon(K)=1$, then $\epsilon(\overline{K})=-1$, and $\nu(K)=\tau(K)$, while $\nu(\overline{K}) = \tau(\overline{K})+1$, and $\dim\HF(-S^3_{f}(\overline{K})) = \dim\HF(S^3_{-f}(K))$ has a single minimum at $-f = 2\nu(\overline{K})-1$, that is exactly $f=2\nu(K)-1$. Similarly, if $\epsilon(K)=-1$, $\nu(K)=\tau(K)+1$ and $\nu(\overline{K})=\tau(\overline{K})$, and again $\dim\HF(-S^3_f(\overline{K}))$ has a single minimum at $f=2\nu(\overline{K})-1$,

The same argument used in the case $\tau(K)\ge 0$ shows that in either case $F_{-W_f}$ is injective if and only if $f\ge2\nu(K)$.
\end{proof}

\subsection{Sutured Floer lemmas}

One of the two key ingredients in the proof of \ref{ass1} and \ref{ass2} is the associativity of maps in triple Heegaard diagrams: recall the following result of Ozsv\'ath and Szab\'o.

Suppose that we have a quadruple Heegaard diagram $(\Sigma,\ba,\bb,\bc,\bd,z)$, satisfing some additional admissibility assumption \cite{OSPA}: there are triangle count maps associate to the triple Heegaard diagram. Call them $f_{\alpha\beta\gamma}$, $f_{\alpha\beta\delta}$, $f_{\alpha\gamma\delta}$, $f_{\beta\gamma\delta}$, so that, for example, $f_{\alpha\beta\gamma}: \CF(\Sigma,\ba,\bb,z)\to \CF(\Sigma,\ba,\bc,z)$, and label with the capitalized letters $F$ the induced maps on the homology level.

\begin{prop}[\cite{OStriangles}]\label{associativity}
These maps satisfy the identity:
\[F_{\alpha\gamma\delta}(F_{\alpha\beta\gamma}(x\otimes y)\otimes v) =  F_{\alpha\beta\delta}(x\otimes F_{\beta\gamma\delta}(y\otimes v))\]
for all $x\in \HF(Y_{\alpha\beta})$, $y\in \HF(Y_{\beta\gamma})$ and $v\in \HF(Y_{\gamma\delta})$.
\end{prop}

The other key ingredient is given in Rasmussen's paper \cite{JakeHKM}: the philosophy is that gluing maps can be computed via triangle counts, given a handle decomposition of the gluing layer. In this thesis, we need three instances of this general fact: bypass attachments (Proposition \ref{Jakebypassprop} below), $\infty$- and +1-surgery maps (Proposition \ref{Jakesurgeryprop} below).

When we attach a bypass to a sutured manifold $(M,\Gamma)$ to obtain $(M,\Gamma')$, we change the sutures as in Figure \ref{bypass}: up to a 1-handle attachment (see below), we can assume that both $R_+$ and $R_+'$ are connected, so that both $(M,\Gamma)$ and $(M,\Gamma')$ are represented by an arc diagram. We can also suppose (see the rightmost picture in Figure \ref{bypass}) that the two arc diagrams live on the same Heegaard surface, that they share the $\alpha$-curves, all $\beta$-curves and all but one $\beta$-arc. Finally, we can assume that the two $\beta$-arcs where they differ intersect at exactly one point. Arguing as in the closed case, this determines a preferred $\Theta$-element in a triple arc diagram, which in turn allows us to define a triangle count: this triangle count is chain-homotopic to the bypass attachment map.

If we have a sutured manifold $(M,\Gamma)$ with torus boundary and $|\Gamma|=2$, we can attach a +1-surgery layer to get $(M',\{\gamma\})$ with sphere boundary. As before, we construct an arc diagram for $(M,\Gamma)$, and an arc diagram for $M'$ on the same Heegaard surface: all $\alpha$- and $\beta$-curves can be chosen to coincide, and the new $\beta$-curve can be chosen to intersect the $\beta$-arc exactly once. This determines a $\Theta$-element in a triple arc diagram, and the resulting triangle count induces $\psi_{+1}$ in homology.

\subsubsection{The proof of Proposition \ref{ass1}}

Recall that we want to prove that gluing maps associated to overtwisted contact structures vanish.

\begin{proof}[Proof of Proposition \ref{ass1}]
Suppose that there's an overtwisted disc $D\subset N$, and consider a small neighbourhood $B$ of it, with convex boundary. Then join $B$ to a boundary component of $N$ that is going to be glued to $M$, using a small neighbourhood $A$ of an arc. Call $N'$ the union of $A$, $B$ and a neighbourhood of the component of the boundary we've joined $B$ to, and suppose that the boundary of $N'$ is convex with respect to $\xi$. Call $N''$ the closure of the complement of $N'$ in $N$. Finally, let $\xi', \xi''$ be the restrictions of $\xi$ to $N'$ and $N''$ respectively.

\begin{claim}
We can suppose $N=N'$.
\end{claim}

\begin{proof}
By naturality of gluing maps, $\Psi_\xi = \Psi_{\xi''}\circ \Psi_{\xi'}$, and if $\Psi_{\xi'}=0$, then in particular $\Psi_\xi=0$.
\end{proof}

Following Ozbagci \cite{Ozbh} (see also Giroux's criterion for overtwistedness of contact structures near a convex surface), we can write the gluing of the overtwisted disc as a double bypass attachment, along a curve that makes a small dollar symbol \$ across a single suture, as in the top left of Figure \ref{zeromap}: unfortunately, there's a small technical detail we need to face: attaching the second bypass disconnects $R_+$. To overcome this obstacle, we first attach a contact 1-handle $H$ -- and this doesn't affect the sutured Floer homology groups, since it's the inverse of a product disc decomposition -- and then attach the two bypasses to the new manifold, as shown in the second left figure in \ref{zeromap}.

Suppose that we start off with an arc diagram $\H_{0} = (\Sigma_0, \ba_0, \bb^a_0, \bb^c_0, D_0)$ for $(M,\Gamma)$, as in the top left of Figure \ref{zeromap}. We obtain an arc diagram for $(M\cup H, \Gamma)$ by adding a 1-handle to $\Sigma_0$, obtaining a surface $\Sigma=\Sigma_0\# T^2$. The set of $\alpha$-curves is the same as before, plus a single $\alpha$-curve $\alpha_0$ that is the belt of the (3-dimensional) handle $H$. The set of $\beta$-curves is $\bb_0^c$, and we add a single $\beta$-arc $\beta_0$ that runs once through the handle, as in the top right corner of Figure \ref{zeromap}. Call $\H_\beta$ this new diagram.

Attaching the first bypass we obtain an arc diagram $\H_\gamma$. After attaching the second bypass in the same region, we obtain a third diagram $\H_\delta$, looking like the bottom right picture in Figure \ref{zeromap}. Call $\H_{\alpha\beta\gamma}$, $\H_{\alpha\beta\delta}$, $\H_{\alpha\gamma\delta}$ and $\H_{\beta\gamma\delta}$ the three triple Heegaard diagram we obtain.

It's straightforward to check that the admissibility conditions of \cite{OSPA} are satisfied by the arc diagram $(\Sigma,\ba,\bb,\bc,\bd,D)$.

As for the proof of Proposition \ref{EHgrading}, in order to obtain the bypass attachment maps we need to count triangles in the triple Heegaard diagrams $\H_{\alpha\gamma\beta}$, $\H_{\alpha\delta\gamma}$ and then take the associated \emph{cohomological} maps. More precisely, to the first bypass attachment on $\H_\beta$ we can associate a $\Theta$-element $\bfTheta_{\beta\gamma}$ constructed as follows: the point on the arc $\beta_0$ is the only intersection point of $\beta_0$ with the arc $\gamma_0$; every other $\gamma$-curve in $\H_\gamma$ is a small perturbation of a $\beta$-curve in $\H_\beta$, and therefore there's a preferred choice among the two intersection points as in \cite{OSHF}. We then have:

\begin{prop}[\cite{JakeHKM}]\label{Jakebypassprop}
The map induced in cohomology by the triangle count map $f_{\alpha\gamma\beta}(\cdot\otimes\bfTheta_{\beta\gamma})$ is the gluing map associated to the bypass attachment.
\end{prop}

Similarly, there's a $\Theta$-element $\bfTheta_{\gamma\delta}$ in $\H_{\alpha\delta\gamma}$, and the associated triangle count map $f_{\alpha\delta\gamma}(\cdot\otimes\bfTheta_{\gamma\delta})$ induces the gluing map associated to the second bypass attachment.

Since we're working over the field $\F_2$, studying the maps induced in cohomology is the same as studying the maps associated in homology, which is what we're going to do from now on.

\begin{figure}
\begin{center}
\includegraphics[scale=.53]{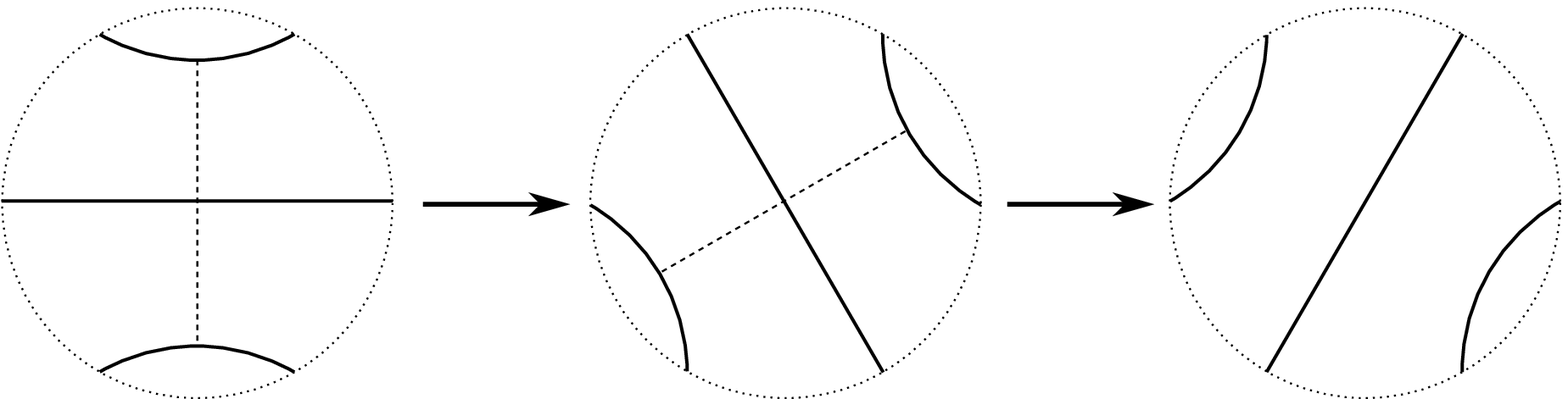}
\hskip 1 cm
\includegraphics[scale=.53]{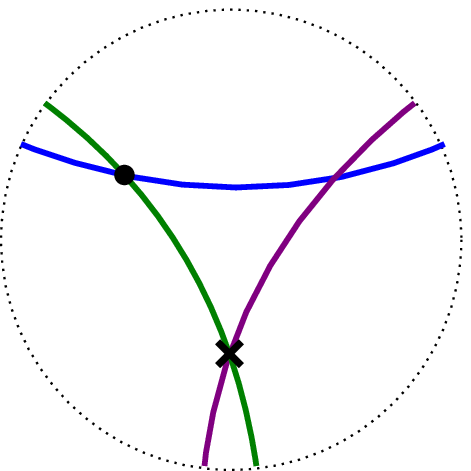}
\end{center}
\caption{The three circles on the left show the (local) effect of a double bypass attachment to the dividing curves of a convex boundary. The figure on the right shows what happens locally to the $\beta$-curves of the three arc diagrams coming from the figure on the left: the blue curve is a $\beta$-arc for the first diagram, the green curve is a $\gamma$-arc for the second diagram, and the purble one is a $\delta$-arc for the third diagram. The two intersection points in evidence are the points in $\bfTheta_{\beta\gamma}$ and $\bfTheta_{\gamma\delta}$ on the arcs shown.}\label{bypass}
\end{figure}

Call $(M',\Gamma')$ the sutured manifold defined by $\H_\delta$, so that, at the three-manifold level, $M' = M \cup H \cup N$, and let $\bfTheta_{\beta\delta}$ be the $\Theta$-element in the triple Heegaard diagram $\H_{\alpha\delta\beta}$. The following claim is a triangle count in $\H_{\delta\gamma\beta}$.

\begin{claim}\label{zerolemma}
$f_{\delta\gamma\beta}(\bfTheta_{\beta\gamma}\otimes\bfTheta_{\gamma\delta}) = \bfTheta_{\beta\delta}$.
\end{claim}

\begin{figure}
\begin{center}
\includegraphics[scale=.75]{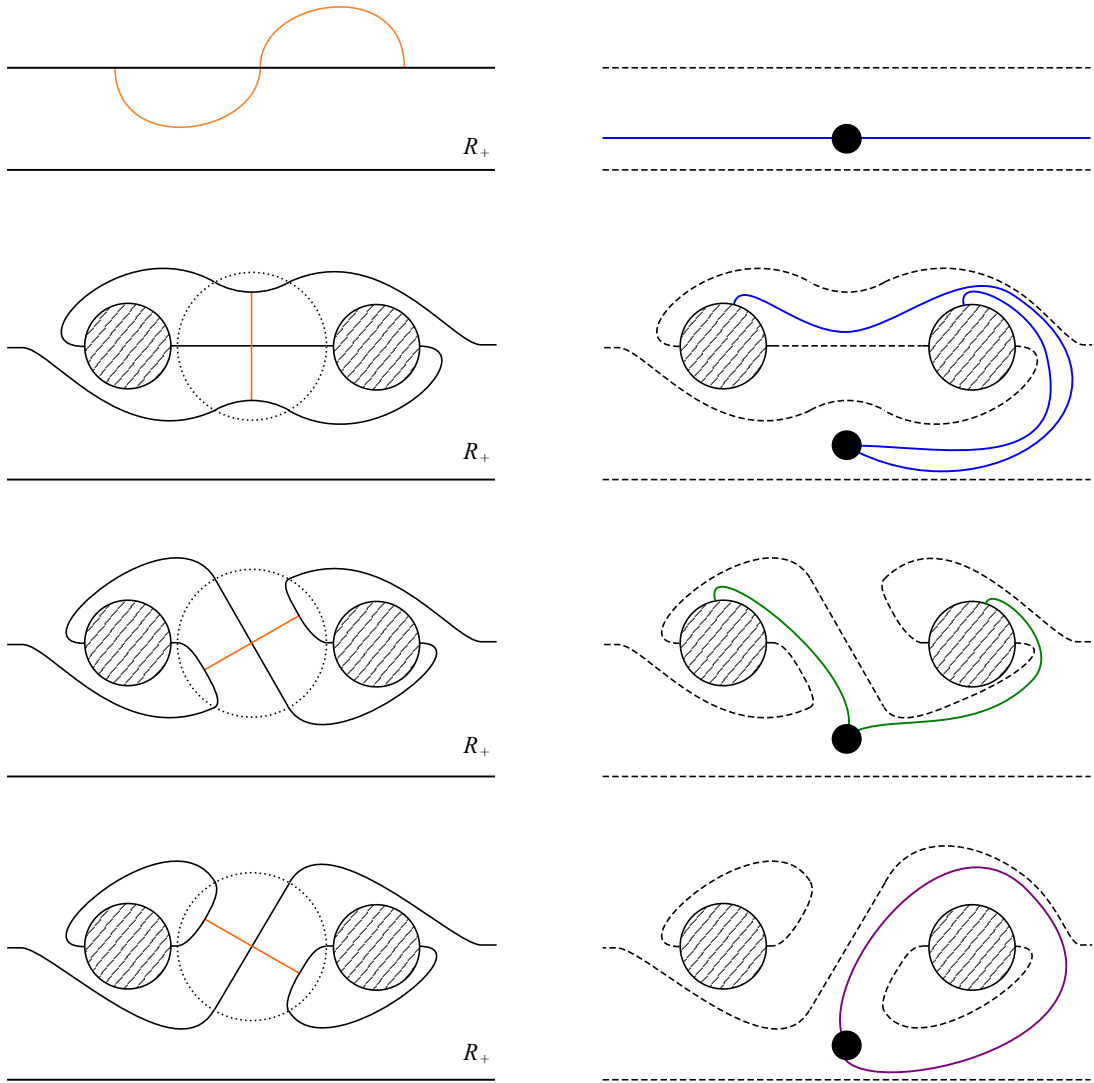}
\end{center}
\caption{On the left column we show $\de M$, sutures, bypasses and their effect on the sutures. On the right, we show associated arc diagrams.}\label{zeromap}
\end{figure}

\begin{figure}
\begin{center}
\includegraphics[scale=.85]{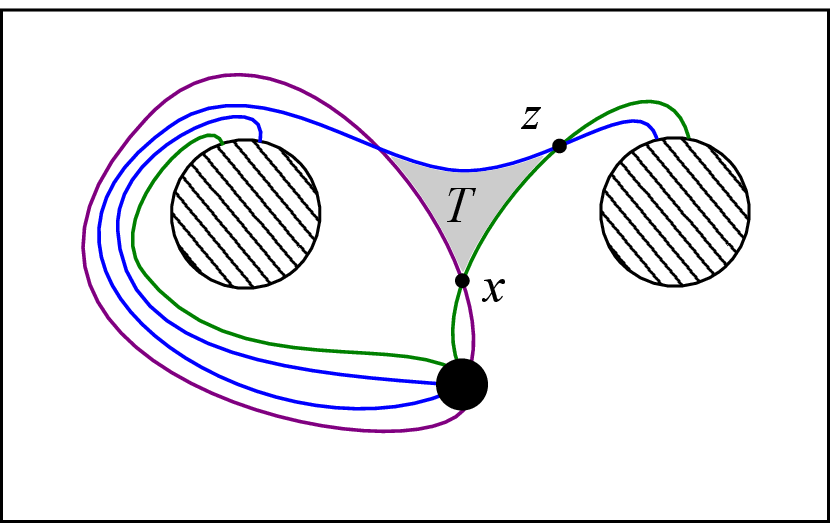}
\end{center}
\caption{The interesting portion of the triple arc diagram of Claim \ref{zerolemma}.}\label{bypass-arcs}
\end{figure}

\begin{proof}
We want to count all possible triangular domains $\D$ in $\H_{\delta\gamma\beta}$.

For each index $i$, $\beta_i$ intersects $\gamma_i$, $\delta_i$ and no other curve. Moreover, for $i>0$, both $\beta_i$ and $\delta_i$ are adjacent to a region touching the base disc $D$ on \emph{both} sides, so $\D$ can have positive multiplicities in this area only. In particular, $\D = \sum \D_i$, with $\D_i$ supported in the spanning region for all $i>0$, and $\D_0$ supported near $\beta_0$.

There is a domain $\overline{\D}\in\pi_2(\bfTheta_{\beta\gamma}, \bfTheta_{\gamma\delta}, \bfTheta_{\beta\delta})$ which is easy to spot: it is the sum of the small triangle $T$ in Figure \ref{bypass-arcs} and the small triangles shaded in figure \ref{fassotheta0}. It's well known (see \cite{OStriangles}) that this domain has Maslov index 0 and that the associated moduli space of triangles contains one element, thus providing us with a $\bfTheta_{\beta\delta}$ summand. We want to show that this is the only positive domain of Maslov index 0 in the triple Heegaard diagram.

Let's suppose that $\D=\sum \D_i$ as before, is a positive triangular domain, with multiplicity zero at every region touching the base disc.

In what follows, we'll call $z_i := (\Theta_{\beta\gamma})_i, x_i:=(\Theta_{\gamma\delta})_i, y_i:=(\Theta_{\beta\delta})_i$, and, when $i>0$, $y'_i$ the other intersection point of $\beta_i$ and $\delta_i$.

Let's first consider what happens in the region containing $\beta_0, \gamma_0$ and $\delta_0$: here all pairwise intersections are fixed, and are $x_0,y_0$ and $z_0$. The base disc $D$ lies on all three arcs, only one of the two segments into which the three intersection points divide the arcs can be part of $\de\D_0$. In particular, $\de\D_0$ has to coincide with $\de\overline{\D}$. Also, at every intersection, three of the four angles are contained in regions touching the base disc, therefore multiplicities have to be zero outside $T$, and in particular $\D_0 = T$.

\begin{figure}[h]
\begin{center}
\includegraphics[scale=.8]{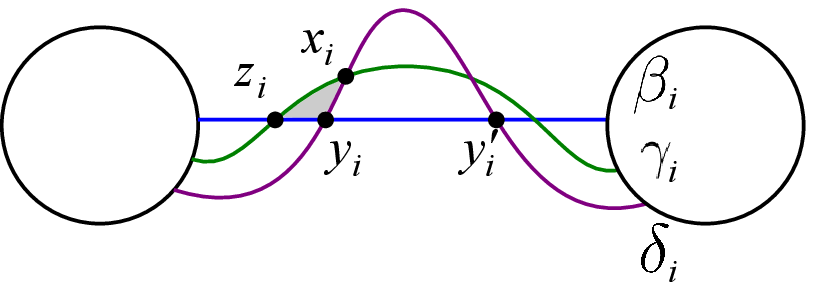}
\end{center}
\caption{The triple Heegaard diagram near $\beta_i$ for $i>0$.}\label{fassotheta0}
\end{figure}

Suppose $\D_i\in\pi_2(z_i, x_i, y_i)$, when $i>0$ (see Figure \ref{fassotheta0}): let's follow $\de\D$ from $z_i$ with the orientation given by $\D_i$. We have to stop at $x_i$ without winding multiple times, because there's a region that touches both sides of $\beta_i$ (and also of $\delta_i$) and the base disc $D$, so the multiplicity of $\D_i$ in that region has to be 0.

There are two possible segments: one is contained in the plane in Figure \ref{fassotheta0}, the other one runs inside the handle. In the first case, when we arrive at $x_i$ we have to turn \emph{left} (because of orientations) and we have to stop at $y_i$ without running around $\delta_i$ multiple times (because now $\delta_i$ touches a region containing the basepoint from both sides), and in particular $\D_i$ is the small triangle shaded in Figure \ref{fassotheta0}.

In the second case, the domain is an immersed triangle that has multiplicity two on the small triangle region shaded: using Sarkar's computation \cite{Sarkar}, we see that this domain gives a contribution to $\mu(\D)$ which is strictly bigger than 1/2.

Suppose now $\D_i\in\pi_2(z_i, x_i, y'_i)$: reasoning as above, we see that there are only two choices for $\D_i$, each obtained by adding one of the bigons in $\pi_2(y_i,y_i')$ to the small shaded triangle. Again, using Sarkar's computation, we see that these domains give a contribution bigger than 1/2 to $\mu(\D)$.

Summing up, if $\D\neq\overline{\D}$, $\mu(\D)$ is strictly bigger than $\mu(\overline{\D})=0$, and therefore $\D$ is not involved in the triangle count.
\end{proof}

Thanks to the claim and Proposition \ref{associativity}, we can consider the single triangle count $f_{\alpha\delta\beta}(\cdot\otimes\bfTheta_{\beta\delta})$. In order to achieve admissibility for $\H_{\alpha\delta\beta}$, we need to perturb the new $\alpha$-curve so that it intersects the new $\delta$-arc in a pair of canceling points as in Figure \ref{zeromap3}.

\begin{figure}
\begin{center}
\includegraphics[scale=.70]{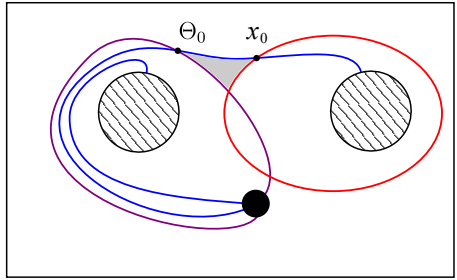}
\end{center}
\caption{The portion of $\H_{\alpha\delta\beta}$ considered in Claim \ref{nopositive}.}\label{zeromap3}
\end{figure}

\begin{claim}\label{nopositive}
There are no positive triangular domains in $\H_{\alpha\delta\beta}$ that appear in the triangle count for $f_{\alpha\delta\beta}$.
\end{claim}

\begin{proof}
Consider Figure \ref{zeromap3}: this is the same part of the diagram of Figure \ref{fassotheta0}, but we're now drawing the $\alpha$-curve instead of the $\gamma$-arc. We'll argue by contradiction: let $\D$ be such a domain.

The two points $\Theta_0$ and $x_0$ are the only two intersection points on the arc $\beta_0$: reasoning as in Claim \ref{zerolemma}, the boundary $\de\D\cap\beta_0$ is the segment between these two points, oriented from $\Theta_0$ to $x_0$. But the region above the segment in Figure \ref{zeromap3} touches the base disc $D$, therefore the multiplicity there has to be zero, showing that $\de\D\cap\beta_0 = \varnothing$.

In particular, $\pi_2(\cdot,\cdot,\bfTheta_{\beta\delta}) = \varnothing$.
\end{proof}

This immediately shows that $f_{\alpha\beta\delta}(\cdot\otimes\bfTheta_{\beta\delta}) = 0$, which in turn implies $\Psi_{\xi} = 0$.
\end{proof}

\subsubsection{The proof of Proposition \ref{ass2}}

Recall now that Proposition \ref{ass2} says that the diagram
\[
\xymatrix{
SFH(-S^3_{K,f})\ar[rr]^{\psi_{+1}}\ar[drr]_{\psi_\infty} & & SFH(-S^3_{f+1}(K)(1))\ar[r]^{\sim} &\HF(-S^3_{f+1}(K))\\
& & SFH(-S^3(1))\ar[r]^{\sim} & \HF(-S^3)\ar[u]^{F_{-W_{f+1}}}
}
\]
commutes. Notice that each of the three maps involved is computed by a triangle count in some triple Heegaard diagram.

\begin{rmk}
Ozv\'ath and Szab\'o \cite{OScontact} proved that the total cobordism map associated to a contact $+1$-surgery cobordism $W$ carries $c(\xi)$ to $c(\xi_{+1})$; more recently, Baldwin \cite{B} proved that there exists a $\spin$-structure $\mathfrak{t}_0$ on $W$ such that $F_{W,\mathfrak{t}_0}$ carries $c(\xi)$ to $c(\xi_{+1})$. In Proposition \ref{ass2} we don't worry about $\spin$-structures, and consider the total map only.
\end{rmk}

\begin{proof}[Proof of Proposition \ref{ass2}]
Let's fix any Heegaard diagram $\H'$ for $-(S^3,K)$, attach a 1-handle with feet next to the two basepoints $z,w$, and build the three Heegaard diagrams involved in the statement in the usual way: all curves except the ones intersecting the core $\alpha_0$ of the 1-handle are small Hamiltonian perturbations one of the other, and $\gamma_0$ and $\beta_0$ are parallel outside a small neigbhourhood of $\beta_0$. Moreover, any two among $\beta_0, \gamma_0$ and $\delta_0$ intersect in one single point. As for the proof of Proposition \ref{ass1}, we can associate a $\Theta$-element $\bfTheta_{\beta\gamma}$, $\bfTheta_{\gamma\delta}$, $\bfTheta_{\beta\delta}$ to each of the three triple diagrams.

It's also easy to check that all three triple diagrams are compatible \cite{OSPA}.

The map $F_{-W_f}$ is the map induced in cohomology by the triangle count $f_{\alpha\delta\gamma}(\cdot\otimes\bfTheta_{\gamma\delta})$ \cite{OStriangles}.

\begin{prop}[\cite{JakeHKM}]\label{Jakesurgeryprop}
The map $\psi_\infty$ is the map induced in cohomology by the triangle count $f_{\alpha\gamma\beta}(\cdot\otimes\bfTheta_{\beta\gamma})$.

The map $\psi_{+1}$ is the map induced in cohomology by the triangle count $f_{\alpha\delta\beta}(\cdot\otimes\bfTheta_{\beta\delta})$.
\end{prop}

Since we're working with $\F$ coefficients, proving the cohomological statement is equivalentt o proving the dual homological statement: Proposition \ref{ass2} can now be rephrased as:
\[F_{\alpha\delta\beta}(\cdot\otimes\bfTheta_{\beta\delta}) = F_{\alpha\gamma\beta}(F_{\alpha\delta\gamma}(\cdot\otimes\bf\Theta_{\gamma\delta})\otimes\bfTheta_{\beta\gamma}).\]

Let's call $\phi_{\gamma\beta}, \phi_{\delta\beta}, \phi_{\delta\gamma}$ the three triangle counts: namely, $\phi_{\gamma\beta} = f_{\alpha\gamma\beta}(\cdot\otimes\bfTheta_{\beta\gamma})$, and similarly for the other $\phi$-maps.

\begin{lemma}\label{assotheta}
$f_{\delta\gamma\beta}(\bfTheta_{\gamma\delta}\otimes\bfTheta_{\beta\gamma}) = \bfTheta_{\beta\delta}$.
\end{lemma}

\begin{proof}
Let's consider Figure \ref{fassotheta}: there are small triangles in the region spanned by $\beta_i$ during the Hamiltonian isotopy that brings $\beta_i$ to $\gamma_i$ and $\delta_i$, as shown in the top part of the figure; there's also a ``bigger'' triangle, shown in the bottom part of the figure, around the three curves $\beta_0,\gamma_0, \delta_0$ involved in the surgeries: as before, the domain $\overline{\D}$ obtained by summing these triangular regions gives the summand $\bfTheta_{\beta\delta}$.

We claim that there are no other positive domains of Maslov index 0 in the sum $f_{\beta\gamma\delta}(\bfTheta_{\beta\gamma}\otimes\bfTheta_{\gamma\delta})$. Suppose that $\D$ is one of these triangular domains.

\begin{figure}[h!]
\begin{center}
\includegraphics[scale=.75]{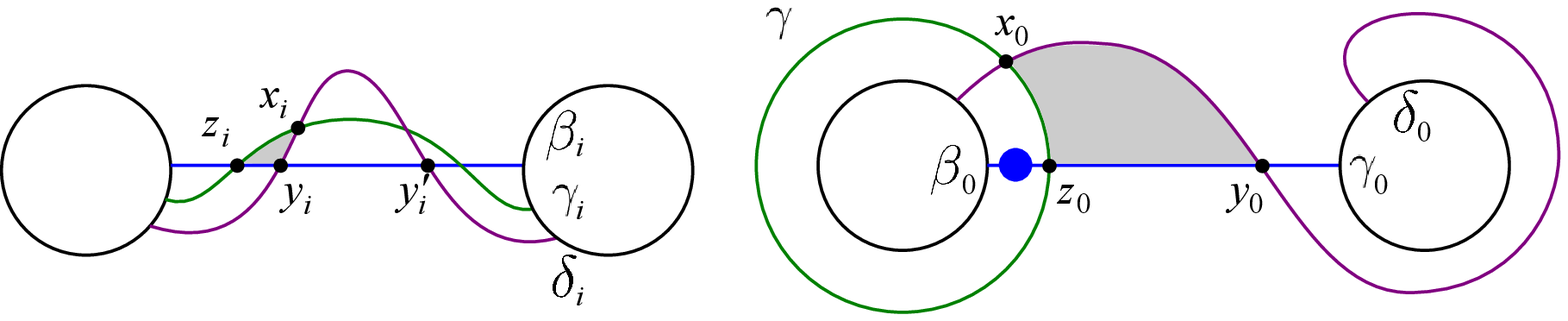}
\end{center}
\caption{The triple arc diagram $(\Sigma,\bd,\bc,\bb)$ of Lemma \ref{assotheta}.}\label{fassotheta}
\end{figure}

\begin{claim}
If $\D$ is as above, $\D = \overline{\D}$.
\end{claim}

As before, we'll call $z_i := (\Theta_{\beta\gamma})_i, x_i:=(\Theta_{\gamma\delta})_i, y_i:=(\Theta_{\beta\delta})_i$, and, when $i>0$, let $y'_i$ be the other intersection point of $\beta_i$ and $\delta_i$.

\begin{proof}
The situation is very similar to the situation in the proof of Lemma \ref{zerolemma}: for each index $i$, $\beta_i$ intersects $\gamma_i$, $\delta_i$ and no other curve. Moreover, for $i>0$, the boundary of every neighbourhood of the area spanned by $\beta_i$ under the isotopy lies in a region that touches the base disc $D$, so $\D$ can have positive multiplicities in this area only. In particular, $\D = \sum \D_i$, with $\D_i$ supported in the spanning region for all $i>0$, and $\D_0$ supported near $\beta_0$.

Let's consider what happens in the region containing $\beta_0, \gamma_0$ and $\delta_0$: here all pairwise intersections are fixed, and are $x_0,y_0$ and $z_0$. The base disc $D$ lies on $\beta_0$, so one of the two arcs into which $z_0$ and $y_0$ divide $\beta_0$ can't be part of $\de\D_0$. In particular, $\de\D_0\cap\beta_0$ has to coincide with $\de\overline{\D}\cap\beta_0$. Also, the big region (below this arc in the figure) touches the basepoint, so the multiplicity here has to be 0, and the multiplicity above it has to be 1 (we're crossing an arc in $\de\D_0$), and therefore $\D_0$ coincides with $\overline{\D}$ near $\beta_0$.

The situation around $\beta_i$ is exactly the same as in Lemma \ref{zerolemma}, and the same argument applies \emph{verbatim}, showing that $\D=\overline{\D}$.
\end{proof}

In particular, we have that the only summand in the triangle count is $\#\MS(\overline{\D})\cdot\bfTheta_{\beta\delta}$, concluding the proof of the lemma.
\end{proof}

Let's now get back to the proposition:
\[
F_{\alpha\delta\beta}(\cdot\otimes\bfTheta_{\delta\beta}) = F_{\alpha\delta\beta}(\cdot\otimes F_{\delta\gamma\beta}(\bfTheta_{\gamma\delta}\otimes\bfTheta_{\beta\gamma})) = F_{\alpha\delta\gamma}(F_{\alpha\gamma\delta}(\cdot\otimes\bf\Theta_{\gamma\delta})\otimes\bfTheta_{\gamma\delta}),
\]
which is exactly what we wanted to prove.
\end{proof}


\begin{thebibliography}{I}\setlength{\itemsep}{-0.5mm}

\small

\bibitem[Ba]{B}
J. Baldwin: \emph{Capping off open books and the Ozsv\'ath-Szab\'o contact invariant}, preprint, \url{http://arXiv.org/abs/09013797}.

\bibitem[BVV]{BVV}
J. Baldwin, D. Vela--Vick, V. V\'ertesi: \emph{On the equivalence of Legendrian and transverse invariants in knot Floer homology}, \url{http://arXiv.org/abs/1112.5970}.

\bibitem[CL]{knotinfo}
J. C. Cha, C. Livingston: \emph{KnotInfo: Table of Knot Invariants}, \url{http://www.indiana.edu/~knotinfo}.

\bibitem[CFHH]{CFHH}
T. Cochran, B. Franklin, M. Hedden, P. Horn: \emph{Knot concordance and homology cobordism}, preprint, \url{http://arXiv.org/abs/11025730}.

\bibitem[DG1]{DG}
F. Ding, H. Geiges: \emph{A Legendrian surgery presentation of contact 3-manifolds}, Math. Proc. Cambridge Philos. Soc. \textbf{136} (2004), 583--598.

\bibitem[DG2]{DGpd}
F. Ding, H. Geiges: \emph{A unique decomposition theorem for tight contact 3-manifolds}, Enseign. Math. (2) \textbf{53} (2007), no. 3-4, 333--345.

\bibitem[DG3]{DGls}
F. Ding, H. Geiges: \emph{Legendrian knots and links classified by classical invariants}, Commun. Contemp. Math. \textbf{9} (2007), no. 2, 135--162.

\bibitem[El1]{El}
Y. Eliashberg: \emph{Classification of overtwisted contact structures on 3-manifolds}, Invent. Math. \textbf{98} (1989), no. 3, 623--637.

\bibitem[El2]{El2}
Y. Eliashberg: \emph{Contact 3-manifolds twenty years since J. Martninet's work}, Ann. Inst. Fourier (Grenoble) \textbf{42} (1992), no. 1-2, 165--192.

\bibitem[Et]{EtnLTK}
J. Etnyre: \emph{Legendrian and transversal knots}, Handbook of knot theory, 105--185, Elsevier B. V., Amsterdam, 2005. 

\bibitem[EH1]{EH}
J. Etnyre, K. Honda: \emph{Knots and contact geometry I: Torus knots and the figure eight knot}, J. Symplectic Geom. \textbf{1} (2001), no. 1, 63--120.

\bibitem[EH2]{EHcables}
J. Etnyre, K. Honda: \emph{Cabling and transverse simplicity}, Ann. of Math. (2) \textbf{162} (2005), no. 3, 1305--1333.

\bibitem[ELT]{ELT}
J. Etnyre, D. LaFountain, B. Tosun: \emph{Legendrian and transverse cables of positive torus knots}, preprint, \url{http://arXiv.org/abs/1104.0550}.

\bibitem[Ga]{Ga}
D. Gabai: \emph{Foliations and the topology of 3-manifolds}, J. Differential Geom. \textbf{18} (1983), 445--503.

\bibitem[GV]{GVH-M}
P. Ghiggini, J. Van Horn-Morris: \emph{Tight contact structures on the Brieskorn spheres $-\Sigma(2,3,6n-1)$ and contact invariants}, preprint, \url{http://arXiv.org/abs/0910.2752}.

\bibitem[Go]{me}
M. Golla: \emph{Comparing invariants of Legendrian knots}, in preparation.

\bibitem[He]{He2}
M. Hedden: \emph{An Ozsv\'ath-Szab\'o Floer homology invariant of knots in a contact manifold}, Adv. Math. \textbf{219} (2008), no. 1, 89--117.

\bibitem[Hom]{Hom}
J. Hom: \emph{Heegaard Floer invariants and cabling}, PhD thesis, University of Pennsylvania (2011).

\bibitem[Ho]{Ho}
K. Honda: \emph{On the classification of tight contact structures. I}, Geom. Topol. \textbf{4} (2000), 309--368.

\bibitem[HKM1]{HKM1}
K. Honda, W. Kazez, G. Mati\'c: \emph{The contact invariant in sutured Floer homology}, Invent. Math. \textbf{176} (2009), no. 3, 637--676.

\bibitem[HKM2]{HKM2}
K. Honda, W. Kazez, G. Mati\'c: \emph{Contact structures, sutured Floer homology and TQFT}, preprint, \url{http://arXiv.org/abs/0807.2431}.

\bibitem[Ju]{Ju}
A. Juh\'asz: \emph{Holomorphic discs and sutured manifolds}, Algebr. Geom. Topol. \textbf{6} (2006), 1429--1457.

\bibitem[JT]{JOT}
A. Juh\'asz, D. Thurston: \emph{Naturality and mapping class groups in Heegaard Floer homology}, preprint, \url{http://arXiv.org/abs/1210.4996}.

\bibitem[KS]{KS}
K. Kodama, M. Sakuzi: \emph{Symmetry groups of prime knots up to 10 crossings}, Knots 90 (Osaka, 1990), 323--340, de Gruyter, Berlin, 1992.

\bibitem[Li]{Lip}
R. Lipshitz: \emph{A cylindrical reformulation of Heegaard Floer homology}, Geom. Topol. \textbf{10} (2006), 955--1097.

\bibitem[LOT]{LOT}
R. Lipshitz, P. Ozsv\'ath, D. Thurston: \emph{Bordered Floer homology: invariance and pairing}, preprint, \url{http://arXiv.org/abs/0810.0687}.

\bibitem[LOSS]{LOSS}
P. Lisca, P. Ozsv\'ath, A. Stipsicz, Z. Szab\'o: \emph{Heegaard Floer invariants of Legendrian knots in contact three-manifolds}, J. Eur. Math. Soc. \textbf{11} (2009), no. 6, 1307--1363.

\bibitem[LS1]{LS1}
P. Lisca, A. Stipsicz: \emph{Ozsv\'ath-Szab\'o invariants and tight contact three-manifolds I}, Geom. Topol. \textbf{8} (2004), 925--945.

\bibitem[LS2]{LS2}
P. Lisca, A. Stipsicz: \emph{Notes on the contact  Ozsv\'ath-Szab\'o invariants}, Pacific J. Math. \textbf{228} (2006), no. 2, 277--295.

\bibitem[LS3]{LStrans}
P. Lisca, A. Stipsicz: \emph{Contact surgery and transverse invariants}, Journal of Topology \textbf{4} (2011), no. 4, 817--834.

\bibitem[Oz1]{Ozbob}
B. Ozbagci: \emph{An open book decomposition compatible with rational contact surgery}, \emph{Proceedings of G\"okova Geometry-Topology Conference 2005}, 175--186, G\"okova Geometry/Topology Conference (GGT), G\"okova, (2006).

\bibitem[Oz2]{Ozbh}
B. Ozbagci: \emph{Contact handle decompositions}, Topol. Appl. \textbf{158} (2011), no. 5, 718--727.

\bibitem[OzsS]{OzvSt}
P. Ozsv\'ath, A. Stipsicz: \emph{Contact surgeries and the transverse invariant in knot Floer homology}, J. Inst. Math. Jussieu \textbf{9} (2010), no. 3, 601--632.

\bibitem[OSz1]{OSHF}
P. Ozsv\'ath, Z. Szab\'o: \emph{Holomorphic disks and topological invariants for closed three-manifolds}, Ann. of Math. (2) \textbf{159} (2004), no. 3, 1027--1158.

\bibitem[OSz2]{OSPA}
P. Ozsv\'ath, Z. Szab\'o: \emph{Holomorphic disks and three-manifold invariants: properties and applications}, Ann. of Math. (2) \textbf{159} (2004), no. 3, 1159--1245.

\bibitem[OSz3]{OStau}
P. Ozsv\'ath, Z. Szab\'o: \emph{Knot Floer homology and the four-ball genus}, Geom. Topol. \textbf{7} (2003), 615--639.

\bibitem[OSz4]{OScontact}
P. Ozsv\'ath, Z. Szab\'o: \emph{Heegaard Floer homologies and contact structures}, Duke Math. J. \textbf{129} (2005), no. 1, 39--61.

\bibitem[OSz5]{OStriangles}
P. Ozsv\'ath, Z. Szab\'o: \emph{Holomorphic triangles and invariants for smooth four-manifolds}, Adv. Math. \textbf{202} (2006), no. 2, 326--400.

\bibitem[OSz6]{OSint}
P. Ozsv\'ath, Z. Szab\'o: \emph{Knot Floer homology and integer surgeries}, Algebr. Geom. Topol. \textbf{8} (2008), no. 1, 101--153.

\bibitem[OSz7]{OSQsurg}
P. Ozsv\'ath, Z. Szab\'o: \emph{Knot Floer homology and rational surgeries}, Algebr. Geom. Topol. \textbf(11) (2011), no. 1, 1--68.

\bibitem[OSzT]{OST}
P. Ozsv\'ath, Z. Szab\'o, D. Thurston: \emph{Legendrian knots, transverse knots and combinatorial Floer homology}, Geom. Topol. \textbf{12} (2008), no. 2, 941--980.

\bibitem[NOT]{NOT}
L. Ng, P. Ozsv\'ath, D. Thurston: \emph{Transverse knots distinguished by knot Floer homology}, J. Symplectic Geom. \textbf{6} (2008), no. 4, 461--490.

\bibitem[Pl]{Pl}
O. Plamenevskaya: \emph{Bounds for the Thurston-Bennequin number from Floer homology}, Algebr. Geom. Topol. \textbf{4} (2004), 399--406.

\bibitem[Ra1]{Jake}
J. Rasmussen: \emph{Lens space surgeries and L-space homology spheres}, preprint, \url{http://arXiv.org/abs/0710.2531}.

\bibitem[Ra2]{JakeHKM}
J. Rasmussen: \emph{Triangle counts and gluing maps}, in preparation.

\bibitem[Ru]{Ru}
L. Rudolph: \emph{An obstruction to sliceness via contact geometry and ``classical'' gauge theory}, Invent. Math. \textbf{119} (1995), no. 1, 155--163.

\bibitem[Sah]{Sahamie}
B. Sahamie: \emph{Dehn twists in Heegaard Floer homology}, Algebr. Geom. Topol. \textbf{10} (2010), no. 1, 465--524.

\bibitem[Sar]{Sarkar}
S. Sarkar: \emph{Maslov index formulas for Whitney $n$-gons}, preprint, \url{http://arXiv.org/abs/0609.673}.

\bibitem[SV]{SV}
A. Stipsicz, V. V\'ertesi: \emph{On invariants for Legendrian knots}, Pacific J. Math. \textbf{239} (2009), no. 1, 157--177.

\bibitem[Za1]{Zarev}
R. Zarev: \emph{Bordered Floer homology for sutured manifolds}, preprint, \url{http://arXiv.org/abs/0908.1106}.

\bibitem[Za2]{Zarev2}
R. Zarev: \emph{Equivalence of gluing maps for SFH}, in preparation.

\end{thebibliography}
\end{document}